\documentclass[a4paper,reqno,11pt]{amsart}
\usepackage{mathtools}
\usepackage{dsfont}
\usepackage{amsmath, amsfonts, amssymb, amsthm, amscd}
\usepackage{graphicx}
\usepackage{psfrag}
\usepackage{perpage}
\usepackage{url}
\usepackage{color}
\usepackage{mathrsfs}
\usepackage{tikz}
\usepackage{mathabx}
\usetikzlibrary{arrows,decorations.pathmorphing,backgrounds,positioning,fit,petri} % for TIKZ
\usepackage{tikz,tikz-cd}\usepackage{caption,subcaption}

\usepackage[normalem]{ulem}

\usepackage{dsfont} % nice indicator function symbol (blackboard 1)
\usepackage{esint}

\usepackage[utf8]{inputenc}
\usepackage[T1]{fontenc}
\usepackage{microtype}

\usepackage[a4paper,scale={0.72,0.74},marginratio={1:1},footskip=7mm,headsep=10mm]{geometry}

\usepackage{hyperref}

\usepackage{amsmath}

\newcommand{\avgint}{%
	\mathop{%
		\vcenter{\hbox{%
				$\int$\kern-0.7em\raisebox{0.6ex}{\rule{0.8em}{0.1pt}}%
		}}%
	}%
}
\makeatletter
\def\@secnumfont{\bfseries\scshape}

\def\section{\@startsection{section}{1}%
  \z@{.7\linespacing\@plus\linespacing}{.5\linespacing}%
  {\normalfont\large\bfseries\scshape\centering}}

\def\subsection{\@startsection{subsection}{2}%
  \z@{.5\linespacing\@plus.7\linespacing}{-.5em}%
  {\normalfont\bfseries\scshape}}

\def\subsubsection{\@startsection{subsubsection}{3}%
  \z@{.5\linespacing\@plus.7\linespacing}{-.5em}%
  {\normalfont\scshape}}

\def\specialsection{\@startsection{section}{1}%
  \z@{\linespacing\@plus\linespacing}{.5\linespacing}%
  {\normalfont\centering\large\bfseries\scshape}}
\makeatother

\makeatletter

\renewenvironment{proof}[1][\proofname]{\par
\pushQED{\qed}%
\normalfont \topsep4\p@\@plus4\p@\relax
\trivlist
\item[\hskip\labelsep
\bfseries
#1\@addpunct{.}]\ignorespaces
}{%
\popQED\endtrivlist\@endpefalse
}
\makeatother

\makeatletter
\newcommand \Dotfill {\leavevmode \leaders \hb@xt@ 6pt{\hss .\hss }\hfill \kern \z@}
\makeatother

\makeatletter
\def\@tocline#1#2#3#4#5#6#7{\relax
  \ifnum #1>\c@tocdepth % then omit
  \else
    \par \addpenalty\@secpenalty\addvspace{#2}%
    \begingroup \hyphenpenalty\@M
    \@ifempty{#4}{%
      \@tempdima\csname r@tocindent\number#1\endcsname\relax
    }{%
      \@tempdima#4\relax
    }%
    \parindent\z@ \leftskip#3\relax \advance\leftskip\@tempdima\relax
    \rightskip\@pnumwidth plus4em \parfillskip-\@pnumwidth
    #5\leavevmode\hskip-\@tempdima
      \ifcase #1
       \or\or \hskip 1.65em \or \hskip 3.3em \else \hskip 4.95em \fi%
      #6\nobreak\relax
    \Dotfill
    \hbox to\@pnumwidth{\@tocpagenum{#7}}\par
    \nobreak
    \endgroup
  \fi}
\makeatother

\makeatletter
\def\l@section{\@tocline{1}{0pt}{1pc}{}{\scshape}}
\renewcommand{\tocsection}[3]{%
\indentlabel{\@ifnotempty{#2}{\ignorespaces#1 #2.\hskip 0.7em}}#3}
\def\l@subsection{\@tocline{2}{0pt}{1pc}{5pc}{}}

\def\l@subsubsection{\@tocline{3}{0pt}{1pc}{7pc}{}}

\makeatother

\numberwithin{equation}{section}

\newtheoremstyle{mytheorem}{.7\linespacing\@plus.3\linespacing}{.7\linespacing\@plus.3\linespacing}%
     {\itshape}%         Body font
     {}%         Indent amount (empty = no indent, \parindent = para indent)
     {\bfseries}% Thm head font (e.g. \bfseries, \scshape, \sffamily)
     {. }%        Punctuation after thm head
     {0.3ex}%     Space after thm head (\newline = linebreak)
     {\thmname{{\bfseries #1}}\thmnumber{ {\bfseries #2}}\thmnote{ (#3)}}  % Thm head spec

\theoremstyle{mytheorem}

\newtheorem{theorem}{Theorem}[section]
\newtheorem{lemma}[theorem]{Lemma}
\newtheorem{proposition}[theorem]{Proposition}

\newtheorem{remark}[theorem]{Remark}

\renewcommand{\tilde}{\widetilde}          % wider `tilde'
\DeclareMathSymbol{\leqslant}{\mathalpha}{AMSa}{"36} % nicer `smaller or equal'
\DeclareMathSymbol{\geqslant}{\mathalpha}{AMSa}{"3E} % nicer `larger or equal'
\DeclareMathSymbol{\eset}{\mathalpha}{AMSb}{"3F}     % nicer `emptyset'
\newcommand{\be}{\begin{equation}}
\newcommand{\ee}{\end{equation}}

\renewcommand{\epsilon}{\varepsilon}
\renewcommand{\theta}{\vartheta}
\renewcommand{\rho}{\varrho}

\newenvironment{myenumerate}{%
\renewcommand{\theenumi}{\arabic{enumi}}%
\renewcommand{\labelenumi}{{\rm(\theenumi)}}%
\begin{list}{\labelenumi}
	{%
	\setlength{\itemsep}{0.4em}%
	\setlength{\topsep}{0.5em}%
	\setlength\leftmargin{2.45em}%
	\setlength\labelwidth{2.05em}%
	\setlength{\labelsep}{0.4em}%
	\usecounter{enumi}%
	}%
	}%
{\end{list}
}

{\end{list}
}

{\end{list}
}

\renewenvironment{enumerate}{
\begin{myenumerate}}%
{\end{myenumerate}}

\newenvironment{myitemize}{%
\begin{list}{$\bullet$}%
 	{%
	\setlength{\itemsep}{0.4em}%
	\setlength{\topsep}{0.5em}%
	\setlength\leftmargin{2.65em}%
	\setlength\labelwidth{2.65em}%
	\setlength{\labelsep}{0.4em}%
	}%
	}%
{\end{list}}

{\end{myitemize}}

\MakePerPage[2]{footnote} % restarts footnote counter at every new page

\date{\today}

\newcommand{\td}{\text{d}}
\title[Extrema of cooling branching Brownian motion and related Gaussian fields]{Extrema of cooling branching Brownian motion and related Gaussian fields}

\author{Zuodi Xie}
\address{Department of Mathematics, University of Warwick, Coventry CV4 7AL, United Kingdom}
\email{zuodi-xie.xie@warwick.ac.uk}

\begin{document}

\begin{abstract}
We introduce a family of Gaussian fields with an inhomogeneous cooling variance profile, indexed by $\alpha\in(0,1/2]$, whose covariance is log--log-correlated at $\alpha=1/2$ and approaches the log-correlated regime as $\alpha\downarrow0$. We consider both one dimensional such objects (which we call {\it cooling Branching Brownian Motions}) as well as the related Gaussian fields. We identify the centering of the maximum at the terminal time $T$ and prove tightness of the recentered maximum. While the exponent in the first-order growth varies linearly with $\alpha$, giving a leading order of $T^{1-\alpha}$, the second-order correction exhibits a phase transition at $\alpha=1/3$. We also show that any subsequential limit cannot be a randomly shifted Gumbel law, in contrast with the logarithmically correlated case.
\end{abstract}


\maketitle

\tableofcontents

\section{Introduction}

\subsection{Background}
The study of extrema of log-correlated random fields has advanced rapidly in the past decade; see, for example, the notes \cite{Arguin2017_ExtremaLogCorrelated,Biskup2017_extremaNotes,Zeitouni2016_NoteBRWGaussianFields}. This direction is important for several reasons. First, real-valued branching Brownian motion (BBM) and its connection with the Fisher--KPP equation provide a fundamental setting in which extremal behavior plays a central role; see \cite{Bramson1983_FKPP}. Second, the high points of log-correlated Gaussian fields (for example, Gaussian free field) are closely tied to Gaussian multiplicative chaos (GMC), which is a fundamental and universal model in random geometry; see the reviews \cite{RV2014_GMCreview} and \cite{BerestyckiPowell2025GFFLQG}. For example, thick points and their Hausdorff dimension are intimately related to subcritical GMC; see \cite{HMP2010_thickpoints_GMC}. In addition, the law of the recentered maximum and the limiting extremal process are themselves closely connected to GMC-type objects. For example, in \cite{Madaule2015_max_log_Gaussian}, a critical GMC-type object, namely the derivative martingale, governs the random shift of the maximum, while in \cite{Biskup_Louidor2016_intermediate_level_sets}, the intermediate high points for discrete Gaussian free fields converge, after rescaling, to a GMC-like measure. In \cite{DRZ_2017_max_log_generalGaussian}, the similar randomly shifted Gumbel limit was proven for the recentered maximum of a family of discrete Gaussian fields with general covariance structure converging in some certain sense.

A closely related line of the present work comes from hierarchical spin-glass models with scale-dependent variance structure. Derrida's generalized random energy model introduced a Gaussian energy landscape whose correlations are organized by a hierarchy of scales \cite{Derrida1985GREM,DerridaGardner1986GREM}. The rigorous analysis of these models, including both finitely many hierarchical levels and the continuous hierarchy analogue, was developed by Bovier and Kurkova \cite{BovierKurkova2004GREM1,BovierKurkova2004GREM2}. These models can be regarded as natural predecessors of time-inhomogeneous branching models: the variance accumulated along a genealogical path depends on the scale, and the resulting extremal behaviour is governed by the geometry of the hierarchical covariance structure.

In the past two decades, substantial progress has also been made in the study of extrema of log-correlated models with time-inhomogeneous variance profiles. For BBM, most existing results concern the variance profile in the form $\sigma(t)=\widetilde{\sigma}(t/T)$ with $\widetilde{\sigma}$ bounded on $[0,1]$. Such models remain log-correlated and preserve many qualitative features of homogeneous BBM. In the homogeneous case, Bramson \cite{Bramson1978_max_BBM} proved convergence in law of the recentered maximum, and the extremal process was later identified in \cite{Arguin2013extremal,ABBS2013_BBMfromTop}. For macroscopic time-inhomogeneous BBM, analogous results have been obtained for two-speed profiles and for more general classes of variance functions; see, for example, \cite{BAH2014_extremal_process_twospeed_BBM,Bovier_and_Hartung2015_weakcor_BBM}. In the decreasing-variance case, a macroscopic slowdown with the characteristic Airy-type $T^{1/3}$ correction was discovered by Fang and Zeitouni \cite{FangZeitouni2012_slowdown_BBM} and further refined by Maillard and Zeitouni \cite{MZ2016_slowdownBBM}. In particular, \cite{MZ2016_slowdownBBM} developed Airy PDE estimates for probabilities associated with slowed-down Brownian motion, leading to the precise centering of the maximum and to a randomly shifted Gumbel limit.

\begingroup
\emergencystretch=2em
Parallel developments have been established for branching random walks (BRW) with time-inhomogeneous variances; see, for instance, \cite{FangZeitouni2012_twoscale_BRW,luo2025_extremal_twospeed_BRW,Ouitmet2018_BRWpwconstant,Mallein2015_slowdownBRW}. Natural higher-dimensional analogues of these BBM and BRW models arise in the form of log-correlated Gaussian fields, such as the scale-inhomogeneous Gaussian free field; see, for example, \cite{AO2016_extremes_scalevarying_GFF,Ouitmet2017_scale_GFF_Gibbs_Measure,Fels_and_Hartung_2020_scaleGFF_extremalprocess}. More broadly, time-inhomogeneous Gaussian models such as those in \cite{FangZeitouni2012_slowdown_BBM,MZ2016_slowdownBBM} can also provide useful intuition for more complicated non-Gaussian systems, see \cite{Bauerschmidt_Hofstetter2020_maxof_sineGordon,Barashkov2023_phi2max,Hofstetter2025_liouvillemaximum, CNZ2025_max_subcriticalDPRE}.
\par
\endgroup

The maxima of discrete log-correlated Gaussian fields were studied, for example, in \cite{BZ2012_tightness_maxdGFF,BDZ2016_weakconv_of_maxdGFF, DRZ_2017_max_log_generalGaussian}. Especially,
\cite{DRZ_2017_max_log_generalGaussian} developed a general framework for maxima of discrete log-correlated Gaussian fields under covariance assumptions. Their approach does not require the field itself to possess an exact branching structure or an independent decomposition over scales like BBM/BRW, although comparison with hierarchical Gaussian fields remains an important tool in the proof. Similar idea of Gaussian comparison was adapted in the setting of continuous Gaussian free field in \cite{Acosta2014_tightness}.

The present paper, in contrast to the previous literature, concerns a different regime. Rather than Gaussian models with variance profile changing but still in $O(1)$ scale, we study a \emph{cooling} setting, in which the variance profile decreases drastically as time proceeds and basically vanishes at the terminal time. This leads to a covariance scale substantially smaller than that in the usual log-correlated setting, and the resulting extremal behavior falls outside the standard macroscopic universality classes. Specifically, we introduce a class of Gaussian fields whose variance structure interpolates between log and $\log\log$ correlation. For these weak-covariance models, the leading orders of their maxima differ from those obtained for i.i.d. models, and the second order exhibits a phase transition in the $(\log)^{1/3}$-covariance case. What's more, we can also show that the weak limit (if exists) of the recentered maximum of the cooling BBM cannot be a randomly shifted Gumbel, in contrast with the classical logarithmically correlated case.

\subsection{Motivations from the two-dimensional Critical Stochastic Heat Flow.}
The present work is primarily motivated by the desire to study high peak values of the two-dimensional critical Stochastic Heat Flow (SHF), which was constructed in \cite{CSZ2023_2dcSHF}. See \cite{CSZ2024_reviewSHF, CSZ2025_SHFicm} for reviews. The Critical SHF is a positive random measure on $\mathbb{R}^2$ (hence not Gaussian), that exhibits log-correlations. Even though it is shown not to be an exponential of a Gaussian field (i.e. not a
GMC) \cite{CSZ2025_singularityregularitySHF}, it appears to share some GMC features. For example its moments on small balls blow up with the same magnitude as those of a GMC associated with a Gaussian field with $\log\log$ covariance structure \cite{LZ2024_micromomentSHF}. Furthermore, a central limit theorem has been found for the mollified SHF at one single point \cite{GuTsai_2026_loglogCLT}. Therefore, we hope that extremal behavior of the weakly correlated Gaussian models studied here, especially that of the $\log\log$-Gaussian fields, can provide useful insights into the extremal landscape of the SHF. Conversely, the results obtained here may help clarify the distinctions between the SHF and $\log\log$-GMCs.

We mention that a similar comparison has led to important progress in the study of extrema of the two-dimensional \textit{subcritical} directed polymer in random environment (DPRE). In the recent work \cite{CNZ2025_max_subcriticalDPRE}, the leading-order asymptotics of the maximum of the rescaled logarithmic partition function for \textit{subcritical} two-dimensional DPRE, viewed as a random field, were determined and the leading term in that asymptotic was shown to admit a form reminiscent to the leading order of the maximum of the BBM with decreasing variance profile (but still of $O(1)$ scale, not 'cooling') obtained in \cite{FangZeitouni2012_slowdown_BBM,MZ2016_slowdownBBM}.

\subsection{Main Results}
\subsubsection{Gaussian fields with covariance weaker than logarithm}
We first construct a family of Gaussian fields on $\mathbb{R}^d$, with $d\geq2$, whose covariance structure interpolates between the $\log\log$-correlated and log-correlated regimes. Specifically, let $W(\td u,\td t)$ be space-time white noise on
$\mathbb{R}^d\times(0,\infty)$. For $t>0$, set
\[
q_t(x)
:=
\frac{g_{\frac{1}{4t} }(x)}
{\bigl(\int_{\mathbb{R}^d}g_{\frac{1}{4t}}(y)^2\,\td y\bigr)^{1/2}}
=
\biggl(\frac{4t}{\pi}\biggr)^{d/4}
\exp\{-2t\lVert x\rVert_2^2\},
\qquad x\in\mathbb{R}^d,
\]
where $g_{1/(4t)}(x):=(2t/\pi)^{d/2}\exp\{-2t\lVert x\rVert_2^2\}$ is the heat kernel in $\mathbb{R}^d$, and
$\lVert x\rVert_2:=\bigl(\sum_{i=1}^d x_i^2\bigr)^{1/2}$ is the Euclidean norm on $\mathbb{R}^d$.
For any $\alpha\in(0,1/2]$, set
\begin{equation}\label{eqModel: alpha epsilon kernel, I_alpha}
	I_\alpha(t):=\int_0^{\log(1+t)} (1+s)^{-2\alpha}\,\td s,
	\qquad t\geq0,
\end{equation}
and define the field $X_\epsilon^\alpha$ by the stochastic integral
\begin{equation}\label{eqModel: X epsilon alpha white noise integral}
	X_\epsilon^\alpha(z)
	:=
	\int_0^{\epsilon^{-2}}
	\int_{\mathbb{R}^d}
	\sqrt{I_\alpha'(t)}\,
	q_t(z-u)\,
	W(\td u,\td t),
	\qquad z\in\mathbb{R}^d,
\end{equation}
with $W$ the $d+1$ white noise and
$I_\alpha'(t)$ is the derivative of $I_\alpha$.
By the Itô isometry for white-noise integrals, we have
\begin{equation}\label{eqModel: alpha epsilon kernel, as Laplacian}
	\begin{aligned}
		k_\alpha^\epsilon(z_1,z_2)&:= \operatorname{Cov}
		\bigl(X_\epsilon^\alpha(z_1),X_\epsilon^\alpha(z_2)\bigr)\\
		&=
		\int_0^{\epsilon^{-2}}
		I_\alpha'(t)
		\int_{\mathbb{R}^d}
		q_t(z_1-u)q_t(z_2-u)\,\td u\,\td t \\
		&=
		\int_0^{\epsilon^{-2}}
		\exp\{-t\lVert z_1-z_2\rVert_2^2\}
		I_\alpha'(t)\,\td t.
	\end{aligned}
\end{equation}
By Proposition~\ref{proposition: k alpha epsilon, covariance estimate}, the covariance structure of $X_\epsilon^\alpha(\cdot)$ has asymptotics, for all $z_1,\ z_2\in[0,1]^d$ with $\lVert z_1-z_2\rVert_2$ sufficiently small,
\[
k_\alpha^\epsilon(z_1,z_2)
\asymp
\begin{cases}
	\displaystyle
	\biggl(\log\frac{1}{\epsilon\vee\lVert z_1-z_2\rVert_2}\biggr)^{1-2\alpha},
	& \alpha\in(0,1/2),\\[6pt]
	\displaystyle
	\log\biggl(\log\frac{1}{\epsilon\vee\lVert z_1-z_2\rVert_2}\biggr),
	& \alpha=1/2.
\end{cases}
\]
where $f\asymp g$ means that there exist positive constants $c$ and $C$ such that $cg\leq f\leq Cg$. Thus as $\alpha\to 0$, the models approaches the $\log$-correlated region, while $\alpha=1/2$ gives the $\log\log$-correlated region. Before introducing our main results, we define the functions needed to specify the centerings for the maxima of our models. For any $\alpha\in(0,1/2]$, define
\begin{equation}\label{eqIntroduction: sigma(t)}
	\sigma_\alpha(t):=\sqrt{2}\,(1+2t)^{-\alpha},\qquad t\geq0,
\end{equation}
which appears in the covariance estimates \eqref{eqK:alpha epsilon kernel, covariance structure} for $\{X_\epsilon^\alpha(\cdot)\}$. Then define
\begin{equation}\label{eqIntroduction: eta(t)}
	\eta_\alpha(t)
	:=
	\alpha_1\,2^{-1/3}(\sqrt{2d})^{2/3}\,
	\lvert\sigma_\alpha'(t)\rvert^{2/3}\sigma_\alpha(t)^{1/3},\qquad t\geq0,
\end{equation}
where $-\alpha_1=-2.33811\ldots$ denotes the principal eigenvalue of the Airy operator $Lu=u''-xu$ on $[0,\infty)$; see \eqref{eqA: Airy eigenvalues} for details. A direct computation yields
\begin{equation}\label{eqIntroduction: eta explicit}
	\eta_\alpha(t)
	=
	c_\eta (1+2t)^{-\alpha-2/3},
	\qquad
	c_\eta
	:=
	\alpha_1\alpha^{2/3}(\sqrt{2d})^{2/3}2^{5/6}.
\end{equation}

We then set
\begin{equation}\label{eqIntroduction: A(t)}
	A_\alpha(t)
	:=
	\int_0^t
	\bigl(\sqrt{2d}\,\sigma_\alpha(s)-\frac{\eta_\alpha(s)}{\sqrt{2d}}
	\bigr)\,\td s,
	\qquad t\geq0.
\end{equation}
In our main results, the centering of the maximum of $\{X_\epsilon^\alpha(x)\colon x\in[0,1]^d\}$ is given by $A_\alpha(T)$, where $T:=\log(1/\epsilon)$. For large $T$, direct computation gives
\begin{equation}\label{eqIntroduction: Aalpha(T) expansion}
	A_\alpha(T)=c_1(\alpha,d)T^{1-\alpha}
	-
	\begin{cases}
		\displaystyle
		c_2(\alpha,d) T^{1/3-\alpha} + O(1),
		& \text{if }\alpha<1/3,\\
		c_2(\alpha,d) \log T + O(1),
		& \text{if }\alpha=1/3,\\
		O(1),
		& \text{if }1/3<\alpha\leq1/2,
	\end{cases}
\end{equation}
with
\begin{equation}\label{eqIntroduction: c1(alpha,d)}
	c_1(\alpha,d):= \frac{2^{1-\alpha}\sqrt{d}}{1-\alpha},
\end{equation}
and
\begin{equation}\label{eqIntroduction: c2(alpha,d)}
	c_2(\alpha,d):= \begin{cases}
		\displaystyle
		\frac{2^{1/3-\alpha}c_\eta}{2\sqrt{2d}\,(1/3-\alpha)} & \text{if }\alpha<\frac{1}{3}, \\[6pt]
		\displaystyle
		\frac{c_\eta}{2\sqrt{2d}} & \text{if }\alpha=\frac{1}{3}.
	\end{cases}
\end{equation}

Here is our main result for the maximum of $\{X_\epsilon^\alpha(x)\colon x\in[0,1]^d\}$.
\begin{theorem}\label{theorem: tightness of recentered maximum for X alpha T}
	Let $T:=\log(1/\epsilon)$ and define $\{X_\epsilon^\alpha(x)\colon x\in[0,1]^d\}$ by \eqref{eqModel: X epsilon alpha white noise integral}. For any $d\geq2$ and $\alpha\in(0,1/2]$, there exist constants $C>0$, $c>0$, $T_*>1$, and $\lambda_*>1$ such that for all $T>T_*$ and $\lambda>\lambda_*$,
	\[
	\mathbb{P}\biggl(
	\bigl\lvert\max_{x\in[0,1]^d}X_\epsilon^\alpha(x)-A_\alpha(T)\bigr\rvert
	\geq\lambda
	\biggr)
	\leq Ce^{-c\lambda}.
	\]
	where $A_\alpha(T)$ is defined in \eqref{eqIntroduction: A(t)}.
\end{theorem}

\begin{remark}
	From the expansion \eqref{eqIntroduction: Aalpha(T) expansion} of $A_\alpha(T)$, we find a continuous phase transition at $\alpha=1/3$, which is due to the integral expression \eqref{eqIntroduction: A(t)}. Specifically, the second-order term is of order $T^{1/3 -\alpha}$ for $\alpha\in(0,1/3)$, of order $\log T$ when $\alpha=1/3$, and is absorbed into the $O(1)$ term when $\alpha>1/3$. Note also that the coefficients $c_2(\alpha,d)$ in \eqref{eqIntroduction: c2(alpha,d)} blow up when $\alpha \uparrow \frac{1}{3}$, while $c_2(1/3,d)$ remains finite.
\end{remark}

\subsubsection{Cooling Branching Brownian Motion}
As an analogue of the Gaussian field $X_\epsilon^\alpha(\cdot)$ on $\mathbb{R}^d$, with $d\geq2$, we introduce our BBM models: At time $t=0$, there is a single particle whose position evolves according to
\[
Y_t^\alpha:=\int_0^t\sigma_\alpha(s)\,\td B_s
=\sqrt{2}\int_0^t(1+2s)^{-\alpha}\,\td B_s,\ t\geq0,
\]
where $\sigma_\alpha$ is the same as \eqref{eqIntroduction: sigma(t)}. The particle carries an independent exponential clock with parameter $1$. When the clock rings, the particle splits into two offspring, and each offspring subsequently evolves as an independent copy of the same process, with its own independent $\operatorname{Exp}(1)$ clock. This branching mechanism is iterated until the terminal time $T$.

To distinguish the present model from the slowed-down BBM models studied in \cite{FangZeitouni2012_slowdown_BBM} and \cite{MZ2016_slowdownBBM}, we refer to it as the \emph{cooling BBM}. The terminology is motivated by the behavior of the variance profile $\sigma_\alpha(\cdot)$: for $\alpha>0$, the profile decreases substantially as time evolves, and when $t$ is close to the terminal time $T$, it is only of order $T^{-\alpha}$. Thus, near the terminal time, the particles move with very small diffusivity and may be viewed as being gradually cooled, or almost frozen. See Figure~\ref{fig: BBM in a colling down env} for an illustration. By contrast, the variance profiles in the $O(1)$-slowed-down BBM models considered in \cite{FangZeitouni2012_slowdown_BBM} and \cite{MZ2016_slowdownBBM} are functions of the macroscopic time variable $t/T$ and remain of order one throughout the time interval. In those models, though the variance profile is decreasing, the particles are not gradually 'frozen' at the terminal scale and have variance profile still of order $O(1)$.

\begin{figure}[htbp]
	\centering
	\includegraphics[width=0.6\textwidth]{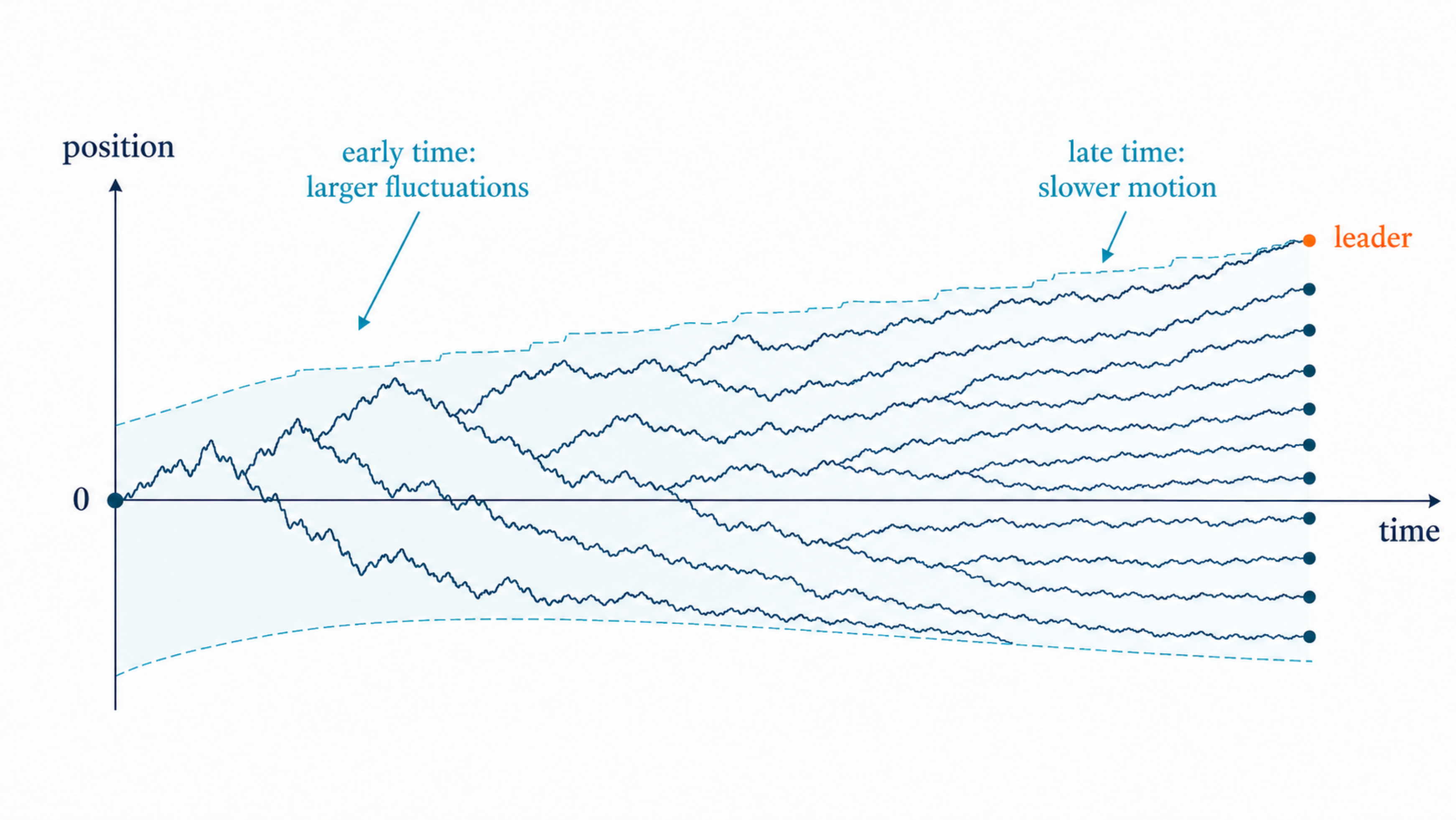}
	\caption{An illustration of cooling BBM}
	\label{fig: BBM in a colling down env}
\end{figure}

Let us introduce some notations for clarity. We write $\mathcal{N}_t$ for the set of particles alive at time $t\in[0,T]$, and $N_t:=\lvert\mathcal{N}_t\rvert$ for the number of such particles. For any $v\in\mathcal{N}_t$, we let $\{Y_s^{\alpha,v}\colon s\in[0,t]\}$ denote the trajectory of particle $v$ before time $t$. It is well known that the number of BBM particles satisfies
\begin{equation}\label{eqBBM: number of BBM particles}
	\begin{aligned}
		\mathbb{E}[N_t] &= e^t,\\
		\mathbb{P}\bigl(N_t=n\bigr) &= e^{-t}(1-e^{-t})^{n-1}, \qquad n\geq1,\\
		\mathbb{P}\bigl(N_t\leq e^{ct}\bigr) &\leq e^{-(1-c)t},
		\qquad \forall\, c<1.
	\end{aligned}
\end{equation}
See, for example, the proof of Lemma~2.2 in \cite{ChenGaoMallein2023_BBMsmallmax}. The following theorem shows that the rightmost positions of our cooling BBM exhibit the same centering formula $A_\alpha(\cdot)$, with dimension $d=1$.

\begin{theorem}\label{theorem: BBM tightness of recentered maximum}
	Fix any $\alpha\in(0,1/2]$. There exist constants $C>0$, $c>0$, $T_*>1$, and $\lambda_*>1$ such that for all $T>T_*$ and $\lambda>\lambda_*$,
	\begin{equation}\label{eqM: main result tightness}
		\mathbb{P}\biggl(
		\bigl\lvert\max_{v\in\mathcal{N}_T}Y_T^{\alpha,v}-A_\alpha(T)\bigr\rvert
		\geq\lambda
		\biggr)
		\leq Ce^{-c\lambda},
	\end{equation}
	where $A_\alpha(\cdot)$ is given in \eqref{eqIntroduction: A(t)} with $d=1$.
\end{theorem}

A further consequence of our estimates is that every subsequential weak limit of the recentered maxima $\max_{v\in\mathcal{N}_T}Y_T^{\alpha,v}-A_\alpha(T)$
cannot be the randomly shifted Gumbel law appearing as the limiting law of recentered maxima of $\log$-correlated Gaussian fields. More precisely, a strengthened upper tail estimate proved in Proposition~\ref{propo: stronger upper tial for BBM}
implies that every subsequential weak limit of the recentered maximum has bounded exponential moments for all exponents, while a randomly shifted Gumbel distribution has divergent exponential moments above a finite threshold. See Subsection~\ref{subsection: discussions on the weak limit} for more details.

To end this subsection, we compare the leading-order behavior of our cooling BBM with that of the corresponding i.i.d.\ model. Consider $e^T$ i.i.d.\ centered Gaussian random variables whose common variance matches that of a single BBM particle at time $T$, i.e.,
\[
\operatorname{Var}(Y_T)=\int_0^T \sigma_\alpha(s)^2\,\td s
=
\begin{cases}
	\dfrac{(1+2T)^{1-2\alpha}-1}{1-2\alpha}, & \alpha\in(0,1/2),\\[6pt]
	\log(1+2T), & \alpha=1/2.
\end{cases}
\]
By adapting the proof of Proposition 4.1.1 in \cite{Arguin2017_ExtremaLogCorrelated}, one obtains that the maximum of the i.i.d.\ model has leading order $\sqrt{2T\,\operatorname{Var}(Y_T)}$. In particular, for $\alpha\in(0,1/2)$, it is $2^{1-\alpha}(1-2\alpha)^{-1/2}T^{1-\alpha}$,
and for $\alpha=1/2$ it is of order $\sqrt{T\log T}$. From \eqref{eqIntroduction: Aalpha(T) expansion}, the leading order of the maximum of our BBM models is given by $\frac{2^{1-\alpha}}{1-\alpha}T^{1-\alpha}$. Thus, for $\alpha\in(0,1/2)$, the maximum of our BBM remains on the same scale as that of the i.i.d.\ model, but with a strictly smaller leading constant. By contrast, in the critical case $\alpha=1/2$, corresponding to the $\log\log$-correlated regime, the BBM maximum is only of order $T^{1/2}$, which is substantially smaller than $(T\log T)^{1/2}$ in the i.i.d.\ setting. A similar discussion also applies to the higher-dimensional models $X_\epsilon^\alpha$.

\subsection{Main ideas for the proofs}
\label{subsec: main ideas}

In this subsection, we explain the main ideas behind the analysis of the maxima of the models introduced above.

\subsubsection{Scale decomposition: From cooling BBM to $O(1)$-slowed-down BBM} We first give a heuristic explanation of why $A_\alpha(T)$ in \eqref{eqIntroduction: A(t)} is the natural centering for the maximum of our cooling BBM. We also explain the mechanism behind the phase transition in the second-order term in \eqref{eqIntroduction: Aalpha(T) expansion}. The key observation is a scale-similarity property of the model, which allows us to compare the dynamics on successive time blocks with the $O(1)$-slowed-down BBM studied in \cite{FangZeitouni2012_slowdown_BBM,MZ2016_slowdownBBM}.

Consider a BBM with variance profile
\begin{equation}\label{eqIntroduction: macro variance profile 1}
	\sigma(t)=\widetilde{\sigma}(t/\mathbf{T}),\qquad t\in[0,\mathbf{T}],
\end{equation}
where $\mathbf{T}$ is the terminal time and
$\widetilde{\sigma}\in C^2([0,1];\mathbb{R}_+)$ is nonincreasing. Let $M_{\mathbf{T}}$ denote the rightmost position of this BBM at time $\mathbf{T}$. By \cite{MZ2016_slowdownBBM},
\begin{equation}\label{eqIntro: macro-BBM centering}
	M_{\mathbf{T}}
	= c_1 \mathbf{T} - c_2 \mathbf{T}^{1/3} - c_3 \log \mathbf{T} + O_{\mathbf{T}}(1),
\end{equation}
where the constants $\{c_i\colon i=1,2,3\}$ are determined in \cite{MZ2016_slowdownBBM}, and $O_{\mathbf{T}}(1)$ is tight as $\mathbf{T}\to\infty$.

We now turn to our cooling BBM $\{Y_T^{\alpha,v}\colon v\in\mathcal{N}_T\}$ in Theorem~\ref{theorem: BBM tightness of recentered maximum} with some fixed $\alpha\in(0,1/2]$, whose variance profile is
\begin{equation}\label{eqIntroduction: micro BBM variance}
	\sigma_\alpha(t)=\sqrt{2}(1+2t)^{-\alpha},\qquad \alpha\in(0,1/2].
\end{equation}
To begin with, we decompose the trajectory of a single $Y^\alpha$-particle according to a geometrically increasing time mesh. Specifically, fix $H>0$ large, and assume for simplicity that the terminal time $T=(2^M H-1)/2$ for some large $M\in\mathbb{Z}_+$. Define $T_m:=(2^mH-1)/2$, $m=0,\ldots,M$, so that $T_M=T$. In the discussion here, we will omit the initial time interval $[0,T_0=(H-1)/2]$ whose contribution is only $O(H)$.

A $Y^\alpha$-particle evolves as
$Y_t^\alpha=\int_0^t\sigma_\alpha(s)\,\td B_s
=\int_0^t\sqrt{2}\,(1+2s)^{-\alpha}\,\td B_s$. For every $m\in\{0,\ldots,M-1\}$, we have
\begin{equation}\label{eqIntro: m-th block law}
	\{Y_{T_m+t}^\alpha-Y_{T_m}^\alpha\colon t\in[0,T_{m+1}-T_m]\}
	\overset{\mathrm d}{=}
	\{(2^mH)^{-\alpha}Z_t^{(m)}\colon t\in[0,2^{m-1}H]\},
\end{equation}
where $Z_t^{(m)}\overset{\mathrm d}{=}
\int_0^t\hat\sigma_m(s)\,\td B_s$ for $t\in[0,2^{m-1}H]$, with
\begin{equation}\label{eqIntroduction: hat sigma m}
	\hat\sigma_m(t)
	:= \sqrt{2}\bigl(1+\frac{t}{2^{m-1}H}\bigr)^{-\alpha},
	\qquad t\in[0,2^{m-1}H].
\end{equation}
The profile $\hat\sigma_m$ in \eqref{eqIntroduction: hat sigma m} is exactly of the form \eqref{eqIntroduction: macro variance profile 1}, with the terminal time $\mathbf{T}$ replaced by $2^{m-1}H$. Thus, recalling \eqref{eqIntro: m-th block law} and \eqref{eqIntro: macro-BBM centering}, one single particle at time $T_m$, after time $T_{m+1}-T_m$, will generate a BBM with largest displacement
\[
(2^mH)^{-\alpha}\Bigl(
c_1\,2^{m-1}H-c_2(2^{m-1}H)^{1/3}
-c_3\log(2^{m-1}H)
\Bigr)
\]
where $c_1$, $c_2$ and $c_3$ should be determined by the specific coefficient $\alpha$. Summing these contributions over $m=0,\ldots,M-1$, we obtain the leading order
\begin{equation}\label{eqIntro: sum up leading terms}
	\sum_{m=0}^{M-1}(2^mH)^{-\alpha}(2^{m-1}H)=O(T^{1-\alpha}),
\end{equation}
and the Airy correction is
\begin{equation}\label{eqIntro: sum up Airy terms}
	\sum_{m=0}^{M-1}(2^mH)^{1/3-\alpha}
	\asymp
	\begin{cases}
		T^{1/3-\alpha}, & \alpha<1/3,\\
		\log T, & \alpha=1/3,\\
		1, & \alpha>1/3
	\end{cases},
\end{equation}
which explains the phase transition at $\alpha=1/3$ in \eqref{eqIntroduction: Aalpha(T) expansion}; and the logarithmic correction is controlled by
\begin{equation}\label{eqIntro: sum of decalying log}
	L_T(\alpha):=\sum_{m=0}^{M-1}(2^mH)^{-\alpha}\log(2^{m-1}H)=O(1),
\end{equation}
due to our choice $\alpha>0$. However, we mention that even if the sum in \eqref{eqIntro: sum of decalying log} is not bounded, we cannot simply take it as the third term in the maximum. Here, we only take advantage of the large decaying coefficient $(2^mH)^{-\alpha}$ for $\alpha>0$ to illustrate the influence of the logarithm term can be bounded. In conclusion, we have heuristically explained the centering $A_\alpha(T)$ in \eqref{eqIntroduction: A(t)}, up to additional $O(1)$ terms. In the proof of the lower bound in Theorem~\ref{theorem: BBM tightness of recentered maximum}, specifically in Subsection~\ref{subsection: BBM Ltail}, we make this block-decomposition argument rigorous. The same scale-decomposition idea can also be used for more general micro-models.

To end the discussion here, let us consider the BBM with variance profile \eqref{eqIntroduction: micro BBM variance} in the regime $\alpha:=\alpha_T\to 0$. Following the same reasoning that leads to \eqref{eqIntro: sum up leading terms}--\eqref{eqIntro: sum of decalying log}, it is natural to conjecture that the leading and Airy-order terms are of size

\[
T^{1-\alpha_T},
\qquad
\alpha_T^{2/3}T^{1/3-\alpha_T}.
\]
However, apart from the first two terms above, the $L_T(\alpha_T)$ in \eqref{eqIntro: sum of decalying log} is no longer bounded, which hints that some non-trivial results should exhibit in the third term, which is left as a future research direction.

\subsubsection{Barrier estimates and Airy PDE} To prove the upper-tail bound for the recentered maximum of the cooling BBM, Proposition~\ref{proposition: BBM tight, upper bound}, and the lower-tail bound for the recentered maximum of the Gaussian fields $X_\epsilon^\alpha$, Proposition~\ref{Prop: max of X, lower tail}, we use a barrier argument, a standard tool in the study of maxima of branching processes. Roughly speaking, a barrier is a deterministic curve in time that tracks the typical location of extremal particles. It is chosen so that particles contributing to the maximum are expected to remain below the barrier until close to the terminal time. A central step in the analysis is therefore to estimate the probability that a particle trajectory crosses a prescribed barrier.

As in \cite{MZ2016_slowdownBBM}, a natural choice of barrier in our setting is
\[
b(t):=
\int_0^t
\bigl(\sqrt{2d}\,\sigma_\alpha(s)
-
\frac{\eta_\alpha(s)+r(s)}{\sqrt{2d}}
\bigr)
\td s,
\qquad t\geq0,
\]
where the main term $\sqrt{2d}\,\sigma_\alpha(s)
-
\frac{\eta_\alpha(s)}{\sqrt{2d}}$
is the infinitesimal contribution to the centering in \eqref{eqIntroduction: Aalpha(T) expansion}, and $r$ is a rectifying function introduced to control the barrier-crossing probability. For example, in the study of upper tail for cooling BBM in Subsection~\ref{subsection: micro-BBM, upper tail}, we take $r$ in \eqref{eqBBM: r} which is negative and thus slightly rises the barrier. The decreasing variance profile \eqref{eqIntroduction: sigma(t)} then leads, through Brownian crossing-probability estimates, to an Airy-type PDE; see Subsection~\ref{subsec: Appendix, Girsanov}.

There is, however, an important difference between the Airy PDE used in \cite{MZ2016_slowdownBBM} and the PDE used here. In \cite[(A.4)]{MZ2016_slowdownBBM}, the coefficients depend on the variable $t/T$, and therefore their derivatives are of order $O(1/T)$. By contrast, in our equation~\eqref{eqA: Airy PDE with sigma diffusion}, the coefficients are functions of the time variable $t$ itself, rather than of $t/T$. Their derivatives must therefore be estimated directly.

This distinction is important for the rectifying function $r$ appearing in the Airy PDE \eqref{eqAiry: PDE, u, sigma, r}. In the setting of \cite{MZ2016_slowdownBBM}, the corresponding rectifying function does not significantly affect the Airy estimate, because the derivatives of the coefficients are on the $T^{-1}$ scale. In our setting, the effect of $r$ is not automatically negligible. We must therefore verify explicitly which conditions on $r$ are sufficient for the Airy estimate to remain valid; see \eqref{eqAiry: conditions on r}. We follow the eigenfunction decomposition method used in Appendix A of \cite{MZ2016_slowdownBBM}, while treating the contribution of the rectifying function more carefully.

We also emphasize that, in our setting, the variance profile \eqref{eqIntroduction: sigma(t)} can be very small when $t$ is close to the terminal time $T$. As a result, the probability of crossing a barrier is highly sensitive to the terminal segment of the path; see, for example, Lemma~\ref{Appendix: Bridge like probability}. This sensitivity is also discussed in Remark~\ref{remark: lower bound sensitivity}.

\subsubsection{Gaussian comparison: From BBM to Gaussian fields} To study the maxima of the Gaussian fields $X_\epsilon^\alpha$ defined in \eqref{eqModel: X epsilon alpha white noise integral}, we use the Gaussian comparison strategy based on Slepian's lemma, as in \cite{Acosta2014_tightness}. Specifically, in Subsection~\ref{subsection: auxiliary Gaussian fields}, we construct a Gaussian field on the grid $V_\epsilon := [0,1)^d \cap \epsilon\mathbb{Z}^d$. This field captures the main contribution to the centering of the maximum. We then construct another Gaussian field, given on each box $v+[0,\epsilon)^d$, $v\in V_\epsilon$, which controls the local fluctuations of $X_\epsilon^\alpha$ inside each $\epsilon$-box.

This reduces the study of the maximum of $X_\epsilon^\alpha$ over $[0,1)^d$ to an extreme-value problem for a BBM decorated by an independent Gaussian field with smaller-order fluctuations. We can therefore use the estimates for cooling BBM obtained in Section~\ref{section: proofs for the BBM}.

\subsection{Discussion on the weak limit} \label{subsection: discussions on the weak limit}
In Theorem~\ref{theorem: BBM tightness of recentered maximum}, we have shown the tightness of the recentered maxima for our BBM models $\{Y_T^{\alpha,v}\colon v\in\mathcal{N}_T\}$. As a result, there exists subsequence say
\[
\bigl\{M_{T_n}^\alpha
:=\max_{v\in\mathcal{N}_{T_n}}Y_{T_n}^{\alpha,v}-A_{T_n}^\alpha
\colon n\geq1\bigr\}
\]
weakly converging to some random variable denoted by $M_\infty^\alpha$. We will illustrate here this limit cannot be a randomly shifted Gumbel as in the case of $\log$-correlated BBM (see \cite{MZ2016_slowdownBBM} for example). By Proposition~\ref{propo: stronger upper tial for BBM}, for any $q\geq1$, there exists large $T_q>0$, $x_q>0$, and some constant $C_q$ such that
\[
\sup_{T\geq T_q}
\mathbb{P}\biggl(
\max_{v\in\mathcal{N}_T}Y_T^{\alpha,v}-A_\alpha(T)>x
\biggr)
\leq C_qe^{-qx},
\qquad x\geq x_q.
\]
\begingroup
\emergencystretch=2em
Passing to the subsequential weak limit, we obtain the same type of upper-tail estimate for $M_\infty^\alpha$. Thus we have the finite exponential moments $\mathbb{E}\bigl[e^{pM_\infty^\alpha}\bigr]<\infty$ for any $p\geq1$. However, for a randomly shifted Gumbel random variable $G$ with distribution
\par
\endgroup
\[
\mathbb{P}\bigl(G<x\bigr)
=\mathbb{E}^{(Z)}\bigl[\exp\{-Ze^{-C_Gx}\}\bigr],
\qquad x\in\mathbb{R}
\]
for some positive random variable $Z$, where $\mathbb{E}^{(Z)}$ means the expectation under the law of $Z$, we have $\mathbb{E}\bigl[e^{pG}\bigr]=\infty$ when $p$ is large. Indeed, if we take some $a<b$ such that $p_{a,b}:= \mathbb{P}\bigl(Z\in[a,b]\bigr)>0$,
then
\[
\begin{aligned}
	\mathbb{P}\bigl(G>x\bigr)=1-\mathbb{E}^{(Z)}\bigl[\exp\{-Ze^{-C_Gx}\}\bigr] \geq p_{a,b}\bigl(1-\exp\{-ae^{-C_Gx}\}\bigr)
	\geq \frac{ap_{a,b}}{2}e^{-C_Gx}.
\end{aligned}
\]
for any $x$ large. As a result, $\mathbb{E}\bigl[e^{pG}\bigr]=\infty$ for $p>C_G$. It would be interesting to investigate the weak limits of the recentered maxima for the cooling BBMs introduced here, as well as for more general non-$\log$-correlated models. In particular, one may ask how the limiting law changes as the decay parameter $\alpha\in(0,1/2]$ in the variance profile \eqref{eqIntroduction: sigma(t)} is varied.

\subsection{Outline of the paper}

In Section~\ref{section: proofs for the BBM}, we prove Theorem~\ref{theorem: BBM tightness of recentered maximum}, establishing the tightness of the recentered maximum for cooling BBM. In Section~\ref{section: proof for the main results}, we prove Theorem~\ref{theorem: tightness of recentered maximum for X alpha T}, the corresponding tightness result for higher-dimensional Gaussian fields with weak correlations. Some auxiliary results are included in the appendices.

\section{Proof for theorem \ref{theorem: BBM tightness of recentered maximum}} \label{section: proofs for the BBM}
The goal of this section is to prove Theorem \ref{theorem: BBM tightness of recentered maximum}. We first recall the basic parameters of the branching Brownian motion considered there: each particle branches at rate $1$ and is replaced by two offspring. The terminal time is $T>0$, which will always be assumed large. For any fixed $\alpha\in(0,1/2]$, the motion of each particle is given by
\begin{equation}\label{eqBBM: one particle movement}
	Y_t^\alpha := \int_0^t \sigma_\alpha(s)\,\td B_s, \qquad t\in[0,T],
\end{equation}
where $B$ is a standard Brownian motion and
\begin{equation}\label{eqBBM: sigma}
	\sigma_\alpha(t):=\sqrt{2}\,(1+2t)^{-\alpha}.
\end{equation}

We also recall some notation: for $t\in[0,T]$, we denote the set of particles alive at time $t$ by $\mathcal{N}_t$, and its cardinality by $N_t:=\lvert\mathcal{N}_t\rvert$. We shall repeatedly use \eqref{eqBBM: number of BBM particles} for estimates on $N_t$. For any fixed $s\in[0,t]$ and $v\in\mathcal{N}_s$, we write $\mathcal{N}_t^{s,v}$ for the set of descendants of $v$ that are alive at time $t$. For each particle at the terminal time, say $v\in\mathcal{N}_T$, we denote by $\{Y_t^{\alpha,v}\colon t\in[0,T]\}$ its ancestral trajectory up to time $T$.

\subsection{Upper tail}\label{subsection: micro-BBM, upper tail}

In this subsection we prove the upper tail for Theorem \ref{theorem: BBM tightness of recentered maximum}:
\begin{proposition}\label{proposition: BBM tight, upper bound}
	Fix any $\alpha\in(0,1/2]$. There exists $c>0$ and $\lambda_*>1$, $T_*>1$ large such that
	\[
	\mathbb{P}\Bigl(\max_{v\in\mathcal{N}_T}Y_T^{\alpha,v}\ge A_\alpha(T)+\lambda\Bigr)
	\le \exp\{-c\lambda\},
	\qquad \lambda\ge\lambda_*,\quad T\ge T_*,
	\]
	where $A_\alpha(T)$ is defined in \eqref{eqIntroduction: A(t)} with $d:=1$.
\end{proposition}

Choose a fixed $T_0$ large enough so that $T_0\ge s_\star$ where $s_\star$ is the time threshold in Proposition \ref{propoAiry: Airy estimate}. Also, we let $T_*$ large enough so that $T-T_0\ge L_*(1+2T_0)^{2/3}$ where $L_\star$ also appears in Proposition \ref{propoAiry: Airy estimate}. We will decompose the total time interval $[0,T]$ into $[0,T_0]$ and $[T_0,T]$. The reason for such a decomposition is that, we can ensure the the displacement of each particle after $T_0$ has the law falling into the setting of Proposition \ref{propoAiry: Airy estimate}. Indeed,
\begin{equation}\label{eqBBM: upper bound, law of Yt-YT0}
	Y_{t+T_0}^\alpha-Y_{T_0}^\alpha
	\overset{\mathrm d}{=}
	\int_{T_0}^{t+T_0}\sigma_\alpha(s)\,\td B_s
	=
	\int_0^t\bar{\sigma}(s)\,\td B_s,
	\qquad t\in[0,T-T_0],
\end{equation}
with
\begin{equation}\label{eqBBM: bar sigma}
	\bar{\sigma}(t):=\sigma_\alpha(t+T_0)
	=\sqrt{2}\,(1+2T_0+2t)^{-\alpha},
	\qquad t\in[0,T-T_0].
\end{equation}

In the proof, we will see that the displacement accumulated during the long time interval $[T_0,T]$ is the main contribution to the centering of the maximum. By the branching property, conditional on the configuration at time $T_0$, the descendant systems branching from different particles in $\mathcal{N}_{T_0}$ evolve independently and each has the law of the auxiliary BBM
\begin{equation}\label{eqBBM: auxiliary BBM in [T0,T]}
	\bigl\{\bar{Y}_{T_0,t}^{\alpha,w}\colon
	t\in[0,T-T_0],\ w\in\mathcal{M}_t\bigr\},
\end{equation}
where $\mathcal{M}_t$ is the set of all particles at time $t\in[0,T-T_0]$, and the variance profile is $\bar{\sigma}(\cdot)$ given in \eqref{eqBBM: bar sigma}.

We now derive a useful estimate for the maximal displacement accumulated over the interval $[T_0,T]$. By the branching property and a union bound,
\begin{equation}\label{eqBBM: bar Y for max displacement in [T0,T] 1}
	\begin{aligned}
		&\mathbb{P}\Bigl(
		\max_{v\in\mathcal{N}_{T_0}}
		\max_{w\in\mathcal{N}_T^{T_0,v}}
		\bigl(Y_T^{\alpha,w}-Y_{T_0}^{\alpha,v}\bigr)\ge L
		\Bigr)\le
		\mathbb{E}\Bigl[N_{T_0}
		\mathbb{P}\Bigl(
		\max_{w\in\mathcal{M}_{T-T_0}}
		\bar{Y}_{T_0,T-T_0}^{\alpha,w}\ge L
		\Bigr)\Bigr]\\
		&=
		\mathbb{E}[N_{T_0}]
		\mathbb{P}\Bigl(
		\max_{w\in\mathcal{M}_{T-T_0}}
		\bar{Y}_{T_0,T-T_0}^{\alpha,w}\ge L
		\Bigr)=
		\exp\{T_0\}
		\mathbb{P}\Bigl(
		\max_{w\in\mathcal{M}_{T-T_0}}
		\bar{Y}_{T_0,T-T_0}^{\alpha,w}\ge L
		\Bigr),
	\end{aligned}
\end{equation}
where in the first equality we use the fact that particles at time $T_0$, i.e. $\mathcal{N}_{T_0}$, generate $N_{T_0}$ number of independent BBMs with the same law as  $\{\bar{Y}_{T_0,T-T_0}^{\alpha,w}\}$, and in the last equality we use $\mathbb{E}[N_{T_0}]=\exp\{T_0\}$.

As another preparation, we set a suitable moving barrier for the auxiliary BBM introduced in \eqref{eqBBM: auxiliary BBM in [T0,T]}. Roughly speaking, a barrier is a positive function of time such that the probability that the whole BBM system crosses it is small. Due to the slowdown property of our BBM model, it is natural to take the integral of the variance profile minus the Airy correction, as in \cite{MZ2016_slowdownBBM}. Here this is the function $A_\alpha(\cdot)$ defined in \eqref{eqIntroduction: A(t)}. However, to control the exceeding probability, we use a small rectification which raises the barrier by a bounded amount only. This is achieved by the rectifying function $\bar{r}(\cdot)$ defined in \eqref{eqBBM: r}.

Specifically, for any $t\in[0,T-T_0]$, we set $\bar{\eta}(t):=\eta(t+T_0)$, with $\eta(\cdot)$ defined in \eqref{eqIntroduction: eta explicit} with $d=1$. Set the auxiliary function
\begin{equation}\label{eqBBM: r}
	\bar{r}(t):=-(1+2T_0+2t)^{-1-\beta},
	\qquad t\in[0,T-T_0],
\end{equation}
for a fixed $\beta\in(0,\alpha)$. It is direct to check that this choice satisfies the conditions \eqref{eqAiry: conditions on r} for Proposition \ref{propoAiry: Airy estimate}. Define
\begin{equation}\label{eqBBM: A_r(t)}
	\bar{A}_r(t):=\int_0^t\bar{k}_r(s)\,\td s,
	\qquad t\in[0,T-T_0],
\end{equation}
with $\bar{k}_r(s):=\sqrt{2}\,\bar{\sigma}(s)-\bigl(\bar{\eta}(s)+\bar{r}(s)\bigr)/\sqrt{2}$ for $s\in[0,T-T_0]$.
Note that, when $T_0$ is large enough,
\begin{equation}\label{eqBBM: bar kr(0)/bar sigma^2}
	\begin{aligned}
		\frac{\bar{k}_r(0)}{\bar{\sigma}(0)^2}
	=\frac{1}{\bar{\sigma}(0)^2}
		\Bigl(\sqrt{2}\,\bar{\sigma}(0)
		-\frac{\bar{\eta}(0)+\bar{r}(0)}{\sqrt{2}}\Bigr)
		\ge c(1+2T_0)^\alpha>0,
	\end{aligned}
\end{equation}
for some constant $c>0$.

Recall the definition \eqref{eqIntroduction: A(t)} of $A_\alpha(\cdot)$. We have
\begin{equation}\label{eqBBM: upper bound, replace A(T) by}
	\begin{aligned}
		\lvert A_\alpha(T)-A_\alpha(T_0)-\bar{A}_r(T-T_0)\rvert
		&=\Bigl\lvert
		\int_{T_0}^T k_\alpha(s)\,\td s
		-\int_0^{T-T_0}\bar{k}_r(s)\,\td s
		\Bigr\rvert\\
		&=\frac{1}{\sqrt{2}}
		\Bigl\lvert\int_0^{T-T_0}\bar{r}(s)\,\td s\Bigr\rvert
		\le C,
	\end{aligned}
\end{equation}
since $\bar{r}(\cdot)$ defined in \eqref{eqBBM: r} is integrable.

%%%%%%%%%%%%
With all the preparations above, we give the proof of Proposition \ref{proposition: BBM tight, upper bound}. The argument follows the proof of Lemma 2.1 in \cite{MZ2016_slowdownBBM}, with the Airy estimate replaced by Proposition \ref{propoAiry: Airy estimate}.
\begin{proof}
	Fix any $\alpha\in(0,1/2]$. By \eqref{eqBBM: upper bound, replace A(T) by}, it is enough to prove that there exist $T_*>1$, $\lambda_*>1$ and $c>0$ such that
	\begin{equation}\label{eqBBM: upper bound, task with m replaced by A(T)}
		\mathbb{P}\Bigl(
		\max_{v\in\mathcal{N}_T}Y_T^{\alpha,v}
		\ge A_\alpha(T_0)+\bar{A}_r(T-T_0)+\lambda
		\Bigr)
		\le \exp\{-c\lambda\},
		\qquad \lambda\ge\lambda_*,\quad T\ge T_*.
	\end{equation}
	
	By the branching property at time $T_0$ and by \eqref{eqBBM: bar Y for max displacement in [T0,T] 1}, we have
	\begin{equation}\label{eqBBM: upper bound, two segments argument}
		\begin{aligned}
			&\mathbb{P}\Bigl(
			\max_{v\in\mathcal{N}_T}Y_T^{\alpha,v}
			\ge A_\alpha(T_0)+\bar{A}_r(T-T_0)+\lambda
			\Bigr)\\
			&\quad\le
			\mathbb{P}\Bigl(
			\exists v\in\mathcal{N}_{T_0}\colon
			Y_{T_0}^{\alpha,v}\ge A_\alpha(T_0)+\frac{\lambda}{2}
			\Bigr) +
			\exp\{T_0\}
			\mathbb{P}\Bigl(
			\max_{w\in\mathcal{M}_{T-T_0}}
			\bar{Y}_{T_0,T-T_0}^{\alpha,w}
			\ge \bar{A}_r(T-T_0)+\frac{\lambda}{2}
			\Bigr).
		\end{aligned}
	\end{equation}
	The first probability is estimated by a union bound and a Gaussian tail: Set $\mathcal{G}\bigl(0,\int_0^{T_0}\sigma_\alpha^2(s)\,\td s\bigr)$ as the Gaussian random variable with mean $0$ and variance $\int_0^{T_0}\sigma_\alpha^2(s)\,\td s$. We have
	\[
	\begin{aligned}
		&\mathbb{P}\Bigl(
		\exists v\in\mathcal{N}_{T_0}\colon
		Y_{T_0}^{\alpha,v}\ge A_\alpha(T_0)+\frac{\lambda}{2}
		\Bigr)\\
		&\quad\le
		\exp\{T_0\}
		\mathbb{P}\Bigl(
		\mathcal{G}\Bigl(0,\int_0^{T_0}\sigma_\alpha^2(s)\,\td s\Bigr)
		\ge A_\alpha(T_0)+\frac{\lambda}{2}
		\Bigr)
		\le \exp\{-c_1(T_0)\lambda^2\},
	\end{aligned}
	\]
	for all $\lambda>4\sqrt{T_0\int_0^{T_0}\sigma_\alpha^2(s)\,\td s}$, where
	\begin{equation}\label{eqBBM: c_1(T_0)}
		c_1(T_0):=\frac{1}{16\int_0^{T_0}\sigma_\alpha^2(s)\,\td s}.
	\end{equation}
	
	Hence it remains to show that
	\begin{equation}\label{eqBBM: upper bound, essential task 1}
		\mathbb{P}\Bigl(
		\max_{w\in\mathcal{M}_{T-T_0}}
		\bar{Y}_{T_0,T-T_0}^{\alpha,w}
		\ge \bar{A}_r(T-T_0)+\frac{\lambda}{2}
		\Bigr)
		\le C\exp\{-c\lambda\}.
	\end{equation}
	To simplify notation, write $\bar{Y}_t:=\bar{Y}_{T_0,t}^\alpha$ for $t\in[0,T-T_0]$. Since the event in \eqref{eqBBM: upper bound, essential task 1} is contained in the event that the moving barrier is crossed at some time, it is enough to bound
	\[
	\mathbb{P}\Bigl(
	\exists t\in[0,T-T_0],\ \exists w\in\mathcal{M}_t\colon
	\bar{Y}_t^w\ge \bar{A}_r(t)+\frac{\lambda}{2}
	\Bigr).
	\]
	
	Let $\tau_0(\bar{Y})$ denote the first hitting time of zero for the diffusion $\bar{Y}$ under $\mathbb{E}^{(\lambda/2)}$, namely $\bar{Y}_t=\lambda/2+\int_0^t\bar{\sigma}(s)\,\td B_s$. Consider the moving barrier $t\mapsto\bar{A}_r(t)+\lambda/2$, $t\in[0,T-T_0]$, and let $\mathscr{L}$ be the set consisting of particles that hit this barrier for the first time before time $T-T_0$. By continuity of the particle trajectories,
	\[
	\Bigl\{
	\exists t\in[0,T-T_0],\ \exists w\in\mathcal{M}_t\colon
	\bar{Y}_t^w\ge\bar{A}_r(t)+\frac{\lambda}{2}
	\Bigr\}
	=
	\{\mathscr{L}\neq\varnothing\}.
	\]
	Therefore, by Markov's inequality and the many-to-one lemma for BBM,
	\begin{equation}\label{eqBBM: stopping line many to one}
		\begin{aligned}
			&\mathbb{P}\Bigl(
			\exists t\in[0,T-T_0],\ \exists w\in\mathcal{M}_t\colon
			\bar{Y}_t^w\ge\bar{A}_r(t)+\frac{\lambda}{2}
			\Bigr)
			\le \mathbb{E}[\lvert\mathscr{L}\rvert]\\
			&\quad=
			\int_0^{T-T_0}\exp\{t\}
			\mathbb{P}\Bigl(
			\tau_0\Bigl(\bar{A}_r+\frac{\lambda}{2}-\bar{Y}\Bigr)
			\in\td t
			\Bigr),
		\end{aligned}
	\end{equation}
	where, in the last line, $\bar{Y}_t=\int_0^t\bar{\sigma}(s)\,\td B_s$ denotes the motion of a single particle.
	
	We now apply Lemma
	\ref{Appendix: Girsanov, exceeding probability} with $c_A=\sqrt{2}$ to
	the first-hitting-time measure in
	\eqref{eqBBM: stopping line many to one}. Since $\exp\{-c_A^2t/2\}=\exp\{-t\}$, the factor $\exp\{-t\}$ arising from the Girsanov transform cancels exactly
	with the factor $\exp\{t\}$ in \eqref{eqBBM: stopping line many to one}. We therefore obtain
\begin{equation}\label{eqBBM: upper bound, exceeding prob of bar Y}
	\begin{aligned}
		&\mathbb{P}\Bigl(
		\exists t\in[0,T-T_0],\ \exists w\in\mathcal{M}_t\colon
		\bar{Y}_t^w\ge \bar{A}_r(t)+\frac{\lambda}{2}
		\Bigr)\\
		&\quad\le
		\exp\Bigl\{-\frac{\bar{k}_r(0)}{\bar{\sigma}(0)^2}
		\frac{\lambda}{2}\Bigr\}
		\mathbb{E}^{(\lambda/2)}\Biggl[
		\exp\Biggl\{
		\int_0^{\tau_0(\bar{Y})}
		\left[
		-\Bigl(\frac{\bar{k}_r}{\bar{\sigma}^2}\Bigr)'(s)\bar{Y}_s
		+\frac{\bar{\eta}(s)+\bar{r}(s)}{\bar{\sigma}(s)}
		\right]\td s
		\Biggr\}
		\mathds{1}_{\{\tau_0(\bar{Y})\le T-T_0\}}
		\Biggr]\\
		&\quad\le
		\exp\{-c_2(T_0)\lambda\}
		\mathbb{E}^{(\lambda/2)}\Biggl[
		\exp\Biggl\{
		\int_0^{\tau_0(\bar{Y})}
		\left[
		-\Bigl(\frac{\bar{k}_r}{\bar{\sigma}^2}\Bigr)'(s)\bar{Y}_s
		+\frac{\bar{\eta}(s)+\bar{r}(s)}{\bar{\sigma}(s)}
		\right]\td s
		\Biggr\}
		\mathds{1}_{\{\tau_0(\bar{Y})\le T-T_0\}}
		\Biggr].
	\end{aligned}
\end{equation}

	where the last inequality follows from
	\eqref{eqBBM: bar kr(0)/bar sigma^2}, and
	\begin{equation}\label{eqBBM: c_2(T_0)}
		c_2(T_0):=\frac{\bar{k}_r(0)}{2\bar{\sigma}(0)^2}
		\ge c(1+2T_0)^\alpha.
	\end{equation}
	
	By the Feynman--Kac representation \eqref{eqH: PDE representation of hitting in some interval}, the expectation in the last display is
	\begin{equation}\label{eqBBM: upper bound, Airy PDE representation for Expectation}
		\int_0^{T-T_0}\frac{1}{2}\bar{\sigma}^2(t)
		\left.\partial_y G_r(\lambda/2,y;0,t)\right|_{y=0}\,\td t,
	\end{equation}
	where $G_r$ is the fundamental solution of
	\begin{equation}\label{eqBBM: Airy-like PDE}
		\partial_tu
		=\frac{1}{2}\bar{\sigma}(t)^2\partial_{xx}u
		-\Bigl(\frac{\bar{k}_r}{\bar{\sigma}^2}\Bigr)'(t)\,xu
		+\frac{\bar{\eta}(t)+\bar{r}(t)}{\bar{\sigma}(t)}u,
	\end{equation}
	on $[0,T-T_0]\times\mathbb{R}_+$ with Dirichlet boundary condition $u(t,0)=0$.
	
	For this equation, Proposition \ref{propoAiry: Airy estimate} applies with $c_A=\sqrt{2}$, $c_0=2$, $s_0=1+2T_0$ and
	\begin{equation}\label{eqBBM: V function}
		V(t):=\frac{2\sqrt{2}}{\bar{\sigma}^2(t)}
		\Bigl(\frac{1}{\bar{\sigma}(t)}\Bigr)'
		=2\alpha(1+2T_0+2t)^{3\alpha-1}.
	\end{equation}
	For $t\ge L_\star(1+2T_0)^{2/3}$ with $L_\star$ appearing in Theorem \ref{propoAiry: Airy estimate} to ensure the Airy estimate, \eqref{eqAiry: upper bound} gives
	\begin{equation}\label{eqBBM: partial derivative of G, 1}
		\begin{aligned}
			\left.\partial_y G_r(\lambda/2,y;0,t)\right|_{y=0}
			&\le C\lambda\sqrt{V(0)V(t)}
			\exp\Bigl\{\int_0^t\frac{\bar{r}(s)}{\bar{\sigma}(s)}\,\td s\Bigr\}\\
			&\le C\lambda(1+2T_0+2t)^{(3\alpha-1)/2}
			\exp\{-ct^{\alpha-\beta}\},
		\end{aligned}
	\end{equation}
	after increasing $C$ and decreasing $c$, since $T_0$ is fixed and $\beta<\alpha$. For $t\le L_\star(1+2T_0)^{2/3}$, the killed heat-kernel estimate used in \eqref{eqA: killed BB ker estimate, derivative} gives
	\begin{equation}\label{eqBBM: partial derivative of G, 2}
		\left.\partial_y G_r(\lambda/2,y;0,t)\right|_{y=0}
		\le C\lambda^{-2}
		\exp\Bigl\{
		\int_0^t\Bigl(\eta(s)+\frac{r(s)}{\sigma(s)}\Bigr)\,\td s
		\Bigr\}.
	\end{equation}
	Plugging
	\eqref{eqBBM: partial derivative of G, 1}, and
	\eqref{eqBBM: partial derivative of G, 2} into \eqref{eqBBM: upper bound, Airy PDE representation for Expectation}, we obtain
	\begin{equation}\label{eqBBM: upper bound, int ...< C lambda}
\begin{aligned}
	&\int_0^{T-T_0}\frac{1}{2}\bar{\sigma}^2(t)
	\left.\partial_y G_r(\lambda/2,y;0,t)\right|_{y=0}\,\td t\\
	&\quad\le
	C\lambda\int_0^{(1+2T_0)^{2/3}}
	(1+2T_0+2t)^{-2\alpha}\,\td t+
	C\lambda\int_{(1+2T_0)^{2/3}}^T
	(1+2T_0+2t)^{-(1+\alpha)/2}
	\exp\{-ct^{\alpha-\beta}\}\,\td t\\
	&\quad\le C\lambda.
\end{aligned}
	\end{equation}
	Plugging this into \eqref{eqBBM: upper bound, exceeding prob of bar Y} yields
	\[
	\mathbb{P}\Bigl(
	\exists t\in[0,T-T_0],\ \exists w\in\mathcal{M}_t\colon
	\bar{Y}_t^w\ge\bar{A}_r(t)+\frac{\lambda}{2}
	\Bigr)
	\le C\lambda\exp\{-c_2(T_0)\lambda\},
	\]
	for all $\lambda$ large enough. This proves \eqref{eqBBM: upper bound, essential task 1}, and hence the proposition.
\end{proof}

In the proof above, we actually prove a stronger result than Proposition \ref{proposition: BBM tight, upper bound}: For any $T_0$ large, $\lambda>A_\alpha(T_0)$, we have
\begin{equation}\label{eqBBM: stronger upper tail}
	\mathbb{P}\Bigl(
	\max_{v\in\mathcal{N}_T}Y_T^{\alpha,v}\ge A_\alpha(T)+\lambda
	\Bigr)
	\le C\exp\{-c_1(T_0)\lambda^2\}
	+C\lambda\exp\{-c_2(T_0)\lambda\},
\end{equation}
with $c_1(T_0)$ and $c_2(T_0)$ defined in \eqref{eqBBM: c_1(T_0)} and \eqref{eqBBM: c_2(T_0)} respectively. Since we can tune $c_2(T_0)$ to be arbitrarily large and let $\lambda\ge \frac{c_2(T_0)}{c_1(T_0)}$, we have the following result:
\begin{proposition}\label{propo: stronger upper tial for BBM}
	For any $q\ge1$, there exists large $T_q>0$, $x_q>0$, and some constant $C_q$ such that
	\[
	\mathbb{P}\Bigl(
	\max_{v\in\mathcal{N}_T}Y_T^{\alpha,v}-A_\alpha(T)>x
	\Bigr)
	\le C_q\exp\{-qx\},
	\qquad T\ge T_q,\quad x\ge x_q.
	\]
\end{proposition}

\subsection{Lower tail}\label{subsection: BBM Ltail}
In this subsection we prove the lower tail in Theorem \ref{theorem: BBM tightness of recentered maximum}:
\begin{proposition}\label{proposition: BBM tight, lower bound}
	Fix any $\alpha\in(0,1/2]$. There exist $c>0$ and $\lambda_*>1$, $T_*>1$ such that
	\[
	\mathbb{P}\Bigl(
	\max_{v\in\mathcal{N}_T}Y_T^{\alpha,v}\le A_\alpha(T)-\lambda
	\Bigr)
	\le \exp\{-c\lambda\},
	\qquad \lambda\ge\lambda_*,\quad T\ge T_*,
	\]
	where $A_\alpha(T)$ is defined in \eqref{eqIntroduction: A(t)} with $d:=1$.
\end{proposition}

Instead of turning to a barrier and second moment argument, as in \cite{MZ2016_slowdownBBM}, we use the scale decomposition idea introduced in Subsection \ref{subsec: main ideas}. Take some large $H>1$ and define the geometrically increasing time mesh
\begin{equation}\label{eqBBM: lower mesh}
	T_m:=\frac{2^mH-1}{2},\qquad m=0,\ldots,M.
\end{equation}
Then $1+2T_m=2^mH$ and $T_{m+1}-T_m=2^{m-1}H$.

For a $Y^\alpha$-particle, $Y_t^\alpha=\int_0^t\sqrt{2}\,(1+2s)^{-\alpha}\,\td B_s$, we have, for $t\in[0,T_{m+1}-T_m]$,
\[
\begin{aligned}
	Y_{t+T_m}^\alpha-Y_{T_m}^\alpha
	&=\int_{T_m}^{t+T_m}\sqrt{2}(1+2s)^{-\alpha}\,\td B_s\\
	&\overset{\mathrm d}{=}
	(2^mH)^{-\alpha}
	\int_0^t\sqrt{2}\bigl(1+\frac{s}{2^{m-1}H}\bigr)^{-\alpha}\,\td B_s.
\end{aligned}
\]
Thus each particle's movement during $[T_m,T_{m+1}]$ has the law
\begin{equation}\label{eqBBM: Y and Z particles}
	\Bigl\{Y_{t+T_m}^\alpha-Y_{T_m}^\alpha\colon
	0\le t\le2^{m-1}H\Bigr\}
	\overset{\mathrm d}{=}
	\Bigl\{(2^mH)^{-\alpha}Z_t^{(m)}\colon
	0\le t\le2^{m-1}H\Bigr\},
\end{equation}
where
\begin{equation}\label{eqBBM: lower bound, Zm definition}
	Z_t^{(m)}:=\int_0^t\hat{\sigma}_m(s)\,\td B_s,
	\qquad
	\hat{\sigma}_m(s):=\sqrt{2}\,
	\Bigl(1+\frac{s}{2^{m-1}H}\Bigr)^{-\alpha}.
\end{equation}
The process $Z^{(m)}$ is in the macroscopic slowdown regime of \cite{FangZeitouni2012_slowdown_BBM,MZ2016_slowdownBBM}. For each $m$, let $\{Z_t^{(m),w}\colon 0\le t\le2^{m-1}H,\ w\in\mathcal{N}_t^Z\}$ be a BBM with branching rate $1$, binary splitting, and variance profile $\hat{\sigma}_m$.

By Lemma \ref{lemma: macro BBM, lower bound}, applied with terminal time $2^{m-1}H$, for any $\beta\in(0,\frac{1}{3}\wedge\alpha)$ there exists $c>0$ such that, uniformly in $m$ and $H$ large,
\begin{equation}\label{eqBBM: lower bound, macro input}
	\mathbb{P}\Bigl(
	\max_{w\in\mathcal{N}^Z_{2^{m-1}H}}
	Z_{2^{m-1}H}^{(m),w}
	\ge
	\int_0^{2^{m-1}H}k_m(s)\,\td s-(2^{m-1}H)^\beta
	\Bigr)
	\ge 1-\exp\{-c(2^{m-1}H)^\beta\},
\end{equation}
where
\[
k_m(s):=\sqrt{2}\,\hat{\sigma}_m(s)
-\frac{\alpha_1}{\sqrt{2}}
\lvert\hat{\sigma}_m'(s)\rvert^{2/3}
\hat{\sigma}_m(s)^{1/3}.
\]
A direct scaling computation gives
\begin{equation}\label{eqBBM: lower bound, km scaling}
	(2^mH)^{-\alpha}
	\int_0^{2^{m-1}H}k_m(s)\,\td s
	=
	\int_{T_m}^{T_{m+1}}k_\alpha(s)\,\td s,
\end{equation}
where
\[
k_\alpha(s):=\sqrt{2}\,\sigma_\alpha(s)
-\frac{\alpha_1}{\sqrt{2}}
\lvert\sigma_\alpha'(s)\rvert^{2/3}
\sigma_\alpha(s)^{1/3}
\]
is the same as the integrand in \ref{eqIntroduction: A(t)}.
Consequently, from \eqref{eqBBM: lower bound, macro input} and \eqref{eqBBM: lower bound, km scaling}, we have
\begin{equation}\label{eqBBM: lower bound, Z BBM lower bound}
	\begin{aligned}
		&\mathbb{P}\Bigl(
		\max_{w\in\mathcal{N}^Z_{2^{m-1}H}}
		(2^mH)^{-\alpha}Z_{2^{m-1}H}^{(m),w}
		\ge
		\int_{T_m}^{T_{m+1}}k_\alpha(s)\,\td s
		-(2^mH)^{-\alpha}(2^{m-1}H)^\beta
		\Bigr)\\
		&\quad\ge 1-\exp\{-c(2^mH)^\beta\}.
	\end{aligned}
\end{equation}

We now use a greedy construction. Starting from a particle $v_0\in\mathcal{N}_{T_0}$, we choose one offspring of $v_0$ at $T_1$ with the largest displacement during $[T_0,T_1]$ and denote it by $v_1$. Then we iterate this process until we find a sequence of leading particles $v_1,\ldots,v_M$. See Figure \ref{fig: leading particles in BBM} for an illustration. By the branching property and \eqref{eqBBM: lower bound, Z BBM lower bound},
\begin{equation}\label{eqBBM: lower bound for max XT- XT0 1}
	\begin{aligned}
		&\mathbb{P}\Bigl(
		\max_{v'\in\mathcal{N}_T^{T_0,v_0}}
		\bigl(Y_T^{\alpha,v'}-Y_{T_0}^{\alpha,v_0}\bigr)
		\ge
		\int_{T_0}^Tk_\alpha(s)\,\td s
		-\sum_{m=0}^{M-1}(2^mH)^{-\alpha}(2^{m-1}H)^\beta
		\Bigr)\\
		&\ge
		\prod_{m=0}^{M-1}\Bigl(1-\exp\{-c(2^mH)^\beta\}\Bigr)\ge
		1-\sum_{m=0}^{M-1}\exp\{-c(2^mH)^\beta\}
		\ge 1-\exp\{-c'H^\beta\}.
	\end{aligned}
\end{equation}
Fix one $\beta\in(0,\frac{1}{3}\wedge\frac{\alpha}{2})$, the series in the event of \eqref{eqBBM: lower bound for max XT- XT0 1} satisfies
\[
\sum_{m=0}^{M-1}(2^mH)^{-\alpha}(2^{m-1}H)^\beta
\le CH^{\beta-\alpha}
\sum_{m=0}^{\infty}2^{-m(\alpha-\beta)}
\le C_\alpha,
\]
for some $C_{\alpha}>0$ only determined by $\alpha$ (recall that we ask $H>1$ to be large). Therefore, for every $v_0\in\mathcal{N}_{T_0}$, we get the following lower tail for sub-BBM originated from $v_0$:
\begin{equation}\label{eqBBM: lower bound for max XT- XT0}
	\mathbb{P}\Bigl(
	\max_{v'\in\mathcal{N}_T^{T_0,v_0}}
	\bigl(Y_T^{\alpha,v'}-Y_{T_0}^{\alpha,v_0}\bigr)
	\ge
	\int_{T_0}^Tk_\alpha(s)\,\td s-C_\alpha
	\Bigr)
	\ge 1-\exp\{-cH^\beta\}.
\end{equation}

\begin{figure}[htbp]
	\centering
	\includegraphics[width=0.6\textwidth]{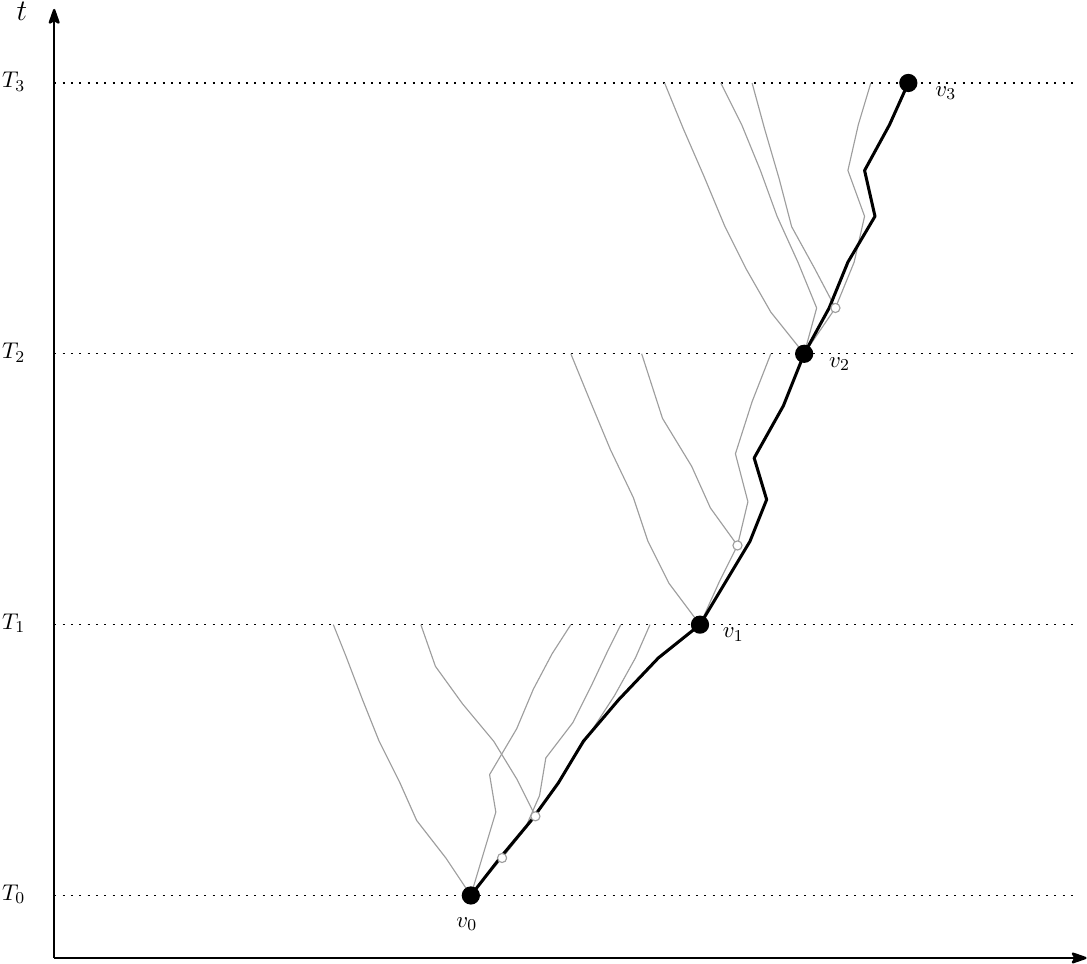}
	\caption{The leading particles in BBM}
	\label{fig: leading particles in BBM}
\end{figure}

\begin{proof}[Proof of Proposition \ref{proposition: BBM tight, lower bound}]
	Fix $\alpha\in(0,1/2]$. First consider the case $\lambda>10T$: To estimate
	\[
	\mathbb{P}\Bigl(
	\max_{v\in\mathcal{N}_T}Y_T^{\alpha,v}\le A_\alpha(T)-\lambda
	\Bigr),
	\]
	we simply use a union bound and a Gaussian tail estimate. Specifically, by \eqref{eqIntroduction: Aalpha(T) expansion}, we have $A_\alpha(T)\le 4T^{1-\alpha}$. For one single particle $v\in\mathcal{N}_T$, its displacement $Y_T^{\alpha,v}$ is a centered Gaussian with variance $\int_0^T\sigma_\alpha(s)^2\,\td s\le 2T$ by \eqref{eqBBM: sigma}. By our assumption $\lambda>10T$, we get
	\[
	\mathbb{P}\Bigl(
	\max_{v\in\mathcal{N}_T}Y_T^{\alpha,v}\le A_\alpha(T)-\lambda
	\Bigr)
	\le
	\mathbb{P}\Bigl(Y_T^{\alpha,v}\le A_\alpha(T)-\lambda\Bigr)
	\le \exp\{-c\lambda\},
	\]
	for some $c>0$ by the Gaussian tail.
	
	It remains to consider the case when $\lambda\le 10T$. Choose $M\in\mathbb{Z}_+$ such that
	$H:=2^{-M}(1+2T)\in\bigl[\lambda/100,\lambda/50\bigr]$.
	Then $T_M=T$ for the mesh \eqref{eqBBM: lower mesh}, and
	\begin{equation}\label{eqBBM: T_0 asymp lambda}
		T_0=\frac{H-1}{2},
		\qquad c_1\lambda\le T_0\le c_2\lambda,
	\end{equation}
	for deterministic constants $c_1,c_2\in(0,1/2)$ and all large $\lambda$.
	
	Define the event basically saying the existence of some particle with rather small displacement during $[0,T_0]$:
	\[
	\mathcal{E}_{T_0}:=
	\Bigl\{
	\exists v\in\mathcal{N}_{T_0}\colon
	Y_{T_0}^{\alpha,v}
	\le \int_0^{T_0}k_\alpha(s)\,\td s-\lambda+C_\alpha
	\Bigr\},
	\]
	where $C_\alpha>0$ appears in \eqref{eqBBM: lower bound for max XT- XT0}. Then we have
	\begin{equation}\label{eqBBM: P(left tail)}
		\begin{aligned}
			&\mathbb{P}\Bigl(
			\max_{v\in\mathcal{N}_T}Y_T^{\alpha,v_0}
			\le A_\alpha(T)-\lambda
			\Bigr)\\
			&\quad\le
			\mathbb{P}(\mathcal{E}_{T_0})
			+\mathbb{P}\Bigl(
			\forall v_0\in\mathcal{N}_{T_0}\colon
			\max_{w\in\mathcal{N}_T^{T_0,v_0}}
			Y_T^{\alpha,w}-Y_{T_0}^{\alpha,v}
			\le \int_{T_0}^Tk_\alpha(s)\,\td s-C_\alpha
			\Bigr).
		\end{aligned}
	\end{equation}
	We estimates the two probability above separately: By a union bound, a Gaussian tail, \eqref{eqBBM: T_0 asymp lambda}, and $\int_0^{T_0}k_\alpha(s)\,\td s\asymp T_0^{1-\alpha}$ by \eqref{eqIntroduction: Aalpha(T) expansion}, we have
	\begin{equation}\label{eqBBM: P(E _T0)}
		\mathbb{P}(\mathcal{E}_{T_0})
		\le \exp\{T_0\}
		\mathbb{P}\Bigl(
		Y_{T_0}^{\alpha,v}
		\le \int_0^{T_0}k_\alpha(s)\,\td s-\lambda+C_\alpha
		\Bigr)
		\le C\exp\{-c\lambda\}.
	\end{equation}
	
	To estimate
	\[
	I(T_0):=
	\mathbb{P}\Bigl(
	\forall v\in\mathcal{N}_{T_0}\colon
	\max_{w\in\mathcal{N}_T^{T_0,v_0}}
	Y_T^{\alpha,w}-Y_{T_0}^{\alpha,v_0}
	\le \int_{T_0}^Tk_\alpha(s)\,\td s-C_\alpha
	\Bigr),
	\]
	it suffices to note that by branching structure, sub-BBMs starting from all particles in $\mathcal{N}_{T_0}$ are independent and have the same law, satisfying \eqref{eqBBM: lower bound for max XT- XT0}. As a result,
	\[
	I(T_0)
	\le \mathbb{E}\Bigl[\exp\{-cH^\beta N_{T_0}\}\Bigr]
	\le \mathbb{E}\Bigl[\exp\{-cN_{T_0}\}\Bigr],
	\]
	for some $\beta\in(0,1/3\wedge\alpha/2)$, where $N_{T_0}$ is the random number of particles at $T_0$. By \eqref{eqBBM: number of BBM particles}, we have $\mathbb{P}\bigl(N_{T_0}\le\exp\{T_0/2\}\bigr)\le\exp\{-T_0/2\}$. Thus
	\begin{equation}\label{eqBBM: I(T_0) <}
		I(T_0)
		\le \exp\{-T_0/2\}
		+\exp\bigl\{-c\exp\{T_0/2\}\bigr\}
		\le C\exp\{-c\lambda\},
	\end{equation}
	where the last inequality is due to \eqref{eqBBM: T_0 asymp lambda}. Plugging \eqref{eqBBM: P(E _T0)} and \eqref{eqBBM: I(T_0) <} into \eqref{eqBBM: P(left tail)} finishes the proof.
\end{proof}

The following lemma is for the macroscopically slowdown BBM, which is used to prove \eqref{eqBBM: lower bound, Z BBM lower bound}.
\begin{lemma}\label{lemma: macro BBM, lower bound}
	Consider a BBM $\{Z_t^v\colon t\in[0,\mathbf{T}],\ v\in\mathcal{N}_t^Z\}$ with branching rate $1$, binary splitting, and variance profile
	\[
	\sigma(t)=\widetilde{\sigma}(t/\mathbf{T}),
	\qquad t\in[0,\mathbf{T}],
	\]
	where $\widetilde{\sigma}\in C^2([0,1];(0,\infty))$ is strictly decreasing and $\inf_{u\in[0,1]}\widetilde{\sigma}(u)>0$. Let
	\[
	\eta(t):=\alpha_1\lvert\sigma'(t)\rvert^{2/3}\sigma(t)^{1/3},
	\]
	where $\alpha_1$ is given in \eqref{eqA: Airy eigenvalues}. Then, for any $\beta\in(0,1/3)$, there exists $c>0$ such that
	\begin{equation}\label{eqBBM: macro BBM, lower bound}
		\mathbb{P}\Bigl(
		\max_{v\in\mathcal{N}_T}Z_{\mathbf{T}}^v
		\le \int_0^{\mathbf{T}}
		\Bigl(\sqrt{2}\,\sigma(s)-\frac{\eta(s)}{\sqrt{2}}\Bigr)\,\td s
		-\mathbf{T}^\beta
		\Bigr)
		\le \exp\{-c\mathbf{T}^\beta\},
	\end{equation}
	for all $\mathbf{T}$ large enough.
\end{lemma}

\begin{proof}
	Set
	\begin{equation}\label{eqBBM: lower, k(s) for macro}
		k(s):=\sqrt{2}\,\sigma(s)-\frac{\eta(s)}{\sqrt{2}},
	\end{equation}
	and $S_{\mathbf{T}}:=\int_0^{\mathbf{T}}k(s)\,\td s$. Choose a small constant $\delta>0$ such that $\int_0^{\delta\mathbf{T}^\beta}k(s)\,\td s\le\mathbf{T}^\beta/4$ for all large $\mathbf{T}$. This is possible because $\sigma$ is uniformly bounded and $\eta\ge0$. Put $t_*:=\delta\mathbf{T}^\beta$ and define
	\[
	\widetilde{\mathcal{E}}:=
	\Bigl\{
	\exists v\in\mathcal{N}_{t_*}\colon
	Z_{t_*}^v\le-\frac{\mathbf{T}^\beta}{2}
	\Bigr\}.
	\]
	By a union bound, the many-to-one lemma, and the Gaussian tail estimate,
	\begin{equation}\label{eqBBM, lower, P(tilde E) <}
	\mathbb{P}(\widetilde{\mathcal{E}})
\le \exp\{t_*\}
\mathbb{P}\Bigl(
\mathcal{N}\Bigl(0,\int_0^{t_*}\sigma^2(s)\,\td s\Bigr)
\le-\frac{\mathbf{T}^\beta}{2}
\Bigr)
\le \exp\{-c\mathbf{T}^\beta\},
	\end{equation}
	after choosing $\delta$ small enough.
	
	On $\widetilde{\mathcal{E}}^c$, all particles at time $t_*$ have position greater than $-\mathbf{T}^\beta/2$. Hence, if
	\[
	\max_{v\in\mathcal{N}_T}Z_{\mathbf{T}}^v
	\le S_{\mathbf{T}}-\mathbf{T}^\beta,
	\]
	then for every $u\in\mathcal{N}_{t_*}$,
	\[
	\max_{v\in\mathcal{N}_{\mathbf{T}}^{t_*,u}}
	\bigl(Z_{\mathbf{T}}^v-Z_{t_*}^u\bigr)
	\le S_{\mathbf{T}}-\frac{\mathbf{T}^\beta}{2}
	\le \int_{t_*}^{\mathbf{T}}k(s)\,\td s-\frac{\mathbf{T}^\beta}{4}.
	\]
	Therefore, conditional on $\mathcal{F}_{t_*}$,
	\[
	\mathbb{P}\Bigl(
	\max_{v\in\mathcal{N}_T}Z_{\mathbf{T}}^v
	\le S_{\mathbf{T}}-\mathbf{T}^\beta,\ \tilde  {\mathcal E}^c
	\,\Big|\,\mathcal{F}_{t_*}
	\Bigr)
	\le p_{\mathbf{T}}^{N_{t_*}},
	\]
	where
	\[
	p_{\mathbf{T}}:=
	\mathbb{P}\Bigl(
	\max_{v\in\mathcal{N}_{\mathbf{T}}^{t_*,u}}
	\bigl(Z_{\mathbf{T}}^v-Z_{t_*}^u\bigr)
	\le \int_{t_*}^{\mathbf{T}}k(s)\,\td s-\frac{\mathbf{T}^\beta}{4}
	\Bigr).
	\]
	To employ Theorem \ref{theorem: MZ2016 recentered maximum until the log T term}, we do a renormalization on the time interval: During $[t_*=\delta T^\beta,T]$, the displacement of each particle has the variance profile
	\[
	\widetilde{\sigma}_T^*\Bigl(\frac{s}{T-t_*}\Bigr),
	\qquad
	\text{with }\quad
	\sigma_T^*(r):=\widetilde{\sigma}\Bigl(
	\frac{t_*}{T}+\Bigl(1-\frac{t_*}{T}\Bigr)r
	\Bigr),
	\]
	which falls in the setting of Theorem \ref{theorem: MZ2016 recentered maximum until the log T term} for any fixed $\mathbf{T}$. By \eqref{eqBBM: macro BBM, lower bound} with $K=1$, applied on the interval $[t_*,\mathbf{T}]$, we have $p_{\mathbf{T}}\ge C\exp\{-1/\widetilde{\sigma}_T^*(0)\}\ge p_0>0$ for some numerical $p_0$. Hence
	\[
	\begin{aligned}
		\mathbb{E}\bigl[p_{\mathbf{T}}^{N_{t_*}}\bigr]
		&\le \mathbb{P}\bigl(N_{t_*}\le\exp\{t_*/2\}\bigr)
		+(p_0)^{\exp\{t_*/2\}}\\
		&\le \exp\{-ct_*\}+\exp\bigl\{-\exp\{ct_*\}\bigr\}
		\le \exp\{-c\mathbf{T}^\beta\},
	\end{aligned}
	\]
	where we used \eqref{eqBBM: number of BBM particles} and the definition $t_* := \delta T^\beta$. Combining this bound with the estimate \eqref{eqBBM, lower, P(tilde E) <} on $\mathbb{P}( \tilde {\mathcal E} )$ proves \eqref{eqBBM: macro BBM, lower bound}.
\end{proof}

In the proofs above, we have used the following result from \cite{MZ2016_slowdownBBM}:
\begin{theorem}[{Theorem 1.1 of \cite{MZ2016_slowdownBBM}}]
	\label{theorem: MZ2016 recentered maximum until the log T term}
	Let $M_{\mathbf{T}}$ denote the maximal displacement at time $\mathbf{T}$ of the branching Brownian motion with the same variance profile $\sigma(t)=\widetilde{\sigma}(t/\mathbf{T})$ as in Lemma \ref{lemma: macro BBM, lower bound}. Set
	\begin{equation}\label{eq:def-mT}
		\begin{aligned}
			m_{\mathbf{T}}
			&:=\sqrt{2}\int_0^{\mathbf{T}}\sigma(t)\,\td t
			-(\sqrt{2})^{-1/3}\alpha_1 2^{-1/3}
			\int_0^{\mathbf{T}}
			\lvert\sigma'(t)\rvert^{2/3}
			\lvert\sigma(t)\rvert^{1/3}\,\td t -\frac{\sigma(\mathbf{T})}{\sqrt{2}}\log \mathbf{T},
		\end{aligned}
	\end{equation}
	where $-\alpha_1$ is the largest zero of the Airy function $\mathrm{Ai}$, as in \eqref{eqA: Airy eigenvalues}. Then there exists $C>0$ such that
	\begin{equation}\label{eq:MT-concentration}
		\lvert M_{\mathbf{T}}-m_{\mathbf{T}}\rvert
	\end{equation}
	converges in law as $T\to\infty$ and there is $C>0$ such that
	\begin{equation}\label{eq: MZ, left tail}
		\mathbb{P}\Bigl(
		\lvert M_{\mathbf{T}}-m_{\mathbf{T}}\rvert\le K
		\Bigr)
		\ge CK\exp\Bigl\{-\frac{K}{\sigma(0)}\Bigr\},
	\end{equation}
	uniformly for all $k\in[1,\mathbf{T}^{1/3}]$ when $\mathbf{T}$ is large.
\end{theorem}
We mention that in \cite{MZ2016_slowdownBBM}, the BBM has splitting time $\operatorname{Exp}(1/2)$, while here we set it to be $\operatorname{Exp}(1)$; this is the reason for the $\sqrt{2}$-contained factors appearing in $m_{\mathbf{T}}$.

\section{Proof for theorem \ref{theorem: tightness of recentered maximum for X alpha T}}\label{section: proof for the main results}
In this section, we prove Theorem~\ref{theorem: tightness of recentered maximum for X alpha T} on the maxima of the Gaussian field $X_\epsilon^\alpha$ defined in \eqref{eqModel: X epsilon alpha white noise integral} for any $\epsilon\in(0,1)$ and $\alpha\in(0,1/2]$.

As a preparation, we show the continuity of $X_\epsilon^\alpha$. By Proposition~\ref{proposition: L2 distances for X epsilon alpha},
\begin{equation}\label{eqIntroduction: L2 distance for X epsilon}
	\mathbb{E}\bigl[(X_\epsilon^\alpha(z_1)-X_\epsilon^\alpha(z_2))^2\bigr]
	\leq C\,\theta_\alpha(\epsilon)\frac{\lVert z_1-z_2\rVert_2}{\epsilon},
\end{equation}
where $\theta_\alpha(\epsilon):=(1+\log(1/\epsilon))^{-2\alpha}$. By \eqref{eqIntroduction: L2 distance for X epsilon} and the moments of Gaussian random variables,
\[
\mathbb{E}\bigl[(X_\epsilon^\alpha(z_1)-X_\epsilon^\alpha(z_2))^{2h}\bigr]
\leq C\biggl(\theta_\alpha(\epsilon)
\frac{\lVert z_1-z_2\rVert_2}{\epsilon}\biggr)^h,
\qquad h\geq2.
\]
Thus, the Kolmogorov continuity theorem implies that, for each fixed $\epsilon>0$, the Gaussian field $X_\epsilon^\alpha$ is almost surely continuous. Consequently,
$\sup_{x\in[0,1]^d}X_\epsilon^\alpha(x)=\sup_{x\in[0,1)^d}X_\epsilon^\alpha(x)$ almost surely. Hence, to prove Theorem~\ref{theorem: tightness of recentered maximum for X alpha T}, it suffices to consider $X_\epsilon^\alpha(x)$ for $x\in[0,1)^d$. We prove the upper and lower tails in Subsections~\ref{subsec: X field, upper tail} and \ref{subsec: X field, lower tail}, respectively.

\subsection{Auxiliary Gaussian fields}\label{subsection: auxiliary Gaussian fields}
In this subsection, we introduce two auxiliary Gaussian fields that will be used later in Subsections~\ref{subsec: X field, upper tail} and \ref{subsec: X field, lower tail} for comparison with $\{X_\epsilon^\alpha(x)\colon x\in[0,1)^d\}$. The mollified branching Brownian motion and the mollified Brownian sheet were introduced in \cite{Acosta2014_tightness} in the study of the recentered maximum of log-correlated Gaussian fields. Here we follow the same general idea, with several modifications adapted to our setting.

For any fixed dimension $d\geq2$, let $V_\epsilon:=\epsilon\mathbb{Z}^d\cap[0,1)^d$. Fix $\alpha\in(0,1/2]$, and let $s_0>0$ be a large but fixed constant. We define the mollified branching Brownian motion (MBBM) by
\begin{equation}\label{eqM: MBBM}
	\xi_\epsilon^{\alpha,z}(t)
	:=\int_{[0,t]\times\mathbb{R}^d}
	\sigma_\alpha(s)\,
	\mathds{1}_{\Lambda_1(e^{s-s_0}z)}(y)\,W(\td s,\td y),
	\qquad z\in V_\epsilon,\quad t\geq0,
\end{equation}
where
\[
\sigma_\alpha(s):=\sqrt{2}\,(1+2s)^{-\alpha},
\qquad
\Lambda_1(x):=x+[-1/2,1/2]^d,
\qquad s\geq0,\quad x\in\mathbb{R}^d,
\]
and $W$ denotes space--time white noise on $[0,\infty)\times\mathbb{R}^d$. The role of the shift $s_0$ is to slightly enlarge the covariance of the field $\{\xi_\epsilon^{\alpha,z}\colon z\in V_\epsilon\}$ without affecting its extreme-value behavior in any essential way. This modification will be convenient when we compare $\{\xi_\epsilon^{\alpha,z}\colon z\in V_\epsilon\}$ with a BBM in the proof of Proposition~\ref{proposition: MBBM exceeding particles num for a subset}.

For distinct $z_1,z_2\in V_\epsilon$, define $t_{z_1,z_2}:=\log(1/\lVert z_1-z_2\rVert_\infty)$. Since the boxes $\Lambda_1(e^{s-s_0}z_1)$ and $\Lambda_1(e^{s-s_0}z_2)$ are disjoint for all $s>t_{z_1,z_2}+s_0$, the increments of $\xi_\epsilon^{\alpha,z_1}$ and $\xi_\epsilon^{\alpha,z_2}$ are independent after time $t_{z_1,z_2}+s_0$. In particular, for every $t\geq s\geq t_{z_1,z_2}+s_0$,
\begin{equation}\label{eqMBBM: independent increments}
	\operatorname{Cov}\Bigl(
	\xi_\epsilon^{\alpha,z_1}(t)-\xi_\epsilon^{\alpha,z_1}(s),
	\xi_\epsilon^{\alpha,z_2}(t)-\xi_\epsilon^{\alpha,z_2}(s)
	\Bigr)=0.
\end{equation}

\begin{remark}
	The MBBM introduced in \eqref{eqM: MBBM} serves as an auxiliary model for studying the maximum of the Gaussian field $X_\epsilon^\alpha$, defined in \eqref{eqModel: X epsilon alpha white noise integral}, on the discrete grid $V_\epsilon$. The key point is that the kernel $q_t(\cdot)$ appearing in the integrand of \eqref{eqModel: X epsilon alpha white noise integral} is not compactly supported. Hence, $X_\epsilon^\alpha$ does not admit an independent-increment decomposition of the form \eqref{eqMBBM: independent increments}. The auxiliary MBBM field is introduced precisely to recover such a hierarchical independent-increment structure.
\end{remark}

The following lemma gives the covariance estimate for the mollified BBM constructed above.

\begin{lemma}\label{lemma: covariance estimate for MBBM}
	Let $\{\xi_\epsilon^{\alpha,z}(t)\colon t\geq0,\ z\in V_\epsilon\}$ be the MBBM defined in \eqref{eqM: MBBM}. There exists $s_*>0$ sufficiently large such that, for all $s_0>s_*$, all distinct $z_1,z_2\in V_\epsilon$, and all $t\geq t_{z_1,z_2}+s_0$,
	\begin{equation}\label{eqMBBM: cov}
		\int_0^{t_{z_1,z_2}}\sigma_\alpha(s)^2\,\td s
		\leq
		\operatorname{Cov}\bigl(\xi_\epsilon^{\alpha,z_1}(t),\xi_\epsilon^{\alpha,z_2}(t)\bigr)
		\leq
		\int_0^{t_{z_1,z_2}+s_0}\sigma_\alpha(s)^2\,\td s.
	\end{equation}
\end{lemma}

\begin{proof}
	By \eqref{eqM: MBBM},
	\begin{equation}\label{eqMBBM: exact covariance}
		\operatorname{Cov}\bigl(\xi_\epsilon^{\alpha,z_1}(t),\xi_\epsilon^{\alpha,z_2}(t)\bigr)
		=
		\int_0^t\sigma_\alpha(s)^2
		\bigl\lvert\Lambda_1(e^{s-s_0}z_1)\cap\Lambda_1(e^{s-s_0}z_2)\bigr\rvert\,\td s.
	\end{equation}
	Moreover, for every $s\geq0$,
	\[
	\bigl\lvert\Lambda_1(e^{s-s_0}z_1)\cap\Lambda_1(e^{s-s_0}z_2)\bigr\rvert
	=
	\prod_{i=1}^d\bigl(1-e^{s-s_0}\lvert z_1^i-z_2^i\rvert\bigr)_+,
	\]
	where $a_+:=\max\{a,0\}$. Since the intersection vanishes whenever $s>t_{z_1,z_2}+s_0$, for all $t\geq t_{z_1,z_2}+s_0$,
	\[
	\operatorname{Cov}\bigl(\xi_\epsilon^{\alpha,z_1}(t),\xi_\epsilon^{\alpha,z_2}(t)\bigr)
	=
	\int_0^{t_{z_1,z_2}+s_0}\sigma_\alpha(s)^2
	\prod_{i=1}^d\bigl(1-e^{s-s_0}\lvert z_1^i-z_2^i\rvert\bigr)_+\,\td s.
	\]
	The upper bound in \eqref{eqMBBM: cov} follows immediately because the product is at most $1$.
	
	For the lower bound, put $t_{12}:=t_{z_1,z_2}$ and $L:=t_{12}+s_0$. For $0\leq s\leq L$,
	\[
	\prod_{i=1}^d\bigl(1-e^{s-s_0}\lvert z_1^i-z_2^i\rvert\bigr)_+
	\geq1-\sum_{i=1}^d e^{s-s_0}\lvert z_1^i-z_2^i\rvert
	\geq1-de^{s-L}.
	\]
	Hence,
	\[
	\operatorname{Cov}\bigl(\xi_\epsilon^{\alpha,z_1}(t),\xi_\epsilon^{\alpha,z_2}(t)\bigr)
	\geq
	\int_0^L\sigma_\alpha(s)^2\,\td s
	-d\int_0^L\sigma_\alpha(s)^2e^{s-L}\,\td s.
	\]
	Moreover,
	\[
	\int_0^L(1+2s)^{-2\alpha}e^{s-L}\,\td s
	\leq C_\alpha(1+2L)^{-2\alpha},
	\]
	for some constant $C_\alpha>0$ independent of $z_1,z_2$, and $L$; this follows by estimating the integral separately over $[0,L/2]$ and $[L/2,L]$. On the other hand,
	\[
	\int_{t_{12}}^L\sigma_\alpha(s)^2\,\td s
	\geq2s_0(1+2L)^{-2\alpha}.
	\]
	Choosing $s_0$ sufficiently large so that the last display dominates $d\int_0^L\sigma_\alpha(s)^2e^{s-L}\,\td s$, we obtain
	\[
	\operatorname{Cov}\bigl(\xi_\epsilon^{\alpha,z_1}(t),\xi_\epsilon^{\alpha,z_2}(t)\bigr)
	\geq\int_0^{t_{12}}\sigma_\alpha(s)^2\,\td s,
	\]
	which proves the lower bound.
\end{proof}

The following proposition will be used in Subsection~\ref{subsec: X field, upper tail}.

\begin{proposition}\label{proposition: MBBM exceeding particles num for a subset}
	Let $\{\xi_\epsilon^{\alpha,z}(t)\colon t\geq0,\ z\in V_\epsilon\}$ be the MBBM defined in \eqref{eqM: MBBM}, and recall $A_\alpha(\cdot)$ from \eqref{eqIntroduction: A(t)}. Fix $T_r\geq s_0$, with $s_0$ as in Lemma~\ref{lemma: covariance estimate for MBBM}, and assume that $T_r=N_r\log 2$ for some $N_r\in\mathbb{Z}_+$. Then there exist constants $C,c>0$ and $T_*,\lambda_*>0$ such that, for every $N\in\mathbb{Z}_+$ with $T:=N\log 2\geq T_*$ and $\epsilon:=2^{-N}$, every subset $\Lambda\subset V_\epsilon$, and every $\lambda\geq\lambda_*$,
	\begin{equation}\label{eqMBBM: exceeding particles num for a subset 1}
		\mathbb{P}\Bigl(
		\max_{z\in\Lambda}\xi_\epsilon^{\alpha,z}(T+T_r)
		\geq A_\alpha(T+T_r)+\lambda-n
		\Bigr)
		\leq Ce^{-c(\lambda-n)},
		\qquad n\in\mathbb{Z}_+,\quad n\leq\lambda,
	\end{equation}
	and
	\begin{equation}\label{eqMBBM: exceeding particles num for a subset 2}
		\begin{aligned}
			&\mathbb{P}\Bigl(
			\max_{z\in\Lambda}\xi_\epsilon^{\alpha,z}(T+T_r)
			\geq A_\alpha(T+T_r)+\lambda-n
			\Bigr)\\
			&\quad\leq
			Ce^{-c(\lambda+n)}
			+C\lvert\Lambda\rvert
			\exp\bigl\{-d(T+T_r)-c\lambda+c(T+T_r)^\alpha(n+1)\bigr\},
			\qquad n\in\mathbb{Z}_+.
		\end{aligned}
	\end{equation}
\end{proposition}

\begin{proof}
	By assumption, $T=N\log 2$, $T_r=N_r\log 2$, and $\epsilon=2^{-N}$. We compare $\{\xi_\epsilon^{\alpha,z}(T+T_r)\colon z\in V_\epsilon\}$ with an auxiliary Gaussian field arising from a BBM with deterministic branching times. The construction proceeds in two steps.
	
	\begin{enumerate}
		\item \textbf{Dyadic hierarchy.}
		For each $k=0,1,\dots,N+N_r$, partition $[0,1)^d$ into the $2^{dk}$ half-open dyadic boxes of side length $2^{-k}$, with the convention that the partition at level $k=0$ consists of the single box $[0,1)^d$. For $z\in V_\epsilon$, let $I_k^{(z)}$ denote the index of the unique dyadic box of scale $2^{-k}$ containing $z$. For $z_1,z_2\in V_\epsilon$, define
		\[
		k_{z_1,z_2}:=
		\max\bigl\{0\leq k\leq N+N_r\colon I_k^{(z_1)}=I_k^{(z_2)}\bigr\}.
		\]
		Thus, $k_{z_1,z_2}$ is the largest dyadic scale at which $z_1$ and $z_2$ lie in the same box.
		
		\item \textbf{Construction of the auxiliary BBM field.}
		For each $1\leq k\leq N+N_r$ and $1\leq i\leq2^{dk}$, let $Y_k^i$ be independent centered Gaussian random variables with variance
		$\mathbb{E}[(Y_k^i)^2]=\int_{(k-1)\log 2}^{k\log 2}\sigma_\alpha(s)^2\,\td s$.
		For any $z\in V_\epsilon$, define
		\begin{equation}\label{eqMBBM: aux BBM}
			S_{T+T_r}^{\alpha,z}:=\sum_{k=1}^{N+N_r}Y_k^{I_k^{(z)}}.
		\end{equation}
	\end{enumerate}
	
	By construction, $\{S_{T+T_r}^{\alpha,z}\colon z\in V_\epsilon\}$ can be identified with a subcollection of the terminal positions of a BBM with deterministic branching period $\log 2$, offspring number $2^d$, and variance profile $\sigma_\alpha(\cdot)$. Denote this BBM by $\{Y_{T+T_r}^{\alpha,v}\colon v\in\mathcal{N}_{T+T_r}\}$, where $\mathcal{N}_{T+T_r}$ is the set of particles at time $T+T_r$.
	
	In the notation of Section~\ref{Appendix: BBM with fixed time}, this corresponds to $\Delta t=\log 2$, $m=2^d$, $c_0=2$, and the parameter $s_0$ in \eqref{eqaBBM: sigma} equal to $1$; in particular, $c_A=\sqrt{2d}$.
	
	Proposition~\ref{proposition: BBM fixed splitting time, tail estimate}, applied at terminal time $T+T_r$ with $K_0=\lambda$ and with $\Gamma$ corresponding to the labels indexed by $\Lambda$, gives, for $n\geq0$,
	\begin{equation}\label{eqMBBM: from S to Y BBM 1}
		\begin{aligned}
			&\mathbb{P}\Bigl(
			\max_{z\in\Lambda}S_{T+T_r}^{\alpha,z}
			\geq A_\alpha(T+T_r)+\lambda-n
			\Bigr)\\
			&\quad\leq
			Ce^{-c(\lambda+n)}
			+C\lvert\Lambda\rvert
			\exp\bigl\{-d(T+T_r)-c\lambda+c(T+T_r)^\alpha(n+1)\bigr\}.
		\end{aligned}
	\end{equation}
	Similarly, Proposition~\ref{proposition: aBBM exceeding probability} implies that, for $n\leq\lambda$,
	\begin{equation}\label{eqMBBM: from S to Y BBM 2}
		\mathbb{P}\Bigl(
		\max_{z\in\Lambda}S_{T+T_r}^{\alpha,z}
		\geq A_\alpha(T+T_r)+\lambda-n
		\Bigr)
		\leq Ce^{-c(\lambda-n)}.
	\end{equation}
	Indeed, if $\lambda-n$ is below the threshold required in Proposition~\ref{proposition: aBBM exceeding probability}, the bound follows trivially after increasing $C$.
	
	It remains to compare $\{\xi_\epsilon^{\alpha,z}(T+T_r)\colon z\in\Lambda\}$ and $\{S_{T+T_r}^{\alpha,z}\colon z\in\Lambda\}$. By construction, these fields have the same pointwise variances. Moreover, for $z_1,z_2\in V_\epsilon$,
	\[
	\operatorname{Cov}\bigl(S_{T+T_r}^{\alpha,z_1},S_{T+T_r}^{\alpha,z_2}\bigr)
	=
	\int_0^{k_{z_1,z_2}\log 2}\sigma_\alpha(s)^2\,\td s.
	\]
	For distinct $z_1,z_2$, since they belong to the same dyadic box of side length $2^{-k_{z_1,z_2}}$,
	$\lVert z_1-z_2\rVert_\infty\leq2^{-k_{z_1,z_2}}$, and hence $k_{z_1,z_2}\log 2\leq t_{z_1,z_2}$. Since also $t_{z_1,z_2}\leq T$ and $T_r\geq s_0$, Lemma~\ref{lemma: covariance estimate for MBBM} yields
	\[
	\operatorname{Cov}\bigl(S_{T+T_r}^{\alpha,z_1},S_{T+T_r}^{\alpha,z_2}\bigr)
	\leq
	\operatorname{Cov}\bigl(\xi_\epsilon^{\alpha,z_1}(T+T_r),
	\xi_\epsilon^{\alpha,z_2}(T+T_r)\bigr).
	\]
	For $z_1=z_2$, the variances are equal. Therefore, Slepian's comparison lemma (see Theorem~2.21 in \cite{AdlerTaylor2007_randomfieldGeometry}) yields, for every $u\in\mathbb{R}$,
	\[
	\mathbb{P}\Bigl(
	\max_{z\in\Lambda}\xi_\epsilon^{\alpha,z}(T+T_r)\geq u
	\Bigr)
	\leq
	\mathbb{P}\Bigl(
	\max_{z\in\Lambda}S_{T+T_r}^{\alpha,z}\geq u
	\Bigr).
	\]
	Taking $u=A_\alpha(T+T_r)+\lambda-n$ and combining this comparison with \eqref{eqMBBM: from S to Y BBM 1} and \eqref{eqMBBM: from S to Y BBM 2} completes the proof.
\end{proof}

The mollified BBM field $\xi_\epsilon^\alpha$ is indexed only by the lattice $V_\epsilon=[0,1)^d\cap\epsilon\mathbb{Z}^d$, whereas the Gaussian field $X_\epsilon^\alpha$ is indexed by the whole cube $[0,1)^d$. Thus, the two fields cannot be compared directly by Slepian's lemma. To control the fluctuations of $X_\epsilon^\alpha$ inside each $\epsilon$-box, we introduce an auxiliary Gaussian field $\{\psi_\epsilon^x\colon x\in[0,1)^d\}$.

Recall from Proposition~\ref{proposition: L2 distances for X epsilon alpha} that, for $x,y$ in the same $\epsilon$-box,
\begin{equation}\label{eqB: X epsilon local L2 distance}
	\mathbb{E}\bigl[(X_\epsilon^\alpha(x)-X_\epsilon^\alpha(y))^2\bigr]
	\leq C\,\theta_\alpha(\epsilon)
	\frac{\lVert x-y\rVert_2}{\epsilon},
\end{equation}
where $\theta_\alpha(\epsilon):=(1+\log(1/\epsilon))^{-2\alpha}$. We choose $\{\psi_\epsilon^x\}$ to be the Brownian-sheet field used in \cite{Acosta2014_tightness}. Its $L^2$ distance, recalled below in \eqref{eqB: Brownian sheet L2 distance}, has the same linear bound inside each $\epsilon$-box.

For $x=(x^1,\dots,x^d)\in[0,1)^d$, define
\[
[x]_\epsilon
:=\bigl(\epsilon\lfloor x^1/\epsilon\rfloor,\dots,
\epsilon\lfloor x^d/\epsilon\rfloor\bigr)\in V_\epsilon.
\]
Thus, $[x]_\epsilon$ is the lower-left corner of the unique $\epsilon$-box containing $x$. For $z\in V_\epsilon$ and $r>0$, write
\[
\Box_r^z
:=\bigl\{x\in[0,1)^d\colon
z^i\leq x^i<z^i+r,\ i=1,\dots,d\bigr\}.
\]

We now define the auxiliary Gaussian field $\psi_\epsilon^\cdot$. Fix $x,y\in[0,1)^d$. If $[x]_\epsilon\neq[y]_\epsilon$, we take $\psi_\epsilon^\cdot $ to be two independent Brownian sheets on $\Box_\epsilon^x$ and $\Box_\epsilon ^y$. Also, we set $\psi_\epsilon ^\cdot$ to be independent to any other auxiliary fields in the rest of this section. \par 

If $[x]_\epsilon=[y]_\epsilon=z\in V_\epsilon$, set
\[
\operatorname{Cov}(\psi_\epsilon^x,\psi_\epsilon^y)
=
\prod_{i=1}^d
\min\bigl\{L_{\epsilon,1}^z(x)^i,L_{\epsilon,1}^z(y)^i\bigr\},
\]
where $L_{\epsilon,1}^z\colon\Box_\epsilon^z\to\mathbf{1}+[0,1)^d$ is the affine map defined by
\[
L_{\epsilon,1}^z\biggl(z+\sum_{i=1}^d a_i\epsilon e_i\biggr)
=\mathbf{1}+\sum_{i=1}^d a_i e_i,
\qquad (a_1,\dots,a_d)\in[0,1)^d,
\]
and $\mathbf{1}=(1,\dots,1)$. Equivalently, when $x$ and $y$ belong to the same $\epsilon$-box, the covariance is the volume of the rectangle anchored at the origin whose opposite corner is the coordinatewise minimum of $L_{\epsilon,1}^z(x)$ and $L_{\epsilon,1}^z(y)$, as illustrated in Figure~\ref{fig: L epsilon 1 map}.

\begin{figure}[htbp]
	\centering
	\includegraphics[width=0.6\textwidth]{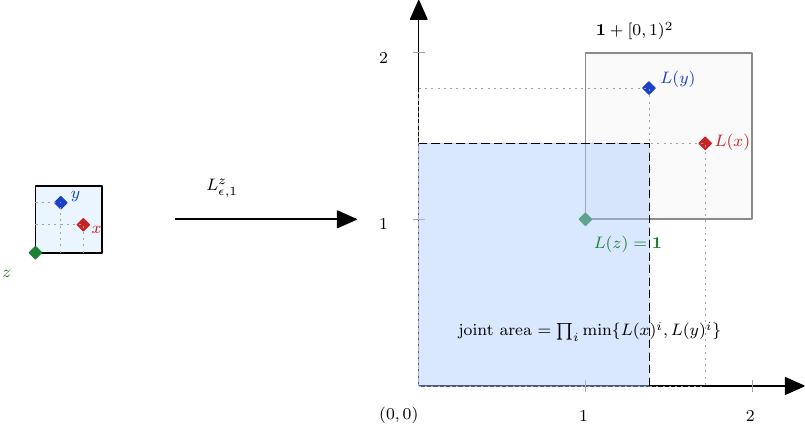}
	\caption{The affine map $L_{\epsilon,1}^z$.}
	\label{fig: L epsilon 1 map}
\end{figure}

By formulas (2.7)--(2.8) in \cite{Acosta2014_tightness}, the field $\{\psi_\epsilon^x\}$ satisfies
\begin{equation}\label{eqB: Brownian sheet variance}
	1\leq\operatorname{Var}(\psi_\epsilon^x)\leq2^d,
	\qquad x\in[0,1)^d,
\end{equation}
and, for every $z\in V_\epsilon$ and $x,y\in\Box_\epsilon^z$,
\begin{equation}\label{eqB: Brownian sheet L2 distance}
	\frac{\lVert x-y\rVert_1}{\epsilon}
	\leq
	\mathbb{E}\bigl[(\psi_\epsilon^x-\psi_\epsilon^y)^2\bigr]
	\leq
	2^d\frac{\lVert x-y\rVert_1}{\epsilon}.
\end{equation}
Finally, Lemma~2.2 of \cite{Acosta2014_tightness}, which follows from Fernique's majorizing-measure criterion, gives constants $c,C>0$, independent of $\lambda$ and $\epsilon$, such that
\begin{equation}\label{eqB: Brownian sheet max in a epsilon box}
	\sup_{v\in V_\epsilon}
	\mathbb{P}\Bigl(
	\sup_{x\in\Box_\epsilon^v}\psi_\epsilon^x\geq\lambda
	\Bigr)
	\leq Ce^{-c\lambda^2},
	\qquad \lambda\geq0,\quad\epsilon>0.
\end{equation}

\subsection{Upper Tail}\label{subsec: X field, upper tail}
In this subsection, we prove the following upper-tail estimate.

\begin{proposition}\label{Prop: upper bound for X epsilon}
	Fix $\alpha\in(0,1/2]$. For $\epsilon\in(0,1)$, set $T:=\log(1/\epsilon)$ and let $X_\epsilon^\alpha$ be the Gaussian field defined in \eqref{eqModel: X epsilon alpha white noise integral}. There exist constants $C,c>0$ and $T_*,\lambda_*>0$ such that
	\begin{equation}\label{eqU: upper bound, main result}
		\mathbb{P}\Bigl(
		\sup_{x\in[0,1)^d}X_\epsilon^\alpha(x)
		\geq A_\alpha(T)+\lambda
		\Bigr)
		\leq Ce^{-c\lambda}
	\end{equation}
	for every $T\geq T_*$ and $\lambda\geq\lambda_*$, where $A_\alpha(T)$ is given by \eqref{eqIntroduction: A(t)}.
\end{proposition}

According to Remark~\ref{Remark: dyadic subsequence}, it suffices to consider $\epsilon=2^{-N}$ for some $N\in\mathbb{Z}_+$, which makes it convenient to use Proposition~\ref{proposition: MBBM exceeding particles num for a subset}.

Instead of working directly with $X_\epsilon^\alpha$, we first study the rescaled field $\{X_{\delta\epsilon}^\alpha(\delta x)\colon x\in[0,1)^d\}$, where $\delta:= 2^{-m}>0$ for some $m$ large  is a small fixed constant to be chosen below. We mention that, the certain form $2^{-m}$ is only for convenience of using Proposition \ref{proposition: MBBM exceeding particles num for a subset} concerning the fixed time BBM. The point of this rescaling is that the $\delta$-rescaled field has, up to constants, larger covariance than $X_\epsilon^\alpha$, which is needed for the Slepian comparison. We consider the following form of the probability in \eqref{eqU: upper bound, main result}:
\[
\begin{aligned}
	&\mathbb{P}\Bigl(
	\sup_{x\in[0,1)^d}X_{\delta\epsilon}^\alpha(x)
	\geq A_\alpha\bigl(\log(1/(\delta\epsilon))\bigr)+\lambda
	\Bigr)\\
	&\qquad=
	\mathbb{P}\Bigl(
	\sup_{x\in[0,1)^d}X_{\delta\epsilon}^\alpha(x)
	\geq A_\alpha\bigl(T+\log(1/\delta)\bigr)+\lambda
	\Bigr).
\end{aligned}
\]
Assume, for notational simplicity, that $\delta^{-1}\in\mathbb{N}$. By a union bound over the $\delta$-boxes and translation invariance of the covariance structure,
\[
\begin{aligned}
	\mathbb{P}\Bigl(
	\sup_{x\in[0,1)^d}X_{\delta\epsilon}^\alpha(x)
	\geq A_\alpha\bigl(T+\log(1/\delta)\bigr)+\lambda
	\Bigr)
	\leq &
	\delta^{-d}\mathbb{P}\Bigl(
	\sup_{x\in[0,\delta)^d}X_{\delta\epsilon}^\alpha(x)
	\geq A_\alpha\bigl(T+\log(1/\delta)\bigr)+\lambda
	\Bigr)\\
	=&
	\delta^{-d}\mathbb{P}\Bigl(
	\sup_{x\in[0,1)^d}X_{\delta\epsilon}^\alpha(\delta x)
	\geq A_\alpha\bigl(T+\log(1/\delta)\bigr)+\lambda
	\Bigr).
\end{aligned}
\]
Moreover, by the expansion \eqref{eqIntroduction: Aalpha(T) expansion},
$\lvert A_\alpha(T+\log(1/\delta))-A_\alpha(T)\rvert\leq C_\delta$ uniformly for $T\geq0$. Therefore, it suffices to prove that, for some $C,c>0$,
\begin{equation}\label{eqU: upper bound, delta rescale field, main result}
	\mathbb{P}\Bigl(
	\sup_{x\in[0,1)^d}X_{\delta\epsilon}^\alpha(\delta x)
	\geq A_\alpha(T)+\lambda
	\Bigr)
	\leq Ce^{-c\lambda}
\end{equation}
for all $T\geq T_*$ and $\lambda\geq\lambda_*$.

Recall the mollified BBM $\{\xi_\epsilon^{\alpha,z}(t)\colon z\in V_\epsilon,\ t\geq0\}$ defined in \eqref{eqM: MBBM} and the Brownian-sheet field $\{\psi_\epsilon^x\colon x\in[0,1)^d\}$ from Subsection~\ref{subsection: auxiliary Gaussian fields}. We take these two auxiliary fields to be independent. Let
\begin{equation}\label{eqU: upper bound, theta}
	\theta(\epsilon):=C_{\mathrm{loc}}\bigl(1+\log(1/\epsilon)\bigr)^{-2\alpha}
	\asymp T^{-2\alpha},
\end{equation}
where $C_{\mathrm{loc}}$ is chosen large enough so that the local $L^2$ estimate for $X_\epsilon^\alpha$ in \eqref{eqK: alpha epsilon kernel, L2 distance} is dominated by $\theta(\epsilon)$.

We compare $\{X_{\delta\epsilon}^\alpha(\delta x)\colon x\in[0,1)^d\}$ with
\begin{equation}\label{eqU: comparison field Y}
	Y_x:=a(x)\xi_\epsilon^{\alpha,[x]_\epsilon}(T+T_r)
	+\sqrt{\theta(\epsilon)}\,\psi_\epsilon^x,
	\qquad x\in[0,1)^d,
\end{equation}
where $a\colon[0,1)^d\to\mathbb{R}_+$ is defined in \eqref{eqU: upper bound, a(x)}. For the Slepian comparison below, choose $\delta>0$ small and then $T_r>0$ so that
\begin{equation}\label{eqU: upper bound, coefficients for fields}
	\begin{aligned}
		&-C_X+\log(1/\delta)\geq\log C_d+s_0+10,\\
		&T_r=N_r\log 2\quad\text{for some }N_r\in\mathbb{Z}_+,\quad 
		T_r\geq\max\bigl\{s_0,C_K+\log(1/\delta)\bigr\},
	\end{aligned}
\end{equation}
where $s_0$ is the constant in Lemma~\ref{lemma: covariance estimate for MBBM}, and $C_d$ is a fixed constant such that
\begin{equation}\label{eqU: grid distance dimensional constant}
	\lVert[x]_\epsilon-[y]_\epsilon\rVert_\infty
	\geq C_d^{-1}\lVert x-y\rVert_\infty
	\qquad\text{whenever }[x]_\epsilon\neq[y]_\epsilon.
\end{equation}
For instance, one may take $C_d=2\sqrt{d}$.

We now check the hypotheses of Slepian's lemma.

\begin{enumerate}
	\item \textbf{Equal variances.}
	Define
	\begin{equation}\label{eqU: upper bound, a(x)}
		a(x)^2
		:=
		\frac{
			\mathbb{E}[X_{\delta\epsilon}^\alpha(\delta x)^2]
			-\theta(\epsilon)\mathbb{E}[(\psi_\epsilon^x)^2]
		}{
			\mathbb{E}[(\xi_\epsilon^{\alpha,[x]_\epsilon}(T+T_r))^2]
		}.
	\end{equation}
	For $T\geq T_*$, after increasing $T_*$ if necessary, the numerator is nonnegative. Hence, $a(x)$ is well defined and
	$\operatorname{Var}(X_{\delta\epsilon}^\alpha(\delta x))=\operatorname{Var}(Y_x)$. Moreover, by \eqref{eqM: MBBM}, \eqref{eqB: Brownian sheet variance}, and \eqref{eqK:alpha epsilon kernel, covariance structure},
	\begin{equation}\label{eqU: upper bound, a() <=1}
		a(x)^2
		\leq
		\frac{
			\displaystyle\int_0^{C_K+\log(1/(\delta\epsilon))}\sigma_\alpha(s)^2\,\td s
		}{
			\displaystyle\int_0^{T_r+\log(1/\epsilon)}\sigma_\alpha(s)^2\,\td s
		}
		\leq1,
	\end{equation}
	where the last inequality follows from \eqref{eqU: upper bound, coefficients for fields}. Thus, $0\leq a(x)\leq1$.
	
	\item \textbf{Comparison of local $L^2$ distances.}
	Suppose $[x]_\epsilon=[y]_\epsilon$. Then $\delta x$ and $\delta y$ lie in the same $(\delta\epsilon)$-box. Hence, by \eqref{eqK: alpha epsilon kernel, L2 distance} and the monotonicity $\theta(\delta\epsilon)\leq\theta(\epsilon)$,
	\[
	\mathbb{E}\bigl[(X_{\delta\epsilon}^\alpha(\delta x)
	-X_{\delta\epsilon}^\alpha(\delta y))^2\bigr]
	\leq\theta(\delta\epsilon)\frac{\lVert x-y\rVert_2}{\epsilon}
	\leq\theta(\epsilon)\frac{\lVert x-y\rVert_2}{\epsilon}.
	\]
	On the other hand, by \eqref{eqB: Brownian sheet L2 distance},
	\[
	\begin{aligned}
		\mathbb{E}[(Y_x-Y_y)^2]
		&\geq\theta(\epsilon)
		\mathbb{E}[(\psi_\epsilon^x-\psi_\epsilon^y)^2]\\
		&\geq\theta(\epsilon)\frac{\lVert x-y\rVert_1}{\epsilon}
		\geq\theta(\epsilon)\frac{\lVert x-y\rVert_2}{\epsilon}.
	\end{aligned}
	\]
	Thus,
	\[
	\mathbb{E}\bigl[(X_{\delta\epsilon}^\alpha(\delta x)
	-X_{\delta\epsilon}^\alpha(\delta y))^2\bigr]
	\leq\mathbb{E}[(Y_x-Y_y)^2].
	\]
	
	\item \textbf{Covariance comparison between distinct boxes.}
	Suppose $[x]_\epsilon\neq[y]_\epsilon$. Since the Brownian-sheet fields are independent across distinct $\epsilon$-boxes and $a(\cdot)\leq1$,
	\[
	\begin{aligned}
		\operatorname{Cov}(Y_x,Y_y)
		&=a(x)a(y)\operatorname{Cov}\Bigl(
		\xi_\epsilon^{\alpha,[x]_\epsilon}(T+T_r),
		\xi_\epsilon^{\alpha,[y]_\epsilon}(T+T_r)
		\Bigr)\\
		&\leq\operatorname{Cov}\Bigl(
		\xi_\epsilon^{\alpha,[x]_\epsilon}(T+T_r),
		\xi_\epsilon^{\alpha,[y]_\epsilon}(T+T_r)
		\Bigr)\\
		&\leq
		\int_0^{\log(1/\lVert[x]_\epsilon-[y]_\epsilon\rVert_\infty)+s_0}
		\sigma_\alpha(s)^2\,\td s,
	\end{aligned}
	\]
	where the last inequality follows from \eqref{eqMBBM: cov}. On the other hand, \eqref{eqU: grid distance dimensional constant} and \eqref{eqU: upper bound, coefficients for fields} imply
	\[
	\log\frac{1}{\lVert[x]_\epsilon-[y]_\epsilon\rVert_\infty}+s_0
	\leq-C_X+\log\frac{1}{\delta\lVert x-y\rVert_2}.
	\]
	Therefore, the covariance lower bound in \eqref{eqK:alpha epsilon kernel, covariance structure} gives
	\[
	\operatorname{Cov}(Y_x,Y_y)
	\leq
	\operatorname{Cov}\bigl(
	X_{\delta\epsilon}^\alpha(\delta x),
	X_{\delta\epsilon}^\alpha(\delta y)
	\bigr).
	\]
	Equivalently, using the established equality of variances,
	\[
	\mathbb{E}\bigl[(X_{\delta\epsilon}^\alpha(\delta x)
	-X_{\delta\epsilon}^\alpha(\delta y))^2\bigr]
	\leq\mathbb{E}[(Y_x-Y_y)^2].
	\]
\end{enumerate}

The variance identity and the two $L^2$ comparisons imply, by Slepian's lemma, that
\begin{equation}\label{eqU: Slepian comparison}
	\mathbb{P}\Bigl(
	\sup_{x\in[0,1)^d}X_{\delta\epsilon}^\alpha(\delta x)
	\geq A_\alpha(T)+\lambda
	\Bigr)
	\leq
	\mathbb{P}\Bigl(
	\sup_{x\in[0,1)^d}Y_x
	\geq A_\alpha(T)+\lambda
	\Bigr).
\end{equation}
To bound the probability on the right, write
\begin{equation}\label{eqU: P(sup Y_x ge ...)=}
\begin{aligned}
	\mathbb{P}\Bigl(
	\sup_{x\in[0,1)^d}Y_x\geq A_\alpha(T)+\lambda
	\Bigr)=
	\mathbb{P}\Bigl(
	\sup_{x\in[0,1)^d}
	\bigl[a(x)\xi_\epsilon^{\alpha,[x]_\epsilon}(T+T_r)
	+\sqrt{\theta(\epsilon)}\,\psi_\epsilon^x\bigr]
	\geq A_\alpha(T)+\lambda
	\Bigr).
\end{aligned}
\end{equation}

We decompose the grid $V_\epsilon=[0,1)^d\cap\epsilon\mathbb{Z}^d$ according to the level sets of $\sqrt{\theta(\epsilon)}\,\psi_\epsilon^x$ and discuss the maxima of $a(x)\xi_\epsilon^{\alpha,[x]_\epsilon}(T+T_r)$ on these level sets. For $z\in V_\epsilon$, define $\psi_\epsilon^{*,z}:=\sup_{x\in\Box_\epsilon^z}\psi_\epsilon^x$. For $n\geq1$, let
\begin{equation}\label{eqU: upper bound, Gamma n def}
	\begin{aligned}
		\Gamma_0
		&:=\bigl\{z\in V_\epsilon\colon
		\sqrt{\theta(\epsilon)}\,\psi_\epsilon^{*,z}\leq0\bigr\},\ 
		\Gamma_n:=\bigl\{z\in V_\epsilon\colon
		\sqrt{\theta(\epsilon)}\,\psi_\epsilon^{*,z}\in[n-1,n)\bigr\},
		\qquad n\in\mathbb{N}.
	\end{aligned}
\end{equation}
From \eqref{eqB: Brownian sheet max in a epsilon box}, for all $n\geq\lambda/2$ and all sufficiently large $\lambda$,
\begin{equation}\label{eqU: upper bound, E Gamma_n}
	\begin{aligned}
		\mathbb{E}\lvert\Gamma_n\rvert
		&\leq
		\sum_{z\in V_\epsilon}
		\mathbb{P}\Bigl(
		\sqrt{\theta(\epsilon)}\,\psi_\epsilon^{*,z}\geq n-1
		\Bigr)\\
		&\leq C\epsilon^{-d}
		\exp\Bigl\{-c\frac{n^2}{\theta(\epsilon)}\Bigr\}\leq C\exp\bigl\{dT-cT^{2\alpha}n^2\bigr\},
	\end{aligned}
\end{equation}
where the last line uses \eqref{eqU: upper bound, theta}.

For $n\geq1$, on the event $z\in\Gamma_n$,
$\sqrt{\theta(\epsilon)}\,\psi_\epsilon^{*,z}<n$, while for $z\in\Gamma_0$,
$\sqrt{\theta(\epsilon)}\,\psi_\epsilon^{*,z}\leq0$. Also, for some $C_0>0$,
$\lvert A_\alpha(T)-A_\alpha(T+T_r)\rvert\leq C_0$ uniformly for $T\geq0$.  

By \eqref{eqU: P(sup Y_x ge ...)=}, we have 
\[
\begin{aligned}
\mathbb{P}\Bigl(
\sup_{x\in[0,1)^d}Y_x\geq A_\alpha(T)+\lambda
\Bigr) \le \sum_{n\ge 0} \mathbb P \bigl(
\sup_{z\in \Gamma_n} \sup_{x\in \Box_{z}^\epsilon} a(x) \xi_\epsilon ^{\alpha,z} (T+T_r) \ge A_\alpha(T) + \lambda-n
\bigr).
\end{aligned}
\]
When $n\le n_T:= \lceil A_\alpha(T) +\lambda\rceil -1$, we have $A_\alpha(T) + \lambda-n\ge 0$. Together with $a(x)\in [0,1]$ for all $x\in [0,1)^d$ and 
\[
| A_\alpha(T) - A_\alpha(T+T_r) | \le C_0
\]
for some $C_0>0$ uniformly for $T\ge 0$, we have 
\[
\begin{aligned}
& \mathbb P \bigl(
\sup_{z\in \Gamma_n} \sup_{x\in \Box_{z}^\epsilon} a(x) \xi_\epsilon ^{\alpha,z} (T+T_r) \ge A_\alpha(T) + \lambda-n
\bigr)\\
\le &\mathbb{P}\Bigl(
\sup_{z\in\Gamma_n}\xi_\epsilon^{\alpha,z}(T+T_r)
\geq A_\alpha(T+T_r)+\lambda-n-C_0
\Bigr)
\end{aligned}
\]

Thus we have 
\begin{equation}\label{eqU: decomposition by Gamma n polished}
	\begin{aligned}
		&\mathbb{P}\Bigl(
		\sup_{x\in[0,1)^d}Y_x\geq A_\alpha(T)+\lambda
		\Bigr)
		\\
		&\quad\leq
		\sum_{0\leq n\leq\lfloor\lambda/2\rfloor}
		\mathbb{P}\Bigl(
		\sup_{z\in\Gamma_n}\xi_\epsilon^{\alpha,z}(T+T_r)
		\geq A_\alpha(T+T_r)+\lambda-n-C_0
		\Bigr)\\
		&\qquad+
		\sum_{\lambda/2<n\leq n_T}
		\mathbb{P}\Bigl(
		\sup_{z\in\Gamma_n}\xi_\epsilon^{\alpha,z}(T+T_r)
		\geq A_\alpha(T+T_r)+\lambda-n-C_0
		\Bigr)\\
		&\qquad+
		\sum_{n>n_T}
		\mathbb{P}(\Gamma_n\neq\varnothing).
	\end{aligned}
\end{equation}

For the first sum in \eqref{eqU: decomposition by Gamma n polished}, since $\Gamma_n$ is independent of the MBBM field, Proposition~\ref{proposition: MBBM exceeding particles num for a subset}, specifically \eqref{eqMBBM: exceeding particles num for a subset 1}, gives
\[
\begin{aligned}
	&\sum_{0\leq n\leq\lfloor\lambda/2\rfloor}
	\mathbb{P}\Bigl(
	\sup_{z\in\Gamma_n}\xi_\epsilon^{\alpha,z}(T+T_r)
	\geq A_\alpha(T+T_r)+\lambda-n-C_0
	\Bigr)\\
	&\qquad\leq C\lambda e^{-c\lambda}
	\leq Ce^{-c\lambda}.
\end{aligned}
\]
We next treat the intermediate range
$\lambda/2<n\leq n_T$. Using \eqref{eqMBBM: exceeding particles num for a subset 2} in Proposition~\ref{proposition: MBBM exceeding particles num for a subset} and \eqref{eqU: upper bound, E Gamma_n},
\[
\begin{aligned}
	&\sum_{\lambda/2<n\leq n_T}
	\mathbb{P}\Bigl(
	\sup_{z\in\Gamma_n}\xi_\epsilon^{\alpha,z}(T+T_r)
	\geq A_\alpha(T+T_r)+\lambda-n-C_0
	\Bigr)\\
	&\quad\leq
	\sum_{\lambda/2<n\leq n_T}
	\Bigl(
	Ce^{-c\lambda-cn}
	+C\mathbb{E}\lvert\Gamma_n\rvert
	\exp\bigl\{-d(T+T_r)-c\lambda+c(T+T_r)^\alpha(n+1)\bigr\}
	\Bigr)\\
	&\quad\leq
	\sum_{\lambda/2<n\leq n_T}
	\Bigl(
	Ce^{-c\lambda-cn}
	+C\exp\bigl\{-c\lambda+c(T+T_r)^\alpha(n+1)-c'T^{2\alpha}n^2\bigr\}
	\Bigr)\\
	&\quad\leq Ce^{-c\lambda}.
\end{aligned}
\]
Finally, for the large-$n$ range, Markov's inequality and \eqref{eqU: upper bound, E Gamma_n} give
\[
\begin{aligned}
	\sum_{n>n_T}
	\mathbb{P}(\Gamma_n\neq\varnothing)
	&\leq
	\sum_{n>n_T}
	\mathbb{E}\lvert\Gamma_n\rvert\\
	&\leq
	C\sum_{n>n_T}
	\exp\bigl\{dT-cT^{2\alpha}n^2\bigr\}.
\end{aligned}
\]
Since $A_\alpha(T+T_r)\asymp T^{1-\alpha}$, we have $T^{2\alpha}A_\alpha(t)^2\asymp T^2$. Thus, after increasing $T_*$ and $\lambda_*$ if necessary, for $T\geq T_*$ and $\lambda\geq\lambda_*$,
\[
\sum_{n>n_T}
\mathbb{P}(\Gamma_n\neq\varnothing)
\leq Ce^{-cT^2-c\lambda}.
\]
Combining these three bounds with \eqref{eqU: decomposition by Gamma n polished} and \eqref{eqU: Slepian comparison} proves \eqref{eqU: upper bound, delta rescale field, main result}, and hence Proposition~\ref{Prop: upper bound for X epsilon}.

\subsection{Lower tail}\label{subsec: X field, lower tail}

The goal of this subsection is to prove the following lower-tail estimate for the maximum of the Gaussian field $X_\epsilon^\alpha$. The essential tools are the Paley--Zygmund inequality and a second-moment estimate.

\begin{remark}\label{remark: lower bound sensitivity}
	We emphasize that, unlike the case in \cite{MZ2016_slowdownBBM}, where the variance profile is always of order $O(1)$, the fast decay of the variance profile $\sigma_\alpha(t)$ in \eqref{eqIntroduction: sigma(t)} requires particular care in choosing a suitable barrier; see \eqref{eqL: lower bound, good event}. Specifically, when $t$ is very close to $T$, $\sigma_\alpha(t)$ is of order $T^{-\alpha}$, so all ``particles'' move rather slowly. Any slight change in the barrier---see the rectified function $r$ in \eqref{eqL: lower bound, r function} and the target set $I$ in \eqref{eqL: lower bound, how to choose I corrected} at terminal time---may have an overwhelming influence on the exceedance probability. Moreover, a technical estimate such as Lemma~\ref{lem: killed BM weighted L2} is required.
\end{remark}

\begin{proposition}\label{Prop: max of X, lower tail}
	Fix $\alpha\in(0,1/2]$ and $d\geq2$, and set $T:=\log(1/\epsilon)$. Consider the Gaussian field $X_\epsilon^\alpha$ defined in \eqref{eqModel: X epsilon alpha white noise integral}. There exist constants $C,c>0$ and $T_*,\lambda_*>0$ such that
	\begin{equation}\label{eqL: lower bound, main result}
		\mathbb{P}\Bigl(
		\sup_{x\in[0,1)^d}X_\epsilon^\alpha(x)
		\leq A_\alpha(T)-\lambda
		\Bigr)
		\leq Ce^{-c\lambda},
		\qquad T\geq T_*,\quad\lambda\geq\lambda_*,
	\end{equation}
	where $A_\alpha(T)$ is defined in \eqref{eqIntroduction: A(t)}.
\end{proposition}

\begin{proof}
	We compare $X_\epsilon^\alpha$ with the auxiliary mollified BBM field introduced in \eqref{eqM: MBBM}, now with $s_0=0$. More precisely, define
	\begin{equation}\label{eqL: MBBM}
		\xi_\epsilon^{\alpha,z}(t)
		:=\int_{[0,t]\times\mathbb{R}^d}
		\sigma_\alpha(s)\,
		\mathds{1}_{\Lambda_1(e^sz)}(y)\,W(\td s,\td y),
		\qquad z\in V_\epsilon,\quad t\geq0,
	\end{equation}
	where $V_\epsilon:=[0,1)^d\cap\epsilon\mathbb{Z}^d$ and $\Lambda_1(z):=z+[-1/2,1/2]^d$. As in the derivation of \eqref{eqMBBM: cov}, there exists $C_\xi>0$ such that, for any $t\geq0$ and $z_1,z_2\in V_\epsilon$,
	\begin{equation}\label{eqL: lower bound, Gaussian fields, cov of MBBM}
		\operatorname{Cov}\bigl(
		\xi_\epsilon^{\alpha,z_1}(t),
		\xi_\epsilon^{\alpha,z_2}(t)
		\bigr)
		\geq
		\int_0^{t\wedge(t_{z_1,z_2}-C_\xi)_+}
		\sigma_\alpha(s)^2\,\td s,
	\end{equation}
	where $t_{z_1,z_2}:=\log(1/\lVert z_1-z_2\rVert_\infty)$ and $t_{z,z}:=+\infty$.
	
	Fix $\delta,\rho\in(0,1)$ such that
	\begin{equation}\label{eqL: lower bound two aux constants}
		-C_K+\log(1/\delta)\geq1,
		\qquad
		C_K+\log\rho\leq-C_\xi,
	\end{equation}
	where $C_K$ and $C_\xi$ are the constants in \eqref{eqK:alpha epsilon kernel, covariance structure} and \eqref{eqL: lower bound, Gaussian fields, cov of MBBM}, respectively. We claim that, to prove \eqref{eqL: lower bound, main result}, it suffices to show that there exist constants $C,c>0$, $\epsilon_0>0$, and $\lambda_*>0$ such that
	\begin{equation}\label{eqL: reduced lower bound}
		\mathbb{P}\Bigl(
		\max_{x\in V_\epsilon\cap\rho[0,1)^d}
		X_{\delta\epsilon}^\alpha(x/\rho)
		\leq A_\alpha\bigl(\log(1/(\delta\epsilon))\bigr)-\lambda
		\Bigr)
		\leq Ce^{-c\lambda}
	\end{equation}
	for all $0<\epsilon<\epsilon_0/\delta$ and $\lambda\geq\lambda_*$. Indeed, once this is proved, setting $\epsilon'=\delta\epsilon$ gives
	$\{x/\rho\colon x\in V_\epsilon\cap\rho[0,1)^d\}\subset[0,1)^d$, and therefore
	\[
	\begin{aligned}
		&\mathbb{P}\Bigl(
		\sup_{x\in[0,1)^d}X_{\epsilon'}^\alpha(x)
		\leq A_\alpha\bigl(\log(1/\epsilon')\bigr)-\lambda
		\Bigr)\\
		&\quad\leq
		\mathbb{P}\Bigl(
		\max_{x\in V_\epsilon\cap\rho[0,1)^d}
		X_{\delta\epsilon}^\alpha(x/\rho)
		\leq A_\alpha\bigl(\log(1/(\delta\epsilon))\bigr)-\lambda
		\Bigr)
		\leq Ce^{-c\lambda}.
	\end{aligned}
	\]
	This yields \eqref{eqL: lower bound, main result}.
	
	We first study the range $\lambda\geq(10+d)T$ directly for $X_{\delta\epsilon}^\alpha$. Fix $x_0\in V_\epsilon\cap\rho[0,1)^d$. By the covariance estimates for $X_{\delta\epsilon}^\alpha$, and since $\delta$ and $\rho$ are fixed,
	$\operatorname{Var}(X_{\delta\epsilon}^\alpha(x_0/\rho))\leq CT$. Moreover, by the expansion \eqref{eqIntroduction: Aalpha(T) expansion},
	$A_\alpha(T-\log\delta)\leq CT^{1-\alpha}$. Hence, for all sufficiently large $T$ and all $\lambda\geq(10+d)T$,
	$A_\alpha(T-\log\delta)-\lambda\leq-\lambda/2$. Therefore,
	\[
	\begin{aligned}
		&\mathbb{P}\Bigl(
		\max_{x\in V_\epsilon\cap\rho[0,1)^d}
		X_{\delta\epsilon}^\alpha(x/\rho)
		\leq A_\alpha(T-\log\delta)-\lambda
		\Bigr)\\
		&\quad\leq
		\mathbb{P}\Bigl(
		X_{\delta\epsilon}^\alpha(x_0/\rho)\leq-\lambda/2
		\Bigr)
		\leq\exp\Bigl\{-c\frac{\lambda^2}{T}\Bigr\}
		\leq e^{-c\lambda}.
	\end{aligned}
	\]
	
	From now on, consider only $\lambda\leq(10+d)T$. We compare
	\[
	\bigl\{X_{\delta\epsilon}^\alpha(x/\rho)\colon
	x\in V_\epsilon\cap\rho[0,1)^d\bigr\}
	\quad\text{and}\quad
	\bigl\{b(x)\xi_\epsilon^{\alpha,x}(T)\colon
	x\in V_\epsilon\cap\rho[0,1)^d\bigr\},
	\]
	where
	\begin{equation}\label{eqL: definition of bx}
		b(x)^2
		:=
		\frac{\operatorname{Var}(X_{\delta\epsilon}^\alpha(x/\rho))}
		{\operatorname{Var}(\xi_\epsilon^{\alpha,x}(T))}.
	\end{equation}
	By \eqref{eqK:alpha epsilon kernel, covariance structure}, \eqref{eqL: MBBM}, and the choice of $\delta$,
	\begin{equation}\label{eqL: bx geq one}
		b(x)^2
		\geq
		\frac{
			\displaystyle\int_0^{-C_K+\log(1/(\delta\epsilon))}
			\sigma_\alpha(s)^2\,\td s
		}{
			\displaystyle\int_0^{\log(1/\epsilon)}\sigma_\alpha(s)^2\,\td s
		}
		\geq1.
	\end{equation}
	Moreover, since $\delta$ and $\rho$ are fixed, the same covariance estimates imply
	\begin{equation}\label{eqL: bx close to one}
		1\leq b(x)\leq1+\frac{C}{T},
		\qquad x\in V_\epsilon\cap\rho[0,1)^d,
	\end{equation}
	for all sufficiently large $T$.
	
	For distinct $x,y\in V_\epsilon\cap\rho[0,1)^d$, the covariance upper bound for $X_{\delta\epsilon}^\alpha$ gives
	\[
	\operatorname{Cov}\bigl(
	X_{\delta\epsilon}^\alpha(x/\rho),
	X_{\delta\epsilon}^\alpha(y/\rho)
	\bigr)
	\leq
	\int_0^{(C_K+\log(\rho/\lVert x-y\rVert_\infty))_+}
	\sigma_\alpha(s)^2\,\td s.
	\]
	On the other hand, by \eqref{eqL: lower bound, Gaussian fields, cov of MBBM} and $b(x),b(y)\geq1$,
	\[
	\operatorname{Cov}\bigl(
	b(x)\xi_\epsilon^{\alpha,x}(T),
	b(y)\xi_\epsilon^{\alpha,y}(T)
	\bigr)
	\geq
	\int_0^{(\log(1/\lVert x-y\rVert_\infty)-C_\xi)_+}
	\sigma_\alpha(s)^2\,\td s.
	\]
	The choice of $\rho$ in \eqref{eqL: lower bound two aux constants} therefore implies
	\[
	\operatorname{Cov}\bigl(
	b(x)\xi_\epsilon^{\alpha,x}(T),
	b(y)\xi_\epsilon^{\alpha,y}(T)
	\bigr)
	\geq
	\operatorname{Cov}\bigl(
	X_{\delta\epsilon}^\alpha(x/\rho),
	X_{\delta\epsilon}^\alpha(y/\rho)
	\bigr).
	\]
	For $x=y$, the two fields have the same variance by definition of $b(x)$. Hence, Slepian's comparison theorem gives
	\[
	\begin{aligned}
		&\mathbb{P}\Bigl(
		\max_{x\in V_\epsilon\cap\rho[0,1)^d}
		X_{\delta\epsilon}^\alpha(x/\rho)
		\leq A_\alpha(T-\log\delta)-\lambda
		\Bigr)\\
		&\quad\leq
		\mathbb{P}\Bigl(
		\max_{x\in V_\epsilon\cap\rho[0,1)^d}
		b(x)\xi_\epsilon^{\alpha,x}(T)
		\leq A_\alpha(T-\log\delta)-\lambda
		\Bigr).
	\end{aligned}
	\]
	Using \eqref{eqL: bx close to one}, the expansion \eqref{eqIntroduction: Aalpha(T) expansion}, and the restriction $\lambda\leq(10+d)T$, we have, uniformly in $x\in V_\epsilon\cap\rho[0,1)^d$,
	\[
	\frac{A_\alpha(T-\log\delta)-\lambda}{b(x)}
	\leq A_\alpha(T)-\lambda+C,
	\]
	where $C$ is independent of $T$ and $\lambda$. Hence,
	\[
	\begin{aligned}
		\mathbb{P}\Bigl(
		\max_{x\in V_\epsilon\cap\rho[0,1)^d}
		b(x)\xi_\epsilon^{\alpha,x}(T)
		\leq A_\alpha(T-\log\delta)-\lambda
		\Bigr)\leq
		\mathbb{P}\Bigl(
		\max_{x\in V_\epsilon\cap\rho[0,1)^d}\xi_\epsilon^{\alpha,x}(T)
		\leq A_\alpha(T)-\lambda+C
		\Bigr).
	\end{aligned}
	\]
	Thus, it suffices to prove
	\begin{equation}\label{eqL: lower bound, Gaussian field, MBBM left tail}
		\mathbb{P}\Bigl(
		\max_{x\in V_\epsilon\cap\rho[0,1)^d}\xi_\epsilon^{\alpha,x}(T)
		\leq A_\alpha(T)-\lambda
		\Bigr)
		\leq Ce^{-c\lambda}.
	\end{equation}
	For notational simplicity, write $\xi_\epsilon^x$ for $\xi_\epsilon^{\alpha,x}$ in the rest of the proof.
	
	\medskip
	\noindent\textbf{Step 1: Decomposition at time $t_\lambda$.}
	Set
	\begin{equation}\label{eqL: t_lambda, K_lambda}
		t_\lambda:=\log\lambda\leq\log\bigl((10+d)T\bigr),
		\qquad
		K_\lambda:=t_\lambda^{-\alpha}.
	\end{equation}
	We first construct an auxiliary function $r(\cdot)$ that will be used to define the barrier. Let $\beta\in(3/4,1)$. Choose a nonnegative $C^2$ function $r\colon[0,T]\to\mathbb{R}_+$ such that $r=0$ on $[0,T-T^\beta]$ and, for some constants $a_0,a_1,a_2>0$,
	\begin{equation}\label{eqL: lower bound, r function}
		\begin{aligned}
			&\lvert r(t)\rvert
			\leq\frac{a_0c_r\sigma_\alpha(T)\log T}{T^\beta},
			\lvert r'(t)\rvert
			\leq\frac{a_1c_r\sigma_\alpha(T)\log T}{T^{2\beta}},
			\lvert r''(t)\rvert
			\leq\frac{a_2c_r\sigma_\alpha(T)\log T}{T^{3\beta}},
			 t\in[0,T],\\
			&a_2^{-1}c_r\log T
			\leq\int_0^T\frac{r(t)}{\sigma_\alpha(t)}\,\td t
			\leq a_2c_r\log T,
			\int_0^T r(t)\,\td t
			\leq a_2c_r(\log T)\sigma_\alpha(T).
		\end{aligned}
	\end{equation}
	Such a function is obtained by smoothing a piecewise quadratic-linear profile supported on $[T-T^\beta,T]$. Our auxiliary function is similar to that in Section~3.2 of \cite{MZ2016_slowdownBBM}; however, the support of the function therein is $[T-T^{2/3},T]$. In our case, the condition $\beta>3/4$ guarantees that $r$ satisfies the regularity assumptions required in Proposition~\ref{propoAiry: Airy estimate}.
	
	Define
	\begin{equation}\label{eqL: lower bound, kr definition}
		k_r(t):=k(t)-\frac{r(t)}{\sqrt{2d}},
		\qquad t\in[t_\lambda,T].
	\end{equation}
	By \eqref{eqL: lower bound, r function} and the definitions \eqref{eqIntroduction: sigma(t)} and \eqref{eqIntroduction: eta explicit} of $\sigma_\alpha$ and $\eta_\alpha$, respectively,
	\begin{equation}\label{eqGL: Gaussian field, lower tail, condition for eta + r/sigma}
		\int_0^T
		\biggl(\frac{\eta(s)+r(s)}{\sigma_\alpha(s)}\biggr)^2\,\td s
		\leq C,
	\end{equation}
	and
	\begin{equation}\label{eqL: repalce A_alpha(T)}
		\begin{aligned}
			A_\alpha(T)
			&=\int_0^T k(s)\,\td s\\
			&=\int_0^{t_\lambda}k(s)\,\td s
			+\int_{t_\lambda}^T k_r(s)\,\td s
			+\frac{1}{\sqrt{2d}}\int_0^T r(s)\,\td s\\
			&=\int_0^{t_\lambda}k(s)\,\td s
			+\int_{t_\lambda}^T k_r(s)\,\td s+o_T(1).
		\end{aligned}
	\end{equation}
	Consequently, by \eqref{eqL: repalce A_alpha(T)}, it is enough to prove \eqref{eqL: lower bound, Gaussian field, MBBM left tail} with $A_\alpha(T)$ replaced by
	$\int_0^{t_\lambda}k(s)\,\td s+\int_{t_\lambda}^T k_r(s)\,\td s$.
	
	Choose
	\[
	z_*\in\operatorname*{arg\,max}_{z\in V_\epsilon\cap\rho[0,1)^d}
	\bigl\{\xi_\epsilon^z(T)-\xi_\epsilon^z(t_\lambda)\bigr\}.
	\]
	Then
	\begin{equation}\label{eqL: lower bound, framework eq 1}
		\begin{aligned}
			&\mathbb{P}\Bigl(
			\max_{x\in V_\epsilon\cap\rho[0,1)^d}\xi_\epsilon^x(T)
			\leq
			\int_0^{t_\lambda}k(s)\,\td s
			+\int_{t_\lambda}^T k_r(s)\,\td s-\lambda
			\Bigr)\\
			&\leq
			\mathbb{P}\Bigl(
			\xi_\epsilon^{z_*}(t_\lambda)
			\leq\int_0^{t_\lambda}k(s)\,\td s-\lambda-K_\lambda
			\Bigr) +
			\mathbb{P}\Bigl(
			\xi_\epsilon^{z_*}(T)-\xi_\epsilon^{z_*}(t_\lambda)
			\leq\int_{t_\lambda}^T k_r(s)\,\td s+K_\lambda
			\Bigr).
		\end{aligned}
	\end{equation}
	Recall \eqref{eqL: MBBM}. The random point $z_*$ is measurable with respect to the white noise on $(t_\lambda,T]\times\mathbb{R}^d$, and hence is independent of the $\sigma$-field generated by $\{\xi_\epsilon^z(t_\lambda)\colon z\in V_\epsilon\cap\rho[0,1)^d\}$. Conditioning on $z_*$, the variable $\xi_\epsilon^{z_*}(t_\lambda)$ is centered Gaussian with variance $\int_0^{t_\lambda}\sigma_\alpha(s)^2\,\td s$. Since
	\[
	\operatorname{Var}\bigl(\xi_\epsilon^{z_*}(t_\lambda)\bigr)
	=\int_0^{t_\lambda}\sigma_\alpha(s)^2\,\td s
	\leq Ct_\lambda=C\log\lambda
	\]
	and $\int_0^{t_\lambda}k(s)\,\td s-\lambda-K_\lambda\leq-\lambda/2$ for all sufficiently large $\lambda$, the Gaussian tail estimate gives
	\begin{equation}\label{eqL: lower bound, framework eq 2}
		\mathbb{P}\Bigl(
		\xi_\epsilon^{z_*}(t_\lambda)
		\leq\int_0^{t_\lambda}k(s)\,\td s-\lambda-K_\lambda
		\Bigr)
		\leq C\exp\Bigl\{-c\frac{\lambda^2}{\log\lambda}\Bigr\}
		\leq Ce^{-c\lambda}.
	\end{equation}
	
	It remains to estimate
	\[
	\begin{aligned}
		&\mathbb{P}\Bigl(
		\xi_\epsilon^{z_*}(T)-\xi_\epsilon^{z_*}(t_\lambda)
		\leq\int_{t_\lambda}^T k_r(s)\,\td s+K_\lambda
		\Bigr)\\
		&\quad=
		\mathbb{P}\Bigl(
		\max_{z\in V_\epsilon\cap\rho[0,1)^d}
		\bigl(\xi_\epsilon^z(T)-\xi_\epsilon^z(t_\lambda)\bigr)
		\leq\int_{t_\lambda}^T k_r(s)\,\td s+K_\lambda
		\Bigr).
	\end{aligned}
	\]
	To exploit independence, choose disjoint $\epsilon\mathbb{Z}^d$-lattice boxes $E_1,\dots,E_{N_\lambda}\subset V_\epsilon\cap[0,\rho]^d$ that are translates of one another, have side length comparable to $1/\lambda$, and are separated by distance at least $1/\lambda=e^{-t_\lambda}$. Then
	\[
	N_\lambda\geq c\lambda^d,
	\qquad
	\lvert E_i\rvert\asymp(\lambda\epsilon)^{-d}
	=e^{d(T-t_\lambda)}.
	\]
	For definiteness, we may take $E_1=V_\epsilon\cap[0,1/\lambda]^d$. Since the sets $E_i$ are separated by distance at least $e^{-t_\lambda}$, the definition \eqref{eqL: MBBM} implies that the families
	\[
	\bigl\{\xi_\epsilon^z(T)-\xi_\epsilon^z(t_\lambda)\colon z\in E_i\bigr\},
	\qquad i=1,\dots,N_\lambda,
	\]
	are independent and identically distributed. Hence,
	\[
	\begin{aligned}
		&\mathbb{P}\Bigl(
		\max_{z\in V_\epsilon\cap\rho[0,1)^d}
		\bigl(\xi_\epsilon^z(T)-\xi_\epsilon^z(t_\lambda)\bigr)
		\leq\int_{t_\lambda}^T k_r(s)\,\td s+K_\lambda
		\Bigr)\\
		&\quad\leq
		\exp\Biggl\{
		-N_\lambda\mathbb{P}\Bigl(
		\max_{z\in E_1}
		\bigl(\xi_\epsilon^z(T)-\xi_\epsilon^z(t_\lambda)\bigr)
		\geq\int_{t_\lambda}^T k_r(s)\,\td s+K_\lambda
		\Bigr)
		\Biggr\}.
	\end{aligned}
	\]
	We prove in the next step that
	\begin{equation}\label{eqL: lower bound, final prob to estimate}
		\mathbb{P}\Bigl(
		\max_{z\in E_1}
		\bigl(\xi_\epsilon^z(T)-\xi_\epsilon^z(t_\lambda)\bigr)
		\geq\int_{t_\lambda}^T k_r(s)\,\td s+K_\lambda
		\Bigr)
		\geq ct_\lambda^{(\alpha-1)/2}.
	\end{equation}
	Once this is established, \eqref{eqL: t_lambda, K_lambda} gives
	\begin{equation}\label{eqL: lower bound, framework eq 3}
		\begin{aligned}
			&\mathbb{P}\Bigl(
			\max_{z\in V_\epsilon\cap\rho[0,1)^d}
			\bigl(\xi_\epsilon^z(T)-\xi_\epsilon^z(t_\lambda)\bigr)
			\leq\int_{t_\lambda}^T k_r(s)\,\td s+K_\lambda
			\Bigr)\\
			&\quad\leq
			\exp\bigl\{-c\lambda^d(\log\lambda)^{(\alpha-1)/2}\bigr\}
			\leq e^{-c\lambda},
		\end{aligned}
	\end{equation}
	where the last inequality uses $d\geq2$. Combining \eqref{eqL: lower bound, framework eq 1}, \eqref{eqL: lower bound, framework eq 2}, and \eqref{eqL: lower bound, framework eq 3} proves \eqref{eqL: lower bound, Gaussian field, MBBM left tail}.
	
	\medskip
	\noindent\textbf{Step 2: Second moment and barrier estimate.}
	We now prove \eqref{eqL: lower bound, final prob to estimate}. Set
	\begin{equation}\label{eqL: lower bound, V function}
		V(t)
		:=\frac{2\sqrt{2d}}{\sigma_\alpha(t)^2}
		\biggl(\frac{1}{\sigma_\alpha}\biggr)'(t)
		\asymp(1+2t)^{3\alpha-1},
		\qquad t\geq t_\lambda,
	\end{equation}
	as in Proposition~\ref{propoAiry: Airy estimate}, and write $a(t):=k_r(t)/\sigma_\alpha(t)^2$. By \eqref{eqL: lower bound, kr definition} and \eqref{eqIntroduction: sigma(t)}, uniformly for $t\geq t_\lambda$,
	\begin{equation}\label{eqL: lower bound, a and V estimates}
		a(t)\asymp(1+2t)^\alpha,
		\qquad
		V(t)\asymp(1+2t)^{3\alpha-1}.
	\end{equation}
	
	For any $z\in E_1$, define the good event
	\begin{equation}\label{eqL: lower bound, good event}
		\begin{aligned}
			\mathcal{E}_z:=\Biggl\{
			&\xi_\epsilon^z(t)-\xi_\epsilon^z(t_\lambda)
			\leq\int_{t_\lambda}^t k_r(s)\,\td s+2K_\lambda,
			\qquad t\in[t_\lambda,T],\\
			&\xi_\epsilon^z(T)-\xi_\epsilon^z(t_\lambda)
			\in\int_{t_\lambda}^T k_r(s)\,\td s+2K_\lambda-I
			\Biggr\},
		\end{aligned}
	\end{equation}
	where $I\subset[0,K_\lambda\wedge V(T)^{-1/3}]$ will be chosen in \eqref{eqL: lower bound, how to choose I corrected}. The restriction $I\subset[0,V(T)^{-1/3}]$ allows us to apply the lower bound \eqref{eqAiry: lower bound} at terminal time. Since
	\[
	\frac{2K_\lambda}{V(t_\lambda)^{-1/3}}
	\asymp t_\lambda^{-1/3}\longrightarrow0,
	\]
	we have $2K_\lambda\leq V(t_\lambda)^{-1/3}$ for all sufficiently large $\lambda$. Thus, the initial point $2K_\lambda$ is also in the Airy boundary regime.
	
	Let $N_{\mathrm{good}}:=\sum_{z\in E_1}\mathds{1}_{\mathcal{E}_z}$. By the Paley--Zygmund inequality,
	\begin{equation}\label{eqL: lower bound, step 2, P-Z inequality}
		\mathbb{P}(N_{\mathrm{good}}\geq1)
		\geq
		\frac{\mathbb{E}[N_{\mathrm{good}}]^2}
		{\mathbb{E}[N_{\mathrm{good}}^2]}.
	\end{equation}
	Furthermore, since $I\subset[0,K_\lambda]$,
	\begin{equation}\label{eqL: lower bound, N good implies}
		\{N_{\mathrm{good}}\geq1\}
		\subset
		\Biggl\{
		\max_{z\in E_1}\bigl(\xi_\epsilon^z(T)-\xi_\epsilon^z(t_\lambda)\bigr)
		\geq\int_{t_\lambda}^T k_r(s)\,\td s+K_\lambda
		\Biggr\}.
	\end{equation}
	Thus, it suffices to estimate the ratio in \eqref{eqL: lower bound, step 2, P-Z inequality}.
	
	For $t_\lambda\leq t_1<t_2\leq T$, define $A(x,y;t_1,t_2)\,\td y$ by
	\[
	\begin{aligned}
		A(x,y;t_1,t_2)\,\td y
		:=\mathbb{P}\Biggl(
		&\xi_\epsilon^z(t_2)-\xi_\epsilon^z(t_\lambda)
		\in\int_{t_\lambda}^{t_2}k_r(s)\,\td s+2K_\lambda-\td y,\\
		&\xi_\epsilon^z(t)-\xi_\epsilon^z(t_\lambda)
		\leq\int_{t_\lambda}^t k_r(s)\,\td s+2K_\lambda,
		\quad t\in[t_1,t_2]\\
		&\Bigm|
		\xi_\epsilon^z(t_1)-\xi_\epsilon^z(t_\lambda)
		=\int_{t_\lambda}^{t_1}k_r(s)\,\td s+2K_\lambda-x
		\Biggr).
	\end{aligned}
	\]
	This kernel is independent of $z$ by translation invariance. By Lemma~\ref{Appendix: Bridge like probability} and \eqref{eqGL: Gaussian field, lower tail, condition for eta + r/sigma},
	\begin{equation}\label{eqL: lower bound, Az corrected}
		A(x,y;t_1,t_2)
		\asymp
		e^{-d(t_2-t_1)}e^{-a(t_1)x+a(t_2)y}
		G_r(x,y;t_1,t_2),
	\end{equation}
	where $G_r$ is the fundamental solution of the Airy-type PDE \eqref{eqA: Airy PDE with sigma diffusion} with the present choice of $r$.
	
	We first estimate the first moment. Using \eqref{eqL: lower bound, Az corrected} and the lower bound \eqref{eqAiry: lower bound} in Proposition~\ref{propoAiry: Airy estimate},
	\begin{equation}\label{eqL: lower bound, first moment corrected}
		\begin{aligned}
			\mathbb{E}[N_{\mathrm{good}}]
			&=\sum_{z\in E_1}\mathbb{P}(\mathcal{E}_z)\\
			&\geq
			ce^{-2a(t_\lambda)K_\lambda}
			K_\lambda\sqrt{V(t_\lambda)}\sqrt{V(T)}
			\exp\Biggl\{\int_{t_\lambda}^T
			\frac{r(s)}{\sigma_\alpha(s)}\,\td s\Biggr\}
			\int_I e^{a(T)y}y\,\td y.
		\end{aligned}
	\end{equation}
	
	We now choose $I$ appearing in \eqref{eqL: lower bound, good event} and \eqref{eqL: lower bound, first moment corrected} above. Let $I:= [0,\ell_T]$, for some $\ell_T>0$ such that
	\begin{equation}\label{eqL: lower bound, how to choose I corrected}
		\sqrt{V(T)}
		\exp\Biggl\{\int_{t_\lambda}^T
		\frac{r(s)}{\sigma_\alpha(s)}\,\td s\Biggr\}
		\int_0^{\ell_T}e^{a(T)y}y\,\td y
		=1,
	\end{equation}
	By \eqref{eqL: lower bound, r function}, \eqref{eqL: lower bound, V function}, for all sufficiently large $T$,
	\[
	\begin{aligned}
		1=	\sqrt{V(T)}
		\exp\Biggl\{\int_{t_\lambda}^T
		\frac{r(s)}{\sigma_\alpha(s)}\,\td s\Biggr\}
		\int_0^{\ell_T}e^{a(T)y}y\,\td y \ge C T^{(3\alpha -1)/2 + c_r/a_2} \ell _T^2
	\end{aligned}
	\]
	Thus $\ell_T \le C^{-1} T^{ -(3\alpha-1)/4 -c_r /(2a_2)}$. Especially, when we take $c_r$ to be large, we can ensure the condition $I \subset [0,K_\lambda \wedge V(T)^{-1/3}]$ in \eqref{eqL: lower bound, good event}, and also 
	\begin{equation}\label{eqL: e^a(T) q <C}
		e^{a(T)y}\leq C,
		\qquad y\in[0,\ell_T],
	\end{equation}
	which will be used later.
	
	With this choice of $I$, \eqref{eqL: lower bound, first moment corrected}, \eqref{eqL: lower bound, V function}, and $K_\lambda=t_\lambda^{-\alpha}$ give
	\begin{equation}\label{eqL: E N_good >}
		\mathbb{E}[N_{\mathrm{good}}]
		\geq
		cK_\lambda\sqrt{V(t_\lambda)}e^{-2a(t_\lambda)K_\lambda}
		\geq ct_\lambda^{(\alpha-1)/2}.
	\end{equation}
	
	We now estimate the second moment. Write
	\begin{equation}\label{eqL: E N_good ^2 =}
		\mathbb{E}[N_{\mathrm{good}}^2]
		=
		\mathbb{E}[N_{\mathrm{good}}]
		+
		\sum_{\substack{z_1,z_2\in E_1\\z_1\neq z_2}}
		\mathbb{P}(\mathcal{E}_{z_1}\cap\mathcal{E}_{z_2}).
	\end{equation}

	For distinct $z_1,z_2\in E_1$, write
	\[
	h_{12}:
	=\log\frac{1}{\lVert z_1-z_2\rVert_\infty}.
	\]
	For every $s\in[h_{12},T)$, the noises driving
	$\xi_\cdot^{z_1}$ and $\xi_\cdot^{z_2}$ on $[s,T]$ are independent;
	see \eqref{eqL: MBBM}. Set
	\[
	H_s(y):=\int_I A(y,z;s,T)\,\td z,
	\qquad y>0.
	\]
We claim that
	\begin{equation}\label{eqL: P(E z1 cap E z2)}
		\mathbb{P}(\mathcal{E}_{z_1}\cap\mathcal{E}_{z_2})
		\leq
		\int_0^\infty
		A(2K_\lambda,y;t_\lambda,s)H_s(y)^2\,\td y,
		\qquad s\in[h_{12},T).
	\end{equation}
	Indeed, if $\mathcal{B}_{z_i}(s)$ denotes the event that the path corresponding to $z_i$ stays below the barrier up to time $s$, and $Y_{z_i}(s)$ denotes its distance from the barrier at time $s$, then the left-hand side of \eqref{eqL: P(E z1 cap E z2)} is bounded by
	\[
	\mathbb{E}\bigl[
	\mathds{1}_{\mathcal{B}_{z_1}(s)}
	\mathds{1}_{\mathcal{B}_{z_2}(s)}
	H_s(Y_{z_1}(s))H_s(Y_{z_2}(s))
	\bigr].
	\]
	Using Cauchy--Schwarz inequality and the fact that $\xi^{z_1}$ and $\xi^{z_2}$ have the same law and independent increment after $h_{1,2}\le s$, we can further bound it by 
	\[
	\mathbb{E}\bigl[
	\mathds{1}_{\mathcal{B}_{z_1}(s)}
	H_s(Y_{z_1}(s))^2
	\bigr],
	\]
	which is precisely the right-hand side of
	\eqref{eqL: P(E z1 cap E z2)}.
	
	We next organize the pairs into unit splitting-time shells. First, consider the pairs with
	\[
	h_{12}\geq T-2.
	\]
	For each fixed $z_1\in E_1$, there are at most $C$ such points $z_2$, because
	\[
	h_{12}\geq T-2
	\quad\Longrightarrow\quad
	\lVert z_1-z_2\rVert_\infty\leq e^2\epsilon.
	\]
	Consequently,
	\begin{equation}\label{eqL: terminal splitting pairs}
		\sum_{\substack{z_1,z_2\in E_1,\ z_1\neq z_2\\
				h_{12}\geq T-2}}
		\mathbb{P}(\mathcal{E}_{z_1}\cap\mathcal{E}_{z_2})
		\leq
		C\sum_{z_1\in E_1}\mathbb{P}(\mathcal{E}_{z_1})
		=
		C\mathbb{E}[N_{\mathrm{good}}].
	\end{equation}
	
	For the remaining pairs, set
	\[
	s_j:=t_\lambda+j,
	\qquad
	\mathcal{J}:=\{j\in\mathbb{Z}_{\geq0}:s_j<T-2\},
	\]
	and define the set of ordered pairs
	\[
	\mathcal{P}_j
	:=
	\Bigl\{
	(z_1,z_2)\in E_1^2:
	z_1\neq z_2,\ 
	s_j\leq h_{12}<s_j+1,\ 
	h_{12}<T-2
	\Bigr\}.
	\]
	Since $\epsilon=e^{-T}$ and
	$\lvert E_1\rvert\leq Ce^{d(T-t_\lambda)}$, a lattice-ball count gives
	\begin{equation}\label{eqL: splitting shell count}
		\begin{aligned}
			\lvert\mathcal{P}_j\rvert
			&\leq
			C\lvert E_1\rvert
			\biggl(1+\frac{e^{-s_j}}{\epsilon}\biggr)^d \leq
			Ce^{d(2T-t_\lambda-s_j)},
			\qquad j\in\mathcal{J}.
		\end{aligned}
	\end{equation}
	Here the last inequality uses $s_j<T-2$, so that
	$e^{-s_j}/\epsilon=e^{T-s_j}\geq e^2$.
	
	For $(z_1,z_2)\in\mathcal{P}_j$ and every
	$s\in[s_j+1,s_j+2]$, we have $s\geq h_{12}$ and $s<T$.
	Thus, by \eqref{eqL: P(E z1 cap E z2)}, we have
	\[
	\mathbb{P}(\mathcal{E}_{z_1}\cap\mathcal{E}_{z_2})
	\leq
	\int_{s_j+1}^{s_j+2}\int_0^\infty
	A(2K_\lambda,y;t_\lambda,g)
	\biggl[\int_I A(y,z;s,T)\,\td z\biggr]^2
	\,\td y\,\td s.
	\]

	By \eqref{eqL: lower bound, Az corrected},
	\[
	\mathbb{P}(\mathcal{E}_{z_1}\cap\mathcal{E}_{z_2})
	\leq
	Ce^{-2a(t_\lambda)K_\lambda}
	\int_{s_j+1}^{s_j+2}
	e^{-d(2T-t_\lambda-s)}\Phi(s)\,\td s
	\]
	with 
	
	\[
	\Phi(s)
	:=
	\int_0^\infty
	e^{-a(s)y}G_r(2K_\lambda,y;t_\lambda,s)
	\biggl[
	\int_I e^{a(T)z}G_r(y,z;s,T)\,\td z
	\biggr]^2
	\,\td y.
	\]
	Combining this with \eqref{eqL: splitting shell count}, we obtain
	\[
	\begin{aligned}
		\sum_{(z_1,z_2)\in\mathcal{P}_j}
		\mathbb{P}(\mathcal{E}_{z_1}\cap\mathcal{E}_{z_2})
		&\leq
		Ce^{-2a(t_\lambda)K_\lambda}
		\int_{s_j+1}^{s_j+2}
		e^{d(s-s_j)}\Phi(s)\,\td s\\
		&\leq
		Ce^{-2a(t_\lambda)K_\lambda}
		\int_{s_j+1}^{s_j+2}\Phi(s)\,\td s.
	\end{aligned}
	\]
	Summing over $j$, and then combining the result with
	\eqref{eqL: E N_good ^2 =} and
	\eqref{eqL: terminal splitting pairs}, yields
	\begin{equation}\label{eqL: lower bound second moment representation corrected}
		\begin{aligned}
			\mathbb{E}[N_{\mathrm{good}}^2]
			&\leq
			C\mathbb{E}[N_{\mathrm{good}}]\\
			&\quad+
			Ce^{-2a(t_\lambda)K_\lambda}
			\int_{t_\lambda}^T\int_0^\infty
			e^{-a(h)y}G_r(2K_\lambda,y;t_\lambda,h)
			\biggl[
			\int_I e^{a(T)z}G_r(y,z;h,T)\,\td z
			\biggr]^2
			\,\td y\,\td h.
		\end{aligned}
	\end{equation}
	
	We split the $h$-integral according to the time scale required by the Airy estimate. Choose $C_0>0$ larger than the constants $L_*$ in Proposition~\ref{propoAiry: Airy estimate}, and set
	\[
	\Delta_\lambda:=C_0(1+2t_\lambda)^{2/3},
	\qquad
	\Delta_T:=C_0(1+2T)^{2/3}.
	\]
	Then
	\[
	[t_\lambda,T]
	=[t_\lambda,t_\lambda+\Delta_\lambda]
	\cup[t_\lambda+\Delta_\lambda,T-\Delta_T]
	\cup[T-\Delta_T,T].
	\]
	Denote the corresponding contributions to the second term in \eqref{eqL: lower bound second moment representation corrected} by $\mathrm{I}$, $\mathrm{II}$, and $\mathrm{III}$. They are estimated separately in \eqref{eqL: lower bound second moment initial}, \eqref{eqL: lower bound second moment middle}, and \eqref{eqL: lower bound second moment terminal final}. We then obtain
	\begin{equation}\label{eqL: E N_good^2 <}
		\mathbb{E}[N_{\mathrm{good}}^2]
		\leq
		C\mathbb{E}[N_{\mathrm{good}}]
		+CK_\lambda e^{-2a(t_\lambda)K_\lambda}
		\bigl(t_\lambda^{2\alpha-1}
		+t_\lambda^{(3\alpha-1)/2}o_T(1)\bigr).
	\end{equation}
	By \eqref{eqL: lower bound, step 2, P-Z inequality}, \eqref{eqL: E N_good^2 <}, and \eqref{eqL: E N_good >},
	\[
	\mathbb{P}(N_{\mathrm{good}}\geq1)
	\geq Ct_\lambda^{(\alpha-1)/2}
	\]
	when $T>0$ and $\lambda>0$ are sufficiently large. Together with \eqref{eqL: lower bound, N good implies}, this proves \eqref{eqL: lower bound, final prob to estimate} and hence Proposition~\ref{Prop: max of X, lower tail}.
\end{proof}

\noindent\textbf{Estimates for integrals $\mathrm{I}$, $\mathrm{II}$, and $\mathrm{III}$.}
We shall repeatedly use the killed Brownian-motion kernel
\begin{equation}\label{eqL: p_u^D}
	p_u^D(x,y)
	:=\frac{1}{\sqrt{2\pi u}}
	\biggl[
	\exp\Bigl\{-\frac{(x-y)^2}{2u}\Bigr\}
	-\exp\Bigl\{-\frac{(x+y)^2}{2u}\Bigr\}
	\biggr],
	\qquad x,y>0,\quad u>0.
\end{equation}
For convenience, we collect the relevant notations and some asymptotics:

\[ \begin{aligned} &\lambda\in[\lambda_*,(10+d)T],\quad T\ge T_*, \qquad t_\lambda:=\log\lambda,\quad K_\lambda:=t_\lambda^{-\alpha},\\ &\Delta_\lambda:=C_0(1+2t_\lambda)^{2/3},\quad \Delta_T:=C_0(1+2T)^{2/3},\qquad \sigma_\alpha(t):=\sqrt{2}\,(1+2t)^{-\alpha},\\ &\eta_\alpha(t):=c_\eta(1+2t)^{-\alpha-2/3},\qquad a(t)\asymp(1+2t)^\alpha,\quad V(t)\asymp(1+2t)^{3\alpha-1}\quad (t\ge t_\lambda). \end{aligned} \]

Also recall the auxiliary function $r(\cdot)$ satisfying \eqref{eqL: lower bound, r function}. In particular, by increasing $c_r$, we can ensure that
\[
\int_{t_\lambda}^T\frac{r(s)}{\sigma_\alpha(s)}\,\td s
\geq a_2^{-1}c_r\log T
\]
is arbitrarily large. In what follows, we use these definitions and asymptotics without further comment.

We first consider the initial interval $h\in[t_\lambda,t_\lambda+\Delta_\lambda]$. By definition,
\begin{equation}\label{eqL: lower bound second moment initial expression}
	\begin{aligned}
		\mathrm{I}
		:=Ce^{-2a(t_\lambda)K_\lambda}
		\int_{t_\lambda}^{t_\lambda+\Delta_\lambda}\int_0^\infty
		&e^{-a(h)y}G_r(2K_\lambda,y;t_\lambda,h)\\[-0.2ex]
		&\times
		\biggl[\int_I e^{a(T)z}G_r(y,z;h,T)\,\td z\biggr]^2
		\,\td y\,\td h.
	\end{aligned}
\end{equation}
On this interval, $r=0$, and $a(h)\asymp a(t_\lambda)$, $V(h)\asymp V(t_\lambda)$, and $\sigma_\alpha(h)\asymp\sigma_\alpha(t_\lambda)$. Since $[h,T]$ is long enough for the Airy estimate, the upper bound in Proposition~\ref{propoAiry: Airy estimate}, together with the choice of $I$ in \eqref{eqL: lower bound, how to choose I corrected}, gives
\[
\int_I e^{a(T)z}G_r(y,z;h,T)\,\td z
\leq Cy\sqrt{V(h)}
\leq Cy\sqrt{V(t_\lambda)}.
\]
For the short interval $[t_\lambda,h]$,
\[
G_r(2K_\lambda,y;t_\lambda,h)
\le 
p^D_{\int_{t_\lambda}^h\sigma_\alpha(s)^2\,\td s}(2K_\lambda,y)
\exp\Biggl\{\int_{t_\lambda}^h
\frac{\eta_\alpha(s)+r(s)}{\sigma_\alpha(s)}\,\td s\Biggr\}.
\]
Since
\[
\int_{t_\lambda}^h
\frac{\eta_\alpha(s)+r(s)}{\sigma_\alpha(s)}\,\td s
\leq C\int_{t_\lambda}^{t_\lambda+\Delta_\lambda}
(1+2s)^{-2/3}\,\td s
\leq C,
\]
we have
\[
G_r(2K_\lambda,y;t_\lambda,h)
\leq
Cp^D_{\int_{t_\lambda}^h\sigma_\alpha(s)^2\,\td s}(2K_\lambda,y).
\]
Substituting these two bounds into \eqref{eqL: lower bound second moment initial expression}, we obtain
\[
\begin{aligned}
	\mathrm{I}
	&\leq
	Ce^{-2a(t_\lambda)K_\lambda}V(t_\lambda)
	\int_{t_\lambda}^{t_\lambda+\Delta_\lambda}\int_0^\infty
	e^{-a(t_\lambda)y}y^2
	p^D_{\int_{t_\lambda}^h\sigma_\alpha(s)^2\,\td s}(2K_\lambda,y)
	\,\td y\,\td h\\
	&\leq
	Ce^{-2a(t_\lambda)K_\lambda}
	\frac{V(t_\lambda)}{\sigma_\alpha(t_\lambda)^2}
	\int_0^\infty e^{-a(t_\lambda)y}y^2
	\biggl(\int_0^\infty p_u^D(2K_\lambda,y)\,\td u\biggr)\,\td y.
\end{aligned}
\]
The killed Brownian resolvent satisfies
\[
\int_0^\infty p_u^D(x,y)\,\td u
\leq C(x\wedge y),
\qquad x,y>0.
\]
Therefore,
\[
\begin{aligned}
	\mathrm{I}
	&\leq
	Ce^{-2a(t_\lambda)K_\lambda}
	\frac{V(t_\lambda)}{\sigma_\alpha(t_\lambda)^2}
	\int_0^\infty e^{-a(t_\lambda)y}y^2(2K_\lambda\wedge y)\,\td y\\
	&\leq
	Ce^{-2a(t_\lambda)K_\lambda}
	\frac{V(t_\lambda)}{\sigma_\alpha(t_\lambda)^2}
	a(t_\lambda)^{-4}.
\end{aligned}
\]
Since
\[
a(t_\lambda)\asymp t_\lambda^\alpha,
\qquad
V(t_\lambda)\asymp t_\lambda^{3\alpha-1},
\qquad
\sigma_\alpha(t_\lambda)^2\asymp t_\lambda^{-2\alpha},
\qquad
K_\lambda=t_\lambda^{-\alpha},
\]
we conclude that
\begin{equation}\label{eqL: lower bound second moment initial}
	\mathrm{I}\leq CK_\lambda t_\lambda^{2\alpha-1}.
\end{equation}

We next consider the middle interval $h\in[t_\lambda+\Delta_\lambda,T-\Delta_T]$. Here both $[t_\lambda,h]$ and $[h,T]$ are long enough for the Airy estimate. Thus,
\[
G_r(2K_\lambda,y;t_\lambda,h)
\leq
CK_\lambda y\sqrt{V(t_\lambda)V(h)}
\exp\Biggl\{\int_{t_\lambda}^h
\frac{r(s)}{\sigma_\alpha(s)}\,\td s\Biggr\},
\]
and
\[
G_r(y,z;h,T)
\leq
Cyz\sqrt{V(h)V(T)}
\exp\Biggl\{\int_h^T
\frac{r(s)}{\sigma_\alpha(s)}\,\td s\Biggr\}.
\]
Using the choice of $I$ in \eqref{eqL: lower bound, how to choose I corrected},
\[
\begin{aligned}
	\int_I e^{a(T)z}G_r(y,z;h,T)\,\td z
	&\leq
	Cy\sqrt{V(h)V(T)}
	\exp\Biggl\{\int_h^T
	\frac{r(s)}{\sigma_\alpha(s)}\,\td s\Biggr\}
	\int_I e^{a(T)z}z\,\td z\\
	&=
	Cy\sqrt{V(h)}
	\exp\Biggl\{-\int_{t_\lambda}^h
	\frac{r(s)}{\sigma_\alpha(s)}\,\td s\Biggr\}.
\end{aligned}
\]
Since $r\geq0$, the last exponential is bounded by $1$. Consequently,
\begin{equation}\label{eqL: lower bound second moment middle}
	\begin{aligned}
		\mathrm{II}
		&:=Ce^{-2a(t_\lambda)K_\lambda}
		\int_{t_\lambda+\Delta_\lambda}^{T-\Delta_T}\int_0^\infty
		e^{-a(h)y}G_r(2K_\lambda,y;t_\lambda,h)
		\biggl[\int_I e^{a(T)z}G_r(y,z;h,T)\,\td z\biggr]^2
		\,\td y\,\td h\\
		&\leq
		CK_\lambda\sqrt{V(t_\lambda)}
		\int_{t_\lambda+\Delta_\lambda}^{T-\Delta_T}
		V(h)^{3/2}
		\biggl(\int_0^\infty e^{-a(h)y}y^3\,\td y\biggr)\,\td h\\
		&\leq
		CK_\lambda\sqrt{V(t_\lambda)}
		\int_{t_\lambda}^T V(h)^{3/2}a(h)^{-4}\,\td h \leq
		CK_\lambda\sqrt{V(t_\lambda)}
		\int_{t_\lambda}^T h^{(\alpha-3)/2}\,\td h \leq CK_\lambda t_\lambda^{2\alpha-1}.
	\end{aligned}
\end{equation}

It remains to control the terminal interval $h\in[T-\Delta_T,T]$. Its contribution is
\begin{equation}\label{eqL: lower bound second moment terminal expression}
	\begin{aligned}
		\mathrm{III}
		:=Ce^{-2a(t_\lambda)K_\lambda}
		\int_{T-\Delta_T}^T\int_0^\infty
		&e^{-a(h)y}G_r(2K_\lambda,y;t_\lambda,h)\\[-0.2ex]
		&\times
		\biggl[\int_I e^{a(T)z}G_r(y,z;h,T)\,\td z\biggr]^2
		\,\td y\,\td h.
	\end{aligned}
\end{equation}
In this range, $[h,T]$ is too short for the Airy estimate to be applied to $G_r(y,z;h,T)$. We use the Airy estimate only on $[t_\lambda,h]$ and the killed Brownian-motion kernel on $[h,T]$.

The Airy upper bound on $[t_\lambda,h]$ gives
\[
G_r(2K_\lambda,y;t_\lambda,h)
\leq
CK_\lambda y\sqrt{V(t_\lambda)V(h)}
\exp\Biggl\{\int_{t_\lambda}^h
\frac{r(s)}{\sigma_\alpha(s)}\,\td s\Biggr\}.
\]
Moreover, by \eqref{eqIntroduction: sigma(t)} and \eqref{eqIntroduction: eta(t)},
\[
\int_h^T\frac{\eta_\alpha(s)}{\sigma_\alpha(s)}\,\td s
\leq C,
\qquad h\in[T-\Delta_T,T].
\]
The killed Brownian-motion estimate \eqref{eqA: killed BB ker estimate} therefore gives
\[
G_r(y,z;h,T)
\leq
C\exp\Biggl\{\int_h^T
\frac{r(s)}{\sigma_\alpha(s)}\,\td s\Biggr\}
p^D_{\int_h^T\sigma_\alpha(s)^2\,\td s}(y,z).
\]

Set $M_T:=\int_I e^{a(T)z}z\,\td z$. By the choice of $I$ in \eqref{eqL: lower bound, how to choose I corrected},
\begin{equation}\label{eqL: terminal MT value revised}
	M_T
	=
	\frac{
		\displaystyle\exp\Biggl\{-\int_{t_\lambda}^T
		\frac{r(s)}{\sigma_\alpha(s)}\,\td s\Biggr\}
	}{\sqrt{V(T)}}.
\end{equation}
Moreover, since $a(T)\ell_T\leq C$ and $I=[0,\ell_T]$,
$\sup_{z\in I}e^{a(T)z}\leq C$.

We claim that, uniformly in $h\in[T-\Delta_T,T]$,
\begin{equation}\label{eqL: ineq used for III}
	\int_0^\infty y
	\biggl(
	\int_I e^{a(T)z}
	p^D_{\int_h^T\sigma_\alpha(s)^2\,\td s}(y,z)\,\td z
	\biggr)^2\,\td y
	\leq
	C\biggl(
	M_T\wedge
	\frac{M_T^2}{\int_h^T\sigma_\alpha(s)^2\,\td s}
	\biggr).
\end{equation}
Indeed, this follows from Lemma~\ref{lem: killed BM weighted L2} below, applied to
$w(z)=e^{a(T)z}\mathds{1}_I(z)$. By \eqref{eqL: e^a(T) q <C},
$\lVert w\rVert_\infty\leq C$, and, since $I=[0,\ell_T]$,
\[
M_T=\int_0^{\ell_T}e^{a(T)z}z\,\td z
\geq\frac{\ell_T^2}{2}.
\]

Substituting the Airy bound for $G_r(2K_\lambda,y;t_\lambda,h)$, the killed Brownian comparison for $G_r(y,z;h,T)$, and \eqref{eqL: ineq used for III} into \eqref{eqL: lower bound second moment terminal expression}, and using $e^{-a(h)y}\leq1$, gives
\[
\begin{aligned}
	\mathrm{III}
	&\leq
	CK_\lambda e^{-2a(t_\lambda)K_\lambda}\sqrt{V(t_\lambda)}
	\int_{T-\Delta_T}^T\sqrt{V(h)}
	\exp\Biggl\{
	\int_{t_\lambda}^h\frac{r(s)}{\sigma_\alpha(s)}\,\td s
	+2\int_h^T\frac{r(s)}{\sigma_\alpha(s)}\,\td s
	\Biggr\}\\[-0.2ex]
	&\hspace{8em}\times
	\biggl(
	M_T\wedge
	\frac{M_T^2}{\int_h^T\sigma_\alpha(s)^2\,\td s}
	\biggr)\,\td h.
\end{aligned}
\]
On $h\in[T-\Delta_T,T]$,
\[
V(h)\asymp V(T),
\qquad
\int_h^T\sigma_\alpha(s)^2\,\td s
\asymp\sigma_\alpha(T)^2(T-h),
\qquad
\int_h^T\frac{r(s)}{\sigma_\alpha(s)}\,\td s\leq C.
\]
Consequently,
\[
\begin{aligned}
	\mathrm{III}
	&\leq
	CK_\lambda e^{-2a(t_\lambda)K_\lambda}
	\sqrt{V(t_\lambda)V(T)}
	\exp\Biggl\{\int_{t_\lambda}^T
	\frac{r(s)}{\sigma_\alpha(s)}\,\td s\Biggr\}
	\int_{T-\Delta_T}^T
	\biggl(
	M_T\wedge
	\frac{M_T^2}{\int_h^T\sigma_\alpha(s)^2\,\td s}
	\biggr)\,\td h.
\end{aligned}
\]
Using $\int_h^T\sigma_\alpha(s)^2\,\td s\asymp\sigma_\alpha(T)^2(T-h)$ and splitting the $h$-integral according to whether this quantity is at most or greater than $M_T$, we obtain
\[
\int_{T-\Delta_T}^T
\biggl(
M_T\wedge
\frac{M_T^2}{\int_h^T\sigma_\alpha(s)^2\,\td s}
\biggr)\,\td h
\leq
C\frac{M_T^2}{\sigma_\alpha(T)^2}\log T.
\]
Therefore, by \eqref{eqL: terminal MT value revised},
\[
\begin{aligned}
	\mathrm{III}
	&\leq
	CK_\lambda e^{-2a(t_\lambda)K_\lambda}
	\sqrt{V(t_\lambda)V(T)}
	\exp\Biggl\{\int_{t_\lambda}^T
	\frac{r(s)}{\sigma_\alpha(s)}\,\td s\Biggr\}
	\frac{M_T^2}{\sigma_\alpha(T)^2}\log T\\
	&=
	CK_\lambda e^{-2a(t_\lambda)K_\lambda}
	\sqrt{V(t_\lambda)}
	\frac{
		\displaystyle\exp\Biggl\{-\int_{t_\lambda}^T
		\frac{r(s)}{\sigma_\alpha(s)}\,\td s\Biggr\}
	}{\sqrt{V(T)}\,\sigma_\alpha(T)^2}\log T.
\end{aligned}
\]
By \eqref{eqL: lower bound, r function},
\[
\int_{t_\lambda}^T\frac{r(s)}{\sigma_\alpha(s)}\,\td s
\geq cc_r\log T.
\]
Since
\[
\frac{1}{\sqrt{V(T)}\,\sigma_\alpha(T)^2}
\leq CT^{(1+\alpha)/2},
\]
choosing $c_r$ sufficiently large yields
\[
\frac{
	\displaystyle\exp\Biggl\{-\int_{t_\lambda}^T
	\frac{r(s)}{\sigma_\alpha(s)}\,\td s\Biggr\}
}{\sqrt{V(T)}\,\sigma_\alpha(T)^2}\log T
=o_T(1).
\]
Thus,
\begin{equation}\label{eqL: lower bound second moment terminal final}
	\mathrm{III}
	\leq CK_\lambda t_\lambda^{(3\alpha-1)/2}o_T(1).
\end{equation}

\medskip

The following lemma is used to prove \eqref{eqL: ineq used for III}. Let $D:=[0,\infty)$, and recall from \eqref{eqL: p_u^D} that $p_u^D(x,y)$ is the transition density of Brownian motion killed at the origin. We use the two standard identities
\begin{equation}\label{eqL: properties of p_u^D}
	\int_0^\infty p_u^D(y,z)\,\td z\leq1,
	\qquad
	\int_0^\infty y p_u^D(y,z)\,\td y=z.
\end{equation}

\begin{lemma}\label{lem: killed BM weighted L2}
	Let $w\geq0$ be supported on $[0,L]$, and assume
	\[
	M_1(w):=\int_0^\infty zw(z)\,\td z>0,
	\qquad
	\lVert w\rVert_\infty\leq C_0,
	\qquad
	L^2\leq C_0M_1(w),
	\]
	for some constant $C_0<\infty$. Then there exists $C=C(C_0)<\infty$ such that, for every $u>0$,
	\begin{equation}\label{eq: killed BM weighted L2 estimate}
		\int_0^\infty y
		\biggl(\int_0^\infty w(z)p_u^D(y,z)\,\td z\biggr)^2\,\td y
		\leq
		C\biggl(M_1(w)\wedge\frac{M_1(w)^2}{u}\biggr).
	\end{equation}
\end{lemma}

\begin{proof}
	For $u>0$, define
	\[
	F_u(y):=\int_0^\infty w(z)p_u^D(y,z)\,\td z,
	\qquad y>0.
	\]
	We first prove the bound by $CM_1(w)$. By \eqref{eqL: properties of p_u^D} and the assumption on $w$, $F_u(y)\leq C_0$. Hence, by Fubini's theorem and \eqref{eqL: properties of p_u^D},
	\[
	\begin{aligned}
		\int_0^\infty yF_u(y)^2\,\td y
		&\leq C_0\int_0^\infty yF_u(y)\,\td y\\
		&=C_0\int_0^\infty w(z)
		\biggl(\int_0^\infty yp_u^D(y,z)\,\td y\biggr)\,\td z\\
		&=C\int_0^\infty zw(z)\,\td z
		=CM_1(w).
	\end{aligned}
	\]
	
	It remains to prove the bound by $CM_1(w)^2/u$. If $u<M_1(w)$, then $M_1(w)^2/u>M_1(w)$, and the desired estimate follows from the preceding bound. Assume now that $u\geq M_1(w)$. Since $L^2\leq C_0M_1(w)$,
	\[
	L\leq C\sqrt{M_1(w)}\leq C\sqrt{u}.
	\]
	For $z\in[0,L]$, the killed-kernel estimate gives
	\[
	p_u^D(y,z)
	\leq C\frac{yz}{u^{3/2}}
	\exp\Bigl\{-c\frac{(y-z)^2}{u}\Bigr\}
	\leq C\frac{yz}{u^{3/2}}
	\exp\Bigl\{-c'\frac{y^2}{u}\Bigr\},
	\]
	where the last inequality uses $z\leq C\sqrt{u}$. Therefore, for all $y>0$,
	\[
	F_u(y)
	\leq C\frac{y}{u^{3/2}}e^{-c'y^2/u}M_1(w).
	\]
	Hence,
	\[
	\begin{aligned}
		\int_0^\infty yF_u(y)^2\,\td y
		&\leq C\frac{M_1(w)^2}{u^3}
		\int_0^\infty y^3e^{-2c'y^2/u}\,\td y \leq C\frac{M_1(w)^2}{u},
	\end{aligned}
	\]
	where the last line uses
	$\int_0^\infty y^3e^{-c'y^2/u}\,\td y\leq Cu^2$.
	This proves \eqref{eq: killed BM weighted L2 estimate}.
\end{proof}

\appendix

\section{Estimates on the covariance kernels}
\label{Appendix, estimate for the covariance kernels}

Recall from \eqref{eqModel: alpha epsilon kernel, I_alpha} and
\eqref{eqModel: alpha epsilon kernel, as Laplacian} that, for any
$\alpha\in(0,1/2]$, the centered Gaussian field $X_\epsilon^\alpha$
defined in \eqref{eqModel: X epsilon alpha white noise integral} has
covariance kernel
\begin{equation}\label{eqK: alpha epsilon kernel, definition}
	\begin{aligned}
		k_\alpha^\epsilon(z_1,z_2)
		&:=\operatorname{Cov}\bigl(X_\epsilon^\alpha(z_1),X_\epsilon^\alpha(z_2)\bigr)\\
		&=\int_0^{\epsilon^{-2}}
		\exp\{-t\lVert z_1-z_2\rVert_2^2\}I_\alpha'(t)\,\td t,
		\qquad z_1,z_2\in\mathbb{R}^d.
	\end{aligned}
\end{equation}
where
\begin{equation}\label{eqK: alpha epsilon kernel, I alpha}
	I_\alpha(t)
	:=
	\int_0^{\log(1+t)}
	(1+s)^{-2\alpha}\,\td s.
\end{equation}
Equivalently,
\begin{equation}\label{eqK: I alpha derivative}
	I_\alpha'(t)
	=
	\frac{(1+\log(1+t))^{-2\alpha}}{1+t}.
\end{equation}

We first record a few elementary integral estimates.

\begin{lemma}\label{lemma: I alpha elementary estimates}
	Fix $\alpha\in(0,1/2]$. For every $A>0$, there exists $C_A>0$
	such that, for every $t\geq1$,
	\begin{equation}\label{eqK: alpha epsilon kernel, I_alpha upper property}
		I_\alpha(t)
		+
		A(1+\log(1+t))^{-2\alpha}
		\leq
		\int_0^{2\log(1+\sqrt{t})+C_A}
		(1+s)^{-2\alpha}\,\td s .
	\end{equation}
	Moreover, for every $A>0$, there exist constants $C_A>0$ and
	$t_A\geq1$ such that, for every $t\geq t_A$,
	\begin{equation}\label{eqK: alpha epsilon kernel, I_alpha lower property}
		I_\alpha(t)
		-
		A(1+\log(1+t))^{-2\alpha}
		\geq
		\int_0^{2\log(1+\sqrt{t})-C_A}
		(1+s)^{-2\alpha}\,\td s .
	\end{equation}
	Finally, there exists $C>0$ such that
	\begin{equation}\label{eqK: alpha epsilon kernel, an integral estimate}
		\int_0^L t\,I_\alpha'(t)\,\td t
		\leq
		C\,L\,(1+\log(1+L))^{-2\alpha},
		\qquad L\geq1.
	\end{equation}
\end{lemma}

\begin{proof}
	For $t\geq1$,
	\begin{equation}\label{eqK: log comparison}
		0
		\leq
		2\log(1+\sqrt{t})-\log(1+t)
		\leq
		\log 2 .
	\end{equation}
	
	We first prove the upper estimate. Choose $C_A>0$ so large that
	\[\inf_{a\geq\log 2}
	(1+a)^{2\alpha}
	\int_a^{a+C_A}(1+s)^{-2\alpha}\,\td s
	\geq A .
	\]
	This is possible because the left-hand side tends to infinity as
	$C_A\to\infty$. Applying this with $a=\log(1+t)$ and using
	\eqref{eqK: log comparison}, we obtain
	\[\int_{\log(1+t)}^{2\log(1+\sqrt{t})+C_A}
	(1+s)^{-2\alpha}\,\td s
	\geq
	A(1+\log(1+t))^{-2\alpha}.
	\]
	This proves \eqref{eqK: alpha epsilon kernel, I_alpha upper property}.
	
	We next prove the lower estimate. Choose $C_A>A+\log 2+1$, and then
	choose $t_A\geq1$ so large that
	$2\log(1+\sqrt{t})-C_A\geq0$ for all $t\geq t_A$. By \eqref{eqK: log comparison}, for every $t\geq t_A$,
	$2\log(1+\sqrt{t})-C_A\leq\log(1+t)-(C_A-\log 2)$.
	Therefore
	\[\begin{aligned}
		I_\alpha(t)
		-
		\int_0^{2\log(1+\sqrt{t})-C_A}
		(1+s)^{-2\alpha}\,\td s
		&\geq
		\int_{\log(1+t)-(C_A-\log 2)}^{\log(1+t)}
		(1+s)^{-2\alpha}\,\td s  \\
		&\geq
		(C_A-\log 2)(1+\log(1+t))^{-2\alpha}  \\
		&\geq
		A(1+\log(1+t))^{-2\alpha}.
	\end{aligned}
	\]
	This proves \eqref{eqK: alpha epsilon kernel, I_alpha lower property}.
	
	Finally, by \eqref{eqK: I alpha derivative},
	$\int_0^L t\,I_\alpha'(t)\,\td t
	\leq\int_0^L(1+\log(1+t))^{-2\alpha}\,\td t$.
	For $L\geq4$, split the last integral at $\sqrt{L}$:
	\[\begin{aligned}
		\int_0^L (1+\log(1+t))^{-2\alpha}\,\td t
		&\leq
		\sqrt{L}
		+
		L(1+\log(1+\sqrt{L}))^{-2\alpha}  \\
		&\leq
		C\,L\,(1+\log(1+L))^{-2\alpha}.
	\end{aligned}
	\]
	The case $1\leq L<4$ is absorbed by increasing $C$. This proves
	\eqref{eqK: alpha epsilon kernel, an integral estimate}.
\end{proof}

\begin{proposition}\label{proposition: k alpha epsilon, covariance estimate}
	Fix $\alpha\in(0,1/2]$. Moreover, there exists $C_K>0$ such that, for all $\epsilon\in(0,1)$ and all $z_1,z_2\in[0,1]^d$,
	\begin{equation}\label{eqK:alpha epsilon kernel, covariance structure}
		\int_0^{(t_{z_1,z_2}-C_K)_+}
		\sigma_\alpha(s)^2\,\td s
		\leq
		k_\alpha^\epsilon(z_1,z_2)
		\leq
		\int_0^{t_{z_1,z_2}+C_K}
		\sigma_\alpha(s)^2\,\td s,
	\end{equation}
	where $\sigma_\alpha(s):=\sqrt{2}\,(1+2s)^{-\alpha}$ and
	$t_{z_1,z_2}:=\log\bigl(1/(\epsilon\vee\lVert z_1-z_2\rVert_\infty)\bigr)$.
\end{proposition}

\begin{proof}
	We use the identity
	\begin{equation}\label{eqK: sigma integral identity}
		\int_0^T \sigma_\alpha(s)^2\,\td s
		=
		\int_0^{2T}(1+u)^{-2\alpha}\,\td u,
		\qquad T\geq0.
	\end{equation}
	We also use, without further comment, the equivalence of
	$\lVert\cdot\rVert_2$ and $\lVert\cdot\rVert_\infty$ on $\mathbb{R}^d$.
	
	First we prove the lower bound. Since
	$(\epsilon\vee\lVert z_1-z_2\rVert_2)^{-2}\leq\epsilon^{-2}$,
	\[\begin{aligned}
		k_\alpha^\epsilon(z_1,z_2)
		&\geq
		\int_0^{(\epsilon\vee\lVert z_1-z_2\rVert_2)^{-2}}
		\exp\{-t(\epsilon\vee\lVert z_1-z_2\rVert_2)^2\}
		I_\alpha'(t)\,\td t  \\
		&=
		I_\alpha\bigl((\epsilon\vee\lVert z_1-z_2\rVert_2)^{-2}\bigr)   -
		\int_0^{(\epsilon\vee\lVert z_1-z_2\rVert_2)^{-2}}
		\bigl(1-\exp\{-t(\epsilon\vee\lVert z_1-z_2\rVert_2)^2\}\bigr)
		I_\alpha'(t)\,\td t .
	\end{aligned}
	\]
	Using $1-e^{-x}\leq x$ and
	\eqref{eqK: alpha epsilon kernel, an integral estimate}, we get
	\[\begin{aligned}
		&\int_0^{(\epsilon\vee\lVert z_1-z_2\rVert_2)^{-2}}
		\bigl(1-\exp\{-t(\epsilon\vee\lVert z_1-z_2\rVert_2)^2\}\bigr)
		I_\alpha'(t)\,\td t  \\
		&\hspace{2cm}\leq C\bigl(1+\log\bigl(1+(\epsilon\vee\lVert z_1-z_2\rVert_2)^{-2}\bigr)\bigr)^{-2\alpha}.
	\end{aligned}
	\]
	Therefore
	\[
	\begin{aligned}
		k_\alpha^\epsilon(z_1,z_2)
		&\geq I_\alpha\bigl((\epsilon\vee\lVert z_1-z_2\rVert_2)^{-2}\bigr)\\
		&\quad-C\bigl(1+\log\bigl(1+(\epsilon\vee\lVert z_1-z_2\rVert_2)^{-2}\bigr)\bigr)^{-2\alpha}.
	\end{aligned}
	\]
	If $\epsilon\vee\lVert z_1-z_2\rVert_2$ is bounded away from zero, then
	$t_{z_1,z_2}$ is bounded from above, and the lower bound follows by
	choosing $C_K$ large enough. Otherwise, applying
	\eqref{eqK: alpha epsilon kernel, I_alpha lower property} with
	$t=(\epsilon\vee\lVert z_1-z_2\rVert_2)^{-2}$ gives
	\[k_\alpha^\epsilon(z_1,z_2)
	\geq
	\int_0^{2\log\bigl(1+(\epsilon\vee\lVert z_1-z_2\rVert_2)^{-1}\bigr)-C}
	(1+s)^{-2\alpha}\,\td s .
	\]
	By \eqref{eqK: sigma integral identity}, we have
	$k_\alpha^\epsilon(z_1,z_2)\geq
	\int_0^{(t_{z_1,z_2}-C_K)_+}\sigma_\alpha(s)^2\,\td s$
	after increasing $C_K$.
	
	We now prove the upper bound. If $\lVert z_1-z_2\rVert_2\leq\epsilon$, then
	$k_\alpha^\epsilon(z_1,z_2)\leq I_\alpha(\epsilon^{-2})
	\leq\int_0^{2\log(1+\epsilon^{-1})}(1+s)^{-2\alpha}\,\td s$.
	Since $\log(1+\epsilon^{-1})\leq \log(1/\epsilon)+\log 2$, and since
	$t_{z_1,z_2}=\log(1/\epsilon)+O(1)$ in this case, the upper bound follows
	from \eqref{eqK: sigma integral identity} after increasing $C_K$.
	
	It remains to consider the case $\lVert z_1-z_2\rVert_2>\epsilon$. Then
	\[\begin{aligned}
		k_\alpha^\epsilon(z_1,z_2)
		&\leq
		\int_0^{\lVert z_1-z_2\rVert_2^{-2}} I_\alpha'(t)\,\td t
		+
		\int_{\lVert z_1-z_2\rVert_2^{-2}}^\infty
		\exp\{-t\lVert z_1-z_2\rVert_2^2\}I_\alpha'(t)\,\td t  \\
		&=
		I_\alpha(\lVert z_1-z_2\rVert_2^{-2})
		+
		\int_{\lVert z_1-z_2\rVert_2^{-2}}^\infty
		\exp\{-t\lVert z_1-z_2\rVert_2^2\}I_\alpha'(t)\,\td t .
	\end{aligned}
	\]
	Using the monotonicity of $(1+\log(1+t))^{-2\alpha}$,
	\[
	\begin{aligned}
		&\int_{\lVert z_1-z_2\rVert_2^{-2}}^\infty
		\exp\{-t\lVert z_1-z_2\rVert_2^2\}I_\alpha'(t)\,\td t\\
		&\quad\leq
		\bigl(1+\log(1+\lVert z_1-z_2\rVert_2^{-2})\bigr)^{-2\alpha}
		\int_{\lVert z_1-z_2\rVert_2^{-2}}^\infty
		\frac{\exp\{-t\lVert z_1-z_2\rVert_2^2\}}{1+t}\,\td t\\
		&\quad\leq C\bigl(1+\log(1+\lVert z_1-z_2\rVert_2^{-2})\bigr)^{-2\alpha}.
	\end{aligned}
	\]
	Therefore
	\[k_\alpha^\epsilon(z_1,z_2)
	\leq
	I_\alpha(\lVert z_1-z_2\rVert_2^{-2})
	+
	C\bigl(1+\log(1+\lVert z_1-z_2\rVert_2^{-2})\bigr)^{-2\alpha}.
	\]
	If $\lVert z_1-z_2\rVert_2$ is bounded away from zero, this is bounded by a constant, and the desired upper bound follows by increasing $C_K$. Otherwise, applying
	\eqref{eqK: alpha epsilon kernel, I_alpha upper property} with
	$t=\lVert z_1-z_2\rVert_2^{-2}$ yields
	\[k_\alpha^\epsilon(z_1,z_2)
	\leq
	\int_0^{2\log(1+\lVert z_1-z_2\rVert_2^{-1})+C}
	(1+s)^{-2\alpha}\,\td s .
	\]
	Using again \eqref{eqK: sigma integral identity}, we obtain
	$k_\alpha^\epsilon(z_1,z_2)\leq
	\int_0^{t_{z_1,z_2}+C_K}\sigma_\alpha(s)^2\,\td s$
	after increasing $C_K$.
\end{proof}

\begin{proposition}\label{proposition: L2 distances for X epsilon alpha}
	Fix $\alpha\in(0,1/2]$. There exists $C>0$ such that, for every
	$\epsilon\in(0,1)$ and every $z_1,z_2\in[0,1]^d$ satisfying
	$\lVert z_1-z_2\rVert_2\leq\epsilon$,
	\begin{equation}\label{eqK: alpha epsilon kernel, L2 distance}
		\mathbb{E}\bigl[(X_\epsilon^\alpha(z_1)-X_\epsilon^\alpha(z_2))^2\bigr]
		\leq
		C\,\theta_\alpha(\epsilon)\,
		\frac{\lVert z_1-z_2\rVert_2}{\epsilon},
	\end{equation}
	where $\theta_\alpha(\epsilon):=\bigl(1+\log(1/\epsilon)\bigr)^{-2\alpha}$.
\end{proposition}

\begin{proof}
	By the covariance formula,
	\[\begin{aligned}
		\mathbb{E}\bigl[(X_\epsilon^\alpha(z_1)-X_\epsilon^\alpha(z_2))^2\bigr]
		&=
		2\int_0^{\epsilon^{-2}}
		\bigl(1-\exp\{-t\lVert z_1-z_2\rVert_2^2\}\bigr)
		I_\alpha'(t)\,\td t  \\
		&\leq
		2\lVert z_1-z_2\rVert_2^2
		\int_0^{\epsilon^{-2}} t\,I_\alpha'(t)\,\td t .
	\end{aligned}
	\]
	Using \eqref{eqK: alpha epsilon kernel, an integral estimate} with
	$L=\epsilon^{-2}$, we obtain
	\[\mathbb{E}\bigl[(X_\epsilon^\alpha(z_1)-X_\epsilon^\alpha(z_2))^2\bigr]
	\leq
	C
	\frac{\lVert z_1-z_2\rVert_2^2}{\epsilon^2}
	\bigl(1+\log\frac{1}{\epsilon}\bigr)^{-2\alpha}.
	\]
	Since $\lVert z_1-z_2\rVert_2\leq\epsilon$, we have
	$\lVert z_1-z_2\rVert_2^2/\epsilon^2
	\leq\lVert z_1-z_2\rVert_2/\epsilon$. The claim follows.
\end{proof}

\begin{remark}\label{Remark: dyadic subsequence}
To prove the tails in Theorem \ref{theorem: tightness of recentered maximum for X alpha T}, we only need to discuss the dyadic small scale $\epsilon_N=2^{-N}$ for convenience. For a general $\epsilon \in [2^{-N}, 2^{-N+1}]$, we can turn to the standard Slepian comparison between $X_\epsilon^\alpha$ and $X_{2^{-N+1}}^\alpha(\cdot)$ added by some auxiliary Gaussian constant-value field with suitable variance. See also the proof of Proposition 4.3 in \cite{Acosta2014_tightness}. The comparison introduces only a spatially constant Gaussian of uniformly bounded variance, while the centerings at neighbouring scales differ only by $O(1)$. 
\end{remark}

\section{Airy-like PDE estimate}

In this section, we introduce two tools used in the analysis of slowed-down BBM and its higher-dimensional analogues. The first is a pair of Girsanov-transform-based representations for exceedance probabilities and Brownian bridge-like barrier probabilities. The second consists of Feynman--Kac estimates for the parabolic PDEs associated with these representations.

\subsection{Results from the Girsanov transform} \label{subsec: Appendix, Girsanov}

For any continuous stochastic process $X=\{X_t\colon t\geq 0\}$, we write $\tau_0(X):=\inf\{t\geq 0\colon X_t=0\}$ for its first hitting time of zero.

\begin{lemma}[Hitting probability under a moving barrier]
	\label{Appendix: Girsanov, exceeding probability}
	Fix a probability space $(\Omega,\mathcal{F},\mathbb{P})$ carrying a Brownian motion $B$. Let $\sigma\in C^2(\mathbb{R}_+;\mathbb{R}_+)$ and define
	\[
	Y_t:=\int_0^t\sigma(s)\,\td B_s,
	\qquad t\geq 0.
	\]
	Let $c_A>0$ and let $\eta,r\in C^2([0,\infty);\mathbb R)$ be deterministic. Set
	\begin{equation}\label{eqG: Girsanov, kr(t)}
		k_r(t):=c_A\sigma(t)-\frac{\eta(t)+r(t)}{c_A},
		\qquad t\geq 0,
	\end{equation}
	and
	\begin{equation}\label{eqG: Girsanov, A(t)}
		A_r(t):=\int_0^t k_r(s)\,\td s.
	\end{equation}
	Fix $T>0$. Then, for any $\lambda>0$ and any interval $I\subset[0,T]$,
	\begin{equation}\label{eqG: Girsanov, exceeding prob}
		\begin{aligned}
			\mathbb{P}&(\tau_0(A_r+\lambda-Y)\in I)
			\leq
			\exp\{-\frac{k_r(0)}{\sigma(0)^2}\lambda\}\\
			&\times \mathbb{E}^{(\lambda)}\Biggl[
			\exp\Biggl\{
			-\frac{1}{2}c_A^2\tau_0(Y)
			+\int_0^{\tau_0(Y)}
			\Biggl(
			-\Bigl(\frac{k_r}{\sigma^2}\Bigr)'(s)Y_s
			+\frac{\eta(s)+r(s)}{\sigma(s)}
			\Biggr)\,\td s
			\Biggr\}
			\mathds{1}_{\{\tau_0(Y)\in I\}}
			\Biggr].
		\end{aligned}
	\end{equation}
	Here $\mathbb{E}^{(\lambda)}$ denotes expectation for the diffusion $Y_t=\lambda+\int_0^t\sigma(s)\,\td B_s$.
\end{lemma}

\begin{proof}
	Write
	\[
	\overline{Y}_t:=Y_t-A_r(t)
	=Y_t-\int_0^t k_r(s)\,\td s.
	\]
	Define a probability measure $\mathbb{Q}$ on $\mathcal{F}_T$ by
	\[
	\left.\frac{\td\mathbb{Q}}{\td\mathbb{P}}\right|_{\mathcal{F}_T}
	=
	\exp\Biggl\{
	\int_0^T\frac{k_r(s)}{\sigma(s)^2}\,\td Y_s
	-\frac{1}{2}\int_0^T\frac{k_r(s)^2}{\sigma(s)^2}\,\td s
	\Biggr\}.
	\]
	For $0\leq t\leq T$, the corresponding density process is
	\[
	\left.\frac{\td\mathbb{Q}}{\td\mathbb{P}}\right|_{\mathcal{F}_t}
	=
	\exp\Biggl\{
	\int_0^t\frac{k_r(s)}{\sigma(s)^2}\,\td Y_s
	-\frac{1}{2}\int_0^t\frac{k_r(s)^2}{\sigma(s)^2}\,\td s
	\Biggr\}.
	\]
	By Girsanov's theorem, under $\mathbb{Q}$ the process $\overline{Y}$ has the same law as $Y$ under $\mathbb{P}$.
	
	Let $f(t):=k_r(t)/\sigma(t)^2$. By integration by parts,
	\[
	\int_0^t f(s)\,\td\overline{Y}_s
	=f(t)\overline{Y}_t-f(0)\overline{Y}_0
	-\int_0^t f'(s)\overline{Y}_s\,\td s.
	\]
	Since $\overline{Y}_0=0$, this gives
	\[
	\left.\frac{\td\mathbb{Q}}{\td\mathbb{P}}\right|_{\mathcal{F}_t}
	=
	\exp\Biggl\{
	f(t)\overline{Y}_t
	-\int_0^t f'(s)\overline{Y}_s\,\td s
	+\frac{1}{2}\int_0^t\frac{k_r(s)^2}{\sigma(s)^2}\,\td s
	\Biggr\}.
	\]
	Set $\tau:=\tau_0(\lambda-\overline{Y})$ and $\tau_T:=\tau\wedge T$. Applying the preceding identity at the bounded stopping time $\tau_T$ and then restricting to the event $\{\tau\in I\}$ gives

	\begin{equation}\label{eqG: Girsanov, long}
		\begin{aligned}
			\mathbb{P}\bigl(\tau_0(A_r+\lambda-Y)\in I\bigr)
			&=\mathbb{E}_{\mathbb{Q}}\Biggl[
			\exp\Biggl\{-f(\tau)\overline{Y}_{\tau}
			+\int_0^{\tau}f'(s)\overline{Y}_s\,\td s
			-\frac{1}{2}\int_0^{\tau}\frac{k_r(s)^2}{\sigma(s)^2}\,\td s
			\Biggr\}\\[-0.2ex]
			&\hspace{7em}\times
			\mathds{1}_{\{\tau_0(\lambda-\overline{Y})\in I\}}
			\Biggr]\\
			&=\mathbb{E}_{\mathbb{P}}\Biggl[
			\exp\Biggl\{-f(\tau_0)Y_{\tau_0}
			+\int_0^{\tau_0}f'(s)Y_s\,\td s
			-\frac{1}{2}\int_0^{\tau_0}\frac{k_r(s)^2}{\sigma(s)^2}\,\td s
			\Biggr\}\\[-0.2ex]
			&\hspace{7em}\times
			\mathds{1}_{\{\tau_0(\lambda-Y)\in I\}}
			\Biggr].
		\end{aligned}
	\end{equation}
	
	where, in the last line, $\tau_0=\tau_0(\lambda-Y)$.
	
	By the symmetry $Y\overset{\mathrm d}{=}-Y$, the last display equals
	\[
	\mathbb{E}_{\mathbb{P}}\Biggl[
	\exp\Biggl\{
	-f(\tau_0)\lambda
	-\int_0^{\tau_0}f'(s)Y_s\,\td s
	-\frac{1}{2}\int_0^{\tau_0}\frac{k_r(s)^2}{\sigma(s)^2}\,\td s
	\Biggr\}
	\mathds{1}_{\{\tau_0(\lambda+Y)\in I\}}
	\Biggr].
	\]
	Equivalently, if under $\mathbb{E}^{(\lambda)}$ we write $Y_t=\lambda+\int_0^t\sigma(s)\,\td B_s$, then the preceding expression becomes
	\[
	\mathbb{E}^{(\lambda)}\Biggl[
	\exp\Biggl\{
	-f(0)\lambda
	-\int_0^{\tau_0(Y)}f'(s)Y_s\,\td s
	-\frac{1}{2}\int_0^{\tau_0(Y)}\frac{k_r(s)^2}{\sigma(s)^2}\,\td s
	\Biggr\}
	\mathds{1}_{\{\tau_0(Y)\in I\}}
	\Biggr].
	\]
	Finally, using
	\[
	-\frac{1}{2}\frac{k_r(s)^2}{\sigma(s)^2}
	=-\frac{1}{2}c_A^2
	+\frac{\eta(s)+r(s)}{\sigma(s)}
	-\frac{(\eta(s)+r(s))^2}{2c_A^2\sigma(s)^2}
	\leq -\frac{1}{2}c_A^2+\frac{\eta(s)+r(s)}{\sigma(s)},
	\]
	we obtain \eqref{eqG: Girsanov, exceeding prob}.
\end{proof}

In the main text, we need to estimate an expectation slightly different from \eqref{eqG: Girsanov, exceeding prob}, namely
\[
\mathbb{E}^{(\lambda)}\Biggl[
\exp\Biggl\{
\int_0^{\tau_0(Y)}
\Biggl(
-\Bigl(\frac{k_r}{\sigma^2}\Bigr)'(s)Y_s
+\frac{\eta(s)+r(s)}{\sigma(s)}
\Biggr)\,\td s
\Biggr\}
\mathds{1}_{\{\tau_0(Y)\in I\}}
\Biggr].
\]
This expectation can be represented as
\begin{equation}\label{eqH: PDE representation of hitting in some interval}
	\int_I\frac{1}{2}\sigma(t)^2
	\left.\partial_y G_r(\lambda,y;0,t)\right|_{y=0}\,\td t,
\end{equation}
where $G_r$ is the fundamental solution of the parabolic equation
\begin{equation}\label{eqA: Airy PDE with sigma diffusion}
	\partial_tu
	=\frac{1}{2}\sigma(t)^2\partial_{xx}u
	-\Bigl(\frac{k_r}{\sigma^2}\Bigr)'(t)xu
	+\frac{\eta(t)+r(t)}{\sigma(t)}u,
\end{equation}
on $\mathbb{R}_+\times\mathbb{R}_+$ with Dirichlet boundary condition $u(t,0)=0$. For references on this representation, see the discussion following formula $(2.5)$ in \cite{MZ2016_slowdownBBM}.

The following lemma gives a PDE representation of the bridge--barrier probability.

\begin{lemma}[Bridge-like probability]
	\label{Appendix: Bridge like probability}
	Under the same assumptions as Lemma \ref{Appendix: Girsanov, exceeding probability}, for any $x,y>0$ and $0\leq t_1<t_2$, we have
	\begin{equation}
		\begin{aligned}
			&\mathbb{P}\Biggl(
			Y_t-Y_{t_1}\le x+\int_{t_1}^t k_r(s)\,\td s
			\ \forall t\in[t_1,t_2],\ 
			x+\int_{t_1}^{t_2}k_r(s)\,\td s
			-(Y_{t_2}-Y_{t_1})\in\td y
			\Biggr)\\
			&\quad=
			\exp\Biggl\{
			-\frac{k_r(t_1)}{\sigma(t_1)^2}x
			+\frac{k_r(t_2)}{\sigma(t_2)^2}y
			-\frac{c_A^2}{2}(t_2-t_1)
			-\frac{1}{2}\int_{t_1}^{t_2}
			\frac{(\eta(s)+r(s))^2}{c_A^2\sigma(s)^2}\,\td s
			\Biggr\}
			G_r(x,y;t_1,t_2)\,\td y.
		\end{aligned}
	\end{equation}
	
	Here $G_r$ is the fundamental solution of \eqref{eqA: Airy PDE with sigma diffusion}.
\end{lemma}

\begin{proof}
	Let
	\[
	Z_t:=x+\int_{t_1}^t k_r(s)\,\td s-(Y_t-Y_{t_1}),
	\qquad t\in[t_1,t_2].
	\]
	Then the event in the statement is $\{Z_t\geq 0\ \forall\,t\in[t_1,t_2],\ Z_{t_2}\in\td y\}$. Applying the same Girsanov calculation as in \eqref{eqG: Girsanov, long}, now on the interval $[t_1,t_2]$, gives
	\[
	\begin{aligned}
		&\mathbb{P}\Biggl(
		Y_t-Y_{t_1}\leq x+\int_{t_1}^t k_r(s)\,\td s
		\ \forall\,t\in[t_1,t_2],\ Z_{t_2}\in\td y
		\Biggr)\\
		&\quad=
		\mathbb{E}_{\mathbb{P}}^{(t_1,x)}\Biggl[
		\exp\Biggl\{
		-\frac{k_r(t_1)}{\sigma(t_1)^2}x
		+\frac{k_r(t_2)}{\sigma(t_2)^2}y
		+\int_{t_1}^{t_2}
		\Biggl(
		-\Bigl(\frac{k_r}{\sigma^2}\Bigr)'(s)Y_s
		-\frac{1}{2}\frac{k_r(s)^2}{\sigma(s)^2}
		\Biggr)\,\td s
		\Biggr\} \\
		& \times \mathds{1}_{\{\tau_0(Y)>t_2\}}
		\mathds{1}_{\{Y_{t_2}\in\td y\}}
		\Biggr].
	\end{aligned}
	\]
	Here $\mathbb{E}_{\mathbb{P}}^{(t_1,x)}$ denotes expectation for the diffusion $Y$ started from $x$ at time $t_1$, and $\tau_0(Y):=\inf\{t\geq t_1\colon Y_t=0\}$. Using
	\[
	-\frac{1}{2}\frac{k_r(s)^2}{\sigma(s)^2}
	=-\frac{1}{2}c_A^2
	+\frac{\eta(s)+r(s)}{\sigma(s)}
	-\frac{(\eta(s)+r(s))^2}{2c_A^2\sigma(s)^2},
	\]
	and applying the Feynman--Kac formula for the killed diffusion on $\mathbb{R}_+$ yields the claimed representation.
\end{proof}

\subsection{Airy estimate for general potentials}
We begin by recalling some basic facts about the Airy operator. The Airy function of the first kind is defined by
\[
\operatorname{Ai}(x):=\frac{1}{\pi}\int_0^\infty
\cos\Bigl(\frac{t^3}{3}+xt\Bigr)\,\td t,
\]
and satisfies the ODE $\operatorname{Ai}''(x)-x\operatorname{Ai}(x)=0$. Its zeros are located at $-\alpha_n$, where
\begin{equation}\label{eqA: Airy eigenvalues}
	0<\alpha_1<\alpha_2<\cdots,
	\qquad \alpha_1=2.33811\ldots
\end{equation}
For each $q>0$, consider the operator $L_qu:=\partial_{xx}u-qxu$ on $L^2(\mathbb{R}_+)$ with Dirichlet boundary condition at $x=0$. Its eigenfunctions are $\psi_n^q(x):=q^{1/6}\psi_n(q^{1/3}x)$, $n\geq 1$, with corresponding eigenvalues $-\alpha_nq^{2/3}$, where
\begin{equation}\label{eqA: Airy, psi_n}
	\psi_n(x):=
	\frac{\operatorname{Ai}(x-\alpha_n)}
	{\lVert\operatorname{Ai}(\cdot-\alpha_n)\rVert_{L^2(\mathbb{R}_+)}}.
\end{equation}
Then $\{\psi_n^q\}_{n\geq 1}$ is an orthonormal basis of $L^2(\mathbb{R}_+)$ and
\begin{equation}\label{eq:airy-eigen-equation}
	(\partial_{xx}-qx)\psi_n^q=-\alpha_nq^{2/3}\psi_n^q,
	\qquad \psi_n^q(0)=0.
\end{equation}
We shall use the following facts from Lemma~A.1 of \cite{MZ2016_slowdownBBM} and the proof of Corollary~A.4 therein.

\begin{lemma}\label{lem: Airy bounds}
	There is a constant $C_{\mathrm A}<\infty$ such that, for every $n\geq 1$, $q>0$, and $x\geq 0$,
	\begin{equation}\label{eqAiry: eigenvector, upper}
		\lvert\psi_n^q(x)\rvert\leq C_{\mathrm A}\sqrt{q}\,x.
	\end{equation}
	Moreover, $\alpha_nn^{-2/3}\longrightarrow(3\pi/2)^{2/3}$. Consequently, for every $m>0$,
	\begin{equation}\label{eqAiry: sum e^{-an-a1}}
		\rho(m)^2:=\sum_{n=2}^{\infty}
		\exp\{-2(\alpha_n-\alpha_1)m\}<\infty,
		\qquad \rho(m)\longrightarrow 0\quad(m\to\infty).
	\end{equation}
	Finally, for every $L>0$, there is $C_L>0$ such that
	\begin{equation}\label{eqAiry: eigenvector, lower}
		\psi_1^q(x)\geq C_L\sqrt{q}\,x,
		\qquad x\in[0,Lq^{-1/3}].
	\end{equation}
\end{lemma}

The following elementary scaling estimate will be used in the proof of the lower bound in Proposition~\ref{propoAiry: any q potential}.

\begin{lemma}\label{lemAiry: psi_1^p -psi_1^q}
	There is a numerical constant $C>0$ such that, for all $p,q>0$,
	\[
	\lVert\psi_1^p-\psi_1^q\rVert_{L^2(\mathbb{R}_+)}
	\leq C\bigl\lvert\log\frac{p}{q}\bigr\rvert.
	\]
\end{lemma}

\begin{proof}
	Set
	\[
	F_\lambda(x):=\psi_1^{\exp\{\lambda\}}(x)
	=\exp\{\lambda/6\}\psi_1\bigl(\exp\{\lambda/3\}x\bigr),
	\qquad x\geq 0.
	\]
	A direct calculation gives
	\[
	\partial_\lambda F_\lambda(x)
	=\exp\{\lambda/6\}
	\Biggl[
	\frac{1}{6}\psi_1\bigl(\exp\{\lambda/3\}x\bigr)
	+\frac{1}{3}\exp\{\lambda/3\}x
	\psi_1'\bigl(\exp\{\lambda/3\}x\bigr)
	\Biggr].
	\]
	After the change of variables $z=\exp\{\lambda/3\}x$,
	\[
	\lVert\partial_\lambda F_\lambda\rVert_2^2
	=\int_0^\infty
	\bigl\lvert\frac{1}{6}\psi_1(z)+\frac{1}{3}z\psi_1'(z)\bigr\rvert^2\,\td z
	=:C_0^2<\infty,
	\]
	independently of $\lambda$. Therefore,
	\[
	\begin{aligned}
		\lVert\psi_1^p-\psi_1^q\rVert_2
		&=\lVert F_{\log p}-F_{\log q}\rVert_2\\
		&\leq
		\int_{\min\{\log p,\log q\}}^{\max\{\log p,\log q\}}
		\lVert\partial_\lambda F_\lambda\rVert_2\,\td\lambda
		\leq C_0\bigl\lvert\log\frac{p}{q}\bigr\rvert.
	\end{aligned}
	\]
\end{proof}

We now introduce the Airy PDE considered in this subsection. Let $a<b$ and $q\in C^1([a,b])$ be strictly positive. Consider
\begin{equation}\label{eqAiry: Airy PDE with q}
	\begin{cases}
		\partial_\theta w(\theta,x)
		=\partial_{xx}w(\theta,x)-q(\theta)xw(\theta,x),
		&x>0,\ a\leq\theta\leq b,\\
		w(\theta,0)=0.
	\end{cases}
\end{equation}
Denote its fundamental solution by $K_q(x,y;a,b)$. Standard Feynman--Kac theory for Brownian motion killed at the origin gives positivity of the kernel and the comparison rule
\[
q_-(\theta)\leq q(\theta)\leq q_+(\theta)
\quad\Longrightarrow\quad
K_{q_+}(x,y;a,b)
\leq K_q(x,y;a,b)
\leq K_{q_-}(x,y;a,b).
\]
The same comparison remains valid for bounded piecewise continuous potentials.

We introduce notations that will be used below. Define the Airy action
\begin{equation}\label{eqAiry: Lambda_q(a,b)}
	\Lambda_q(a,b):=\int_a^bq(\theta)^{2/3}\,\td\theta
\end{equation}
and set
\begin{equation}\label{eqAiry: s(theta)}
	s(\theta):=\int_a^\theta q(u)^{2/3}\,\td u,
	\qquad 0\leq s\leq\Lambda_q(a,b).
\end{equation}
Let $\theta=\theta(s)$ denote the inverse map, and put
\begin{equation}\label{eqAiry: B_q(s)}
	B_q(s):=\frac{\td}{\td s}\log q(\theta(s)).
\end{equation}
For brevity, write $\Lambda_q=\Lambda_q(a,b)$, and define
\begin{equation}\label{eqAiry: epsilon(q), mathcal E(q)}
	\varepsilon_\infty(q):=\lVert B_q\rVert_{L^\infty(0,\Lambda_q)},
	\qquad
	\mathcal{E}(q):=\int_0^{\Lambda_q}\lvert B_q(s)\rvert^2\,\td s.
\end{equation}

To study \eqref{eqAiry: Airy PDE with q}, we use the same eigenfunction-decomposition strategy as in \cite{MZ2016_slowdownBBM}. For any $L^2$-solution $w(\theta,\cdot)$ of \eqref{eqAiry: Airy PDE with q}, define
\begin{equation}\label{eqAiry: W(theta,)}
	W(\theta,\cdot)
	:=
	\exp\Bigl\{
	\alpha_1\int_a^\theta q(u)^{2/3}\,\td u
	\Bigr\}w(\theta,\cdot)
\end{equation}
and expand
\begin{equation}\label{eqAiry: expand W by Airy eigenvectors}
	W(\theta,\cdot)
	=
	\sum_{n\geq 1}c_n(\theta)\psi_n^{q(\theta)}.
\end{equation}
The point of the normalization in \eqref{eqAiry: W(theta,)} is that the principal coefficient $c_1$ has no spectral damping, whereas all higher modes decay at rates $\alpha_n-\alpha_1$. Thus, if $c_1$ varies little and the higher modes are sufficiently damped, then $w(\theta,\cdot)$ is well approximated by the principal eigenfunction $\psi_1^{q(\theta)}$ multiplied by
\[
\exp\Bigl\{-\alpha_1\int_a^\theta q(u)^{2/3}\,\td u\Bigr\}.
\]

More precisely, set
\[
c_n(\theta)
:=\bigl\langle W(\theta,\cdot),\psi_n^{q(\theta)}\bigr\rangle,
\qquad n\geq 1,
\]
and write $c_n(s):=c_n(\theta(s))$. The same computation as in formula~(A.2) and Lemma~A.3 of \cite{MZ2016_slowdownBBM} gives
\begin{equation}\label{eqAiry: dc/ds ODE}
	\frac{\td c}{\td s}
	=\bigl(-\Gamma+B_q(s)A\bigr)c,
\end{equation}
where
\[
\begin{gathered}
	c(s):=(c_1(s),c_2(s),\ldots)^{\mathsf T},
	\qquad
	\Gamma:=\operatorname{diag}
	\bigl(0,\alpha_2-\alpha_1,\alpha_3-\alpha_1,\ldots\bigr),\\
	A_{nk}:=
	\biggl\langle
	\psi_k,\frac{1}{6}\psi_n+\frac{1}{3}x\psi_n'
	\biggr\rangle_{L^2(\mathbb{R}_+)}.
\end{gathered}
\]
The matrix $A$ is real and skew-symmetric. Moreover,
\begin{equation}\label{eqAiry: sum |A_1j|^2}
	C_A^2:=\sum_{j=2}^{\infty}\lvert A_{1j}\rvert^2<\infty,
\end{equation}
and
\begin{equation}\label{eqAiry: |c(s)| decreases}
	\frac{\td}{\td s}\lVert c(s)\rVert_{\ell^2}^2
	=-2\sum_{n=2}^{\infty}
	(\alpha_n-\alpha_1)\lvert c_n(s)\rvert^2
	\leq 0.
\end{equation}
The proof is identical to that of Lemma~A.3 in \cite{MZ2016_slowdownBBM} and is omitted.

Let $\gamma:=\alpha_2-\alpha_1>0$ and $c_\perp:=(c_2,c_3,\ldots)$. The following lemma controls the variation of the principal coefficient $c_1$.

\begin{lemma}\label{lemAiry: c_1 small difference}
	Suppose that $c$ solves \eqref{eqAiry: dc/ds ODE} on $[s_0,s_1]$. Then, for $s\in[s_0,s_1]$,
	\begin{equation}\label{eq:excited-mode-variation}
		\lVert c_\perp(s)\rVert_{\ell^2}
		\leq
		\exp\{-\gamma(s-s_0)\}\lVert c_\perp(s_0)\rVert_{\ell^2}
		+C_A\lVert c(s_0)\rVert_{\ell^2}
		\int_{s_0}^s
		\exp\{-\gamma(s-u)\}\lvert B_q(u)\rvert\,\td u
	\end{equation}
	with $C_A$ defined in \eqref{eqAiry: sum |A_1j|^2}. Also,
	\begin{equation}\label{eqAiry: c_1(s) -c_1(s_0)}
		\lvert c_1(s)-c_1(s_0)\rvert
		\leq C \lVert c(s_0)\rVert_{\ell^2}
		\bigl(\mathcal{E}(q)^{1/2}+\mathcal{E}(q)\bigr)
	\end{equation}
	for a numerical constant $C>0$.
\end{lemma}

\begin{proof}
	By \eqref{eqAiry: dc/ds ODE} and the skew-symmetry of $A$,
	\begin{equation}\label{eq:bar-energy}
		\frac{1}{2}\frac{\td}{\td s}\lVert c_\perp(s)\rVert_{\ell^2}^2
		\leq
		-\gamma\lVert c_\perp(s)\rVert_{\ell^2}^2
		+C_A\lvert B_q(s)\rvert\lvert c_1(s)\rvert
		\lVert c_\perp(s)\rVert_{\ell^2}.
	\end{equation}
	Using \eqref{eqAiry: |c(s)| decreases} and applying the standard ODE comparison argument to \eqref{eq:bar-energy} yields \eqref{eq:excited-mode-variation}.
	
	The first coordinate of \eqref{eqAiry: dc/ds ODE} satisfies
	\begin{equation}\label{eq:c1-derivative}
		\lvert \frac{\td c_1}{\td s}(s)\rvert
		\leq C_A\lvert B_q(s)\rvert\lVert c_\perp(s)\rVert_{\ell^2}.
	\end{equation}
	Insert \eqref{eq:excited-mode-variation} into the integral of \eqref{eq:c1-derivative}. The contribution of the first term in \eqref{eq:excited-mode-variation} is bounded by
	\[
	\int_{s_0}^s
	\lvert B_q(v)\rvert \exp\{-\gamma(v-s_0)\}\,\td v
	\leq(2\gamma)^{-1/2}
	\lVert B_q\rVert_{L^2(s_0,s)}.
	\]
	For the second term, Young's convolution inequality gives
	\[
	\int_{s_0}^s\lvert B_q(v)\rvert
	\int_{s_0}^v\exp\{-\gamma(v-u)\}\lvert B_q(u)\rvert\,\td u\,\td v
	\leq\frac{1}{\gamma}
	\lVert B_q\rVert_{L^2(s_0,s)}^2.
	\]
	Since $\lVert B_q\rVert_{L^2(s_0,s)}\leq\mathcal{E}(q)^{1/2}$ by the definition \eqref{eqAiry: epsilon(q), mathcal E(q)}, this proves \eqref{eqAiry: c_1(s) -c_1(s_0)}.
\end{proof}

We now state the main Airy estimates for \eqref{eqAiry: Airy PDE with q}. Recall that $\Lambda_q$, $s(\cdot)$, $B_q(\cdot)$, $\varepsilon_\infty(q)$, and $\mathcal{E}(q)$ are defined in \eqref{eqAiry: Lambda_q(a,b)}--\eqref{eqAiry: epsilon(q), mathcal E(q)}.

\begin{proposition}\label{propoAiry: any q potential}
	Let $K_q$ be the fundamental solution of \eqref{eqAiry: Airy PDE with q}. There exist numerical constants $\varepsilon_\star,\mathcal{E}_\star,\Lambda_\star>0$ with the following properties.
	
	\begin{enumerate}
		\item For every $\ell_0>0$, there is $C=C(\ell_0)<\infty$ such that, if $\Lambda_q(a,b)\geq\ell_0$ and $\varepsilon_\infty(q)\leq\varepsilon_\star$, then
		\begin{equation}\label{eqAiry: abstract upper}
			K_q(x,y;a,b)
			\leq
			Cxy\sqrt{q(a)q(b)}
			\exp\{-\alpha_1\Lambda_q(a,b)\},
			\qquad x,y\geq 0.
		\end{equation}
		
		\item If $\Lambda_q(a,b)\geq\Lambda_\star$, $\varepsilon_\infty(q)\leq\varepsilon_\star$, and $\mathcal{E}(q)\leq\mathcal{E}_\star$, then
		\begin{equation}\label{eqAiry: abstract-lower}
			K_q(x,y;a,b)
			\geq
			C^{-1}xy\sqrt{q(a)q(b)}
			\exp\{-\alpha_1\Lambda_q(a,b)\}
		\end{equation}
		for $0\leq x\leq 2q(a)^{-1/3}$ and $0\leq y\leq 2q(b)^{-1/3}$.
		Here $C<\infty$ is numerical.
	\end{enumerate}
\end{proposition}

\begin{proof}
	\emph{Upper bound.}
	Fix $\ell_0>0$ and, for this part of the proof, set $m_0:=\min\{1,\ell_0/4\}$.
	Since $\Lambda_q(a,b)\geq\ell_0$, we have $2m_0<\Lambda_q(a,b)$. Hence there are unique $\theta_L,\theta_R\in(a,b)$, with $\theta_L<\theta_R$, such that
	\begin{equation}\label{eqAiry: int_a^theta_L ...=m_0}
		\int_a^{\theta_L}q(\theta)^{2/3}\,\td\theta=m_0,
		\qquad
		\int_{\theta_R}^{b}q(\theta)^{2/3}\,\td\theta=m_0.
	\end{equation}
	Write $\varepsilon:=\varepsilon_\infty(q)$. For every $\theta\in[a,\theta_L]$,
	\begin{equation}\label{eqAiry: log q -log q}
		\lvert\log q(\theta)-\log q(a)\rvert
		=
		\lvert\int_0^{s(\theta)}B_q(u)\,\td u\rvert
		\leq\varepsilon m_0.
	\end{equation}
	Similarly,
	\[
	\lvert\log q(\theta)-\log q(b)\rvert
	\leq\varepsilon m_0,
	\qquad \theta\in[\theta_R,b].
	\]
	Define
	\begin{equation}\label{eqAiry: underline q}
		\underline{q}(\theta):=
		\begin{cases}
			p_L:=\exp\{-\varepsilon m_0\}q(a),&a\leq\theta<\theta_L,\\
			q(\theta),&\theta_L\leq\theta\leq\theta_R,\\
			p_R:= \exp\{-\varepsilon m_0\}q(b),&\theta_R<\theta\leq b.
		\end{cases}
	\end{equation}
	It follows from \eqref{eqAiry: log q -log q} and its right-end analogue that
	\begin{equation}\label{eqAiry: underline q <q}
		\underline{q}(\theta)\leq q(\theta),
		\qquad \theta\in[a,b].
	\end{equation}
	
	Set $\underline{J}_L:=p_L^{2/3}(\theta_L-a)$ and $\underline{J}_R:=p_R^{2/3}(b-\theta_R)$.
	Since $q(\theta)\leq \exp\{\varepsilon m_0\}q(a)$ on $[a,\theta_L]$,
	\begin{equation}\label{eqAiry: int underline q ^2/3 L}
		\underline{J}_L
		\geq m_0\exp\{-\frac{4}{3}\varepsilon m_0\}
		\geq m_0\exp\{-\frac{4}{3}\varepsilon_\star\}
		>0,
	\end{equation}
	where we used $m_0\leq 1$. Likewise,
	\begin{equation}\label{eqAiry: int underline q ^2/3 R}
		\underline{J}_R\geq m_0\exp\{-\frac{4}{3}\varepsilon_\star\}>0.
	\end{equation}
	Notice that the lower bound here is allowed to depend on $\ell_0$; this is precisely where the constant in \eqref{eqAiry: abstract upper} acquires its $\ell_0$-dependence.
	
	Let $\underline{K}$ be the fundamental solution of \eqref{eqAiry: Airy PDE with q} associated with $\underline{q}$. For fixed $x\geq 0$, define
	\begin{equation}\label{eq:minorant-normalized-kernel}
		\underline{W}(\theta,y)
		:=
		\exp\Bigl\{
		\alpha_1\int_a^\theta\underline{q}(u)^{2/3}\,\td u
		\Bigr\}
		\underline{K}(x,y;a,\theta)
	\end{equation}
	and let $\underline{c}_n(\theta)$ be its coefficients in the basis $\{\psi_n^{\underline{q}(\theta)}\}_{n\geq 1}$ on each interval of continuity of $\underline{q}$.
	
	On $[a,\theta_L)$, the potential $\underline{q}$ defined in \eqref{eqAiry: underline q} is constant and equal to $p_L$. Therefore,
	\[
	\underline{c}_n(\theta_L^-)
	=
	\exp\{-(\alpha_n-\alpha_1)\underline{J}_L\}
	\psi_n^{p_L}(x),\qquad n\geq 1.
	\]
	Using \eqref{eqAiry: eigenvector, upper}, \eqref{eqAiry: sum e^{-an-a1}}, and \eqref{eqAiry: int underline q ^2/3 L}, we obtain
	\[
	\begin{aligned}
		\lVert \underline{c}(\theta_L^-)\rVert_{\ell^2}^2
		\leq
		C_{\mathrm A}^2p_Lx^2
		\bigl(1+\rho(\underline{J}_L)^2\bigr) \leq C(\ell_0)p_Lx^2
		\leq C(\ell_0)q(a)x^2.
	\end{aligned}
	\]
	At a jump of the potential, the function $\underline{W}$ is unchanged and only the orthonormal basis changes. Parseval's identity therefore gives
	\[
	\lVert \underline{c}(\theta_L^+)\rVert_{\ell^2}
	=
	\lVert \underline{c}(\theta_L^-)\rVert_{\ell^2}.
	\]
	On $[\theta_L,\theta_R]$, \eqref{eqAiry: |c(s)| decreases} implies that the coefficient norm is non-increasing. Applying Parseval once more at $\theta_R$, we obtain
	\[
	\lVert \underline{c}(\theta_R^+)\rVert_{\ell^2}
	\leq C(\ell_0)\sqrt{q(a)}\,x.
	\]
	
	On $(\theta_R,b]$, the potential is constant and equal to $p_R$. Hence
	\[
	\underline{W}(b,y)
	=
	\sum_{n\geq 1}
	\exp\{-(\alpha_n-\alpha_1)\underline{J}_R\}
	\underline{c}_n(\theta_R^+)\psi_n^{p_R}(y).
	\]
	By Cauchy--Schwarz, \eqref{eqAiry: eigenvector, upper}, and \eqref{eqAiry: int underline q ^2/3 R},
	\[
	\begin{aligned}
		\lvert \underline{W}(b,y)\rvert
		&\leq
		\lVert \underline{c}(\theta_R^+)\rVert_{\ell^2}
		\bigl(
		\sum_{n\geq 1}
		\exp\{-2(\alpha_n-\alpha_1)\underline{J}_R\}
		\lvert \psi_n^{p_R}(y)\rvert^2
		\bigr)^{1/2}\\
		&\leq C(\ell_0)xy\sqrt{p_Lp_R}
		\leq C(\ell_0)xy\sqrt{q(a)q(b)}.
	\end{aligned}
	\]
	
	Finally, by \eqref{eqAiry: underline q <q} and the comparison principle,
	$K_q\leq\underline{K}$. Moreover,
	\[
	0\leq
	\Lambda_q(a,b)-\int_a^b\underline{q}(u)^{2/3}\,\td u
	\leq 2m_0
	\]
	by \eqref{eqAiry: int_a^theta_L ...=m_0} and \eqref{eqAiry: underline q}. Consequently,
	\begin{equation}\label{eqAiry: final stroke for K_q upper bound}
		\begin{aligned}
			K_q(x,y;a,b)
			&\leq K_{\underline{q}}(x,y;a,b)\\
			&=
			\exp\Bigl\{-\alpha_1\int_a^b\underline{q}(u)^{2/3}\,\td u\Bigr\}
			\underline{W}(b,y)\\
			&\leq
			C(\ell_0)\exp\{2\alpha_1m_0\}
			xy\sqrt{q(a)q(b)}
			\exp\{-\alpha_1\Lambda_q(a,b)\}.
		\end{aligned}
	\end{equation}
	Absorbing $\exp\{2\alpha_1m_0\}$ into $C(\ell_0)$ proves \eqref{eqAiry: abstract upper}.
	
	\emph{Lower bound.}
	Let $C_3$ be the constant in \eqref{eqAiry: eigenvector, lower} for $L=3$, and set $C_0:=C_{\mathrm A}(1+\rho(1)^2)^{1/2}$, with $C_A$ defined in \eqref{eqAiry: sum |A_1j|^2} and $\rho(\cdot)$ in \eqref{eqAiry: sum e^{-an-a1}}. Choose a numerical constant $m_0>1$ so large that $C_{\mathrm A}C_0\rho(m_0)\leq c_0^2/4$.
	We shall choose $\Lambda_\star>2m_0$ and, after fixing $m_0$, shrink $\varepsilon_\star$ so that $\varepsilon_\star m_0\leq 1$. If $\Lambda_q(a,b)\geq\Lambda_\star$, there are $\theta_L,\theta_R\in(a,b)$, with $\theta_L<\theta_R$, satisfying \eqref{eqAiry: int_a^theta_L ...=m_0}. Again write $\varepsilon:=\varepsilon_\infty(q)$, and define
	\begin{equation}\label{eqAiry: overline q}
		\overline{q}(\theta):=
		\begin{cases}
			p_L:=\exp\{\varepsilon m_0\}q(a),&a\leq\theta<\theta_L,\\
			q(\theta),&\theta_L\leq\theta\leq\theta_R,\\
			p_R:=\exp\{\varepsilon m_0\}q(b),&\theta_R<\theta\leq b.
		\end{cases}
	\end{equation}
	The endpoint oscillation estimate \eqref{eqAiry: log q -log q} and its right-end analogue imply
	\[
	\overline{q}(\theta)\geq q(\theta),
	\qquad \theta\in[a,b].
	\]
	Let $\overline{J}_L:=p_L^{2/3}(\theta_L-a)$ and $\overline{J}_R:=p_R^{2/3}(b-\theta_R)$.
	The same endpoint estimates as \eqref{eqAiry: int underline q ^2/3 L} and \eqref{eqAiry: int underline q ^2/3 R} give
	\begin{equation}\label{eqAiry: int overline q ^2/3 L}
		m_0\leq\overline{J}_L
		\leq m_0\exp\{\frac{4}{3}\varepsilon m_0\},
	\end{equation}
	and
	\begin{equation}\label{eqAiry: int overline q ^2/3 R}
		m_0\leq\overline{J}_R
		\leq m_0\exp\{\frac{4}{3}\varepsilon m_0\}.
	\end{equation}
	
	Denote by $\overline{K}$ the kernel associated with $\overline{q}$, and define
	\[
	\overline{W}(\theta,y)
	:=
	\exp\Bigl\{
	\alpha_1\int_a^\theta\overline{q}(u)^{2/3}\,\td u
	\Bigr\}
	\overline{K}(x,y;a,\theta).
	\]
	Let $\overline{c}_n(\theta)$ be the corresponding coefficients in the basis $\{\psi_n^{\overline{q}(\theta)}\}_{n\geq 1}$.
	
	On $[a,\theta_L)$,
	\[
	\overline{c}_n(\theta_L^-)
	=
	\exp\{-(\alpha_n-\alpha_1)\overline{J}_L\}
	\psi_n^{p_L}(x).
	\]
	Hence, by \eqref{eqAiry: eigenvector, upper}, \eqref{eqAiry: sum e^{-an-a1}}, and $\overline{J}_L\geq m_0>1$,
	\begin{equation}\label{eqAiry: | bar c(theta_L)| 2}
		\lVert \overline{c}(\theta_L^-)\rVert_{\ell^2}
		\leq C_0\sqrt{p_L}\,x
	\end{equation}
	Parseval's identity at the two jumps, together with \eqref{eqAiry: |c(s)| decreases} on the three smooth pieces, gives
	\begin{equation}\label{eqAiry: | bar c| 2}
		\sup_{\theta\in[\theta_L,b]}
		\lVert \overline{c}(\theta^\pm)\rVert_{\ell^2}
		\leq C_0\sqrt{p_L}\,x.
	\end{equation}
	
	Under the restriction $\varepsilon_\star m_0\leq 1$, assume that $0\leq x\leq 2q(a)^{-1/3}$. Then
	\[
	x\leq 2\exp\{\varepsilon m_0/3\}p_L^{-1/3}
	\leq 3p_L^{-1/3}.
	\]
	Thus \eqref{eqAiry: eigenvector, lower}, with $L=3$, gives
	\begin{equation}\label{eqAiry: bar c 1 (theta L)}
		\overline{c}_1(\theta_L^-)
		=\psi_1^{p_L}(x)
		\geq C_3\sqrt{p_L}\,x,
		\qquad 0\leq x\leq 2q(a)^{-1/3}.
	\end{equation}
	
	At $\theta_L$, the potential jumps from $p_L$ to $q(\theta_L)$. Moreover, $\lvert\log(q(\theta_L)/p_L)\rvert\leq 2\varepsilon m_0$. Hence, Lemma~\ref{lemAiry: psi_1^p -psi_1^q} and \eqref{eqAiry: | bar c(theta_L)| 2} give
	\begin{equation}\label{eqAiry: bar c_1 diff 1}
		\lvert \overline{c}_1(\theta_L^+)-\overline{c}_1(\theta_L^-)\rvert
		\leq C\varepsilon m_0\sqrt{p_L}\,x.
	\end{equation}
	On $[\theta_L,\theta_R]$, the potential is $q$. Lemma~\ref{lemAiry: c_1 small difference} applies with the same function $B_q$, restricted to the middle interval. Using \eqref{eqAiry: | bar c| 2},
	\begin{equation}\label{eqAiry: bar c_1 diff 2}
		\lvert \overline{c}_1(\theta_R^-)-\overline{c}_1(\theta_L^+)\rvert
		\leq
		C\sqrt{p_L}\,x
		\bigl(\mathcal{E}(q)^{1/2}+\mathcal{E}(q)\bigr).
	\end{equation}
	At $\theta_R$, the same argument as in \eqref{eqAiry: bar c_1 diff 1} gives
	\begin{equation}\label{eqAiry: bar c-1 diff 3}
		\lvert \overline{c}_1(\theta_R^+)-\overline{c}_1(\theta_R^-)\rvert
		\leq C\varepsilon m_0\sqrt{p_L}\,x.
	\end{equation}
	Combining \eqref{eqAiry: bar c 1 (theta L)}--\eqref{eqAiry: bar c-1 diff 3}, we obtain
	\[
	\overline{c}_1(\theta_R^+)
	\geq
	\bigl[
	C_3-C\bigl(\varepsilon m_0+\mathcal{E}(q)^{1/2}+\mathcal{E}(q)\bigr)
	\bigr]\sqrt{p_L}\,x.
	\]
	After choosing $\varepsilon_\star$ and $\mathcal{E}_\star$ sufficiently small, depending only on the numerical constants above, we therefore have $\overline{c}_1(\theta_R^+)\geq (C_3/2)\sqrt{p_L}\,x$.
	
	On $(\theta_R,b]$, the potential is constant and equal to $p_R$. Hence
	\begin{equation}\label{eqAiry: bar W(b,y)}
		\begin{aligned}
			\overline{W}(b,y)
			&=\overline{c}_1(\theta_R^+)\psi_1^{p_R}(y)\\
			&\quad+
			\sum_{n=2}^\infty
			\exp\{-(\alpha_n-\alpha_1)\overline{J}_R\}
			\overline{c}_n(\theta_R^+)\psi_n^{p_R}(y).
		\end{aligned}
	\end{equation}
	Since $\varepsilon_\star m_0\leq 1$, the same use of \eqref{eqAiry: eigenvector, lower} with $L=3$ gives $\psi_1^{p_R}(y)\geq c_0\sqrt{p_R}\,y$ for $0\leq y\leq 2q(b)^{-1/3}$.
	Therefore, the first term in \eqref{eqAiry: bar W(b,y)} is bounded below by $(C_3^2/2)\sqrt{p_Lp_R}\,xy$.
	On the other hand, by Cauchy--Schwarz, \eqref{eqAiry: eigenvector, upper}, \eqref{eqAiry: | bar c| 2}, and \eqref{eqAiry: int overline q ^2/3 R}, the absolute value of the second term in \eqref{eqAiry: bar W(b,y)} is at most
	\[
	C_{\mathrm A}C_0\rho(m_0)\sqrt{p_Lp_R}\,xy
	\leq \frac{C_3^2}{4}\sqrt{p_Lp_R}\,xy.
	\]
	By the choice of $m_0$, the second term is at most half the first. Hence $\overline{W}(b,y)\geq c\sqrt{p_Lp_R}\,xy$, and therefore $\overline{W}(b,y)\geq c\sqrt{q(a)q(b)}\,xy$, for $0\leq x\leq 2q(a)^{-1/3}$ and $0\leq y\leq 2q(b)^{-1/3}$, where $c>0$ is numerical.
	
	Finally, $\overline{q}\geq q$, so the comparison principle gives $K_q\geq K_{\overline{q}}$. Moreover, by \eqref{eqAiry: int overline q ^2/3 L} and \eqref{eqAiry: int overline q ^2/3 R},
	\[
	0\leq
	\int_a^b\overline{q}(u)^{2/3}\,\td u-\Lambda_q(a,b)
	\leq
	2m_0\bigl(\exp\{\frac{4}{3}\varepsilon m_0\}-1\bigr),
	\]
	which is bounded by a numerical constant after the above choice of $m_0$ and the restriction $\varepsilon\leq\varepsilon_\star$. It follows that
	\[
	\begin{aligned}
		K_q(x,y;a,b)
		&\geq K_{\overline{q}}(x,y;a,b)\\
		&=
		\exp\Bigl\{-\alpha_1\int_a^b\overline{q}(u)^{2/3}\,\td u\Bigr\}
		\overline{W}(b,y)\\
		&\geq
		C^{-1}xy\sqrt{q(a)q(b)}
		\exp\{-\alpha_1\Lambda_q(a,b)\}.
	\end{aligned}
	\]
	which completes the proof of \eqref{eqAiry: abstract-lower}.
\end{proof}

\subsection{Airy estimate for \eqref{eqA: Airy PDE with sigma diffusion}} \label{Appendix Section: Airy PDE estimate}
In this subsection, we estimate the fundamental solution of the Airy-type PDE \eqref{eqA: Airy PDE with sigma diffusion}. The difference between the setting here and that of \cite{MZ2016_slowdownBBM} can be seen in Subsection \ref{subsec: main ideas}.

Fix $\alpha\in(0,1/2]$, $c_A>0$, $c_0>0$, and $s_0>0$, and set
\[
S(t):=s_0+c_0t,
\qquad
\sigma(t):=\sqrt{c_0}\,S(t)^{-\alpha}.
\]
Let $-\alpha_1$ be the largest zero of the Airy function $\operatorname{Ai}$, so that $\alpha_1=2.33811\ldots$. Define
\begin{equation}\label{eqAiry: eta}
	\eta(t):=
	\alpha_1 2^{-1/3}c_A^{2/3}
	\lvert\sigma'(t)\rvert^{2/3}\sigma(t)^{1/3}.
\end{equation}
Let $r\in C^2([0,\infty))$. We assume that, for some constant $C_r<\infty$,
\begin{equation}\label{eqAiry: conditions on r}
	\lvert r(t)\rvert\leq C_rS(t)^{-\beta_1},
	\qquad
	\lvert r'(t)\rvert\leq C_rS(t)^{-\beta_2},
	\qquad
	\lvert r''(t)\rvert\leq C_rS(t)^{-\beta_3},
\end{equation}
where
\begin{equation}\label{eqAiry: beta 123 assumptions}
	\beta_1>\alpha+\frac{1}{2},
	\qquad
	\beta_2>\alpha+\frac{3}{2},
	\qquad
	\beta_3>\alpha+2.
\end{equation}
Define $k(t):=c_A\sigma(t)-(\eta(t)+r(t))/c_A$. We consider the Dirichlet problem on $(t,x)\in\mathbb{R}_+\times\mathbb{R}_+$,
\begin{equation}\label{eqAiry: PDE, u, sigma, r}
	\begin{aligned}
		\partial_tu(t,x)
		&=\frac{1}{2}\sigma(t)^2\partial_{xx}u(t,x)
		-\Bigl(\frac{k}{\sigma^2}\Bigr)'(t)xu(t,x)
		+\frac{\eta(t)+r(t)}{\sigma(t)}u(t,x),\\
		u(t,0)&=0.
	\end{aligned}
\end{equation}
Let $G_u(x,y;t_1,t_2)$ denote its positive fundamental solution. Finally, set
\begin{equation}\label{eqAiry: V(t)}
	V(t):=
	\frac{2}{\sigma(t)^2}\Bigl(\frac{c_A}{\sigma(t)}\Bigr)'
	=2c_A\alpha c_0^{-1/2}S(t)^{3\alpha-1},
\end{equation}
which will be used below.

\begin{proposition}\label{propoAiry: Airy estimate}
	Fix $\alpha\in(0,1/2]$ and consider the PDE \eqref{eqAiry: PDE, u, sigma, r} with $r$ satisfying \eqref{eqAiry: conditions on r}. There exist constants $s_\star,L_\star,C\in(0,\infty)$ such that, for every $s_0\geq s_\star$, the following hold.
	\begin{enumerate}
		\item If $0\leq t_1<t_2$ and $t_2-t_1\geq L_\star S(t_1)^{2/3}$, then, for all $x,y\geq 0$,
		\begin{equation}\label{eqAiry: upper bound}
			G_u(x,y;t_1,t_2)
			\leq Cxy\sqrt{V(t_1)V(t_2)}
			\exp\Biggl\{\int_{t_1}^{t_2}\frac{r(s)}{\sigma(s)}\,\td s\Biggr\}.
		\end{equation}
		\item If $T>0$ and $T\geq L_\star S(0)^{2/3}$, then, for all $0\leq x\leq V(0)^{-1/3}$ and $0\leq y\leq V(T)^{-1/3}$,
		\begin{equation}\label{eqAiry: lower bound}
			G_u(x,y;0,T)
			\geq C^{-1}xy\sqrt{V(0)V(T)}
			\exp\Biggl\{\int_0^T\frac{r(s)}{\sigma(s)}\,\td s\Biggr\}.
		\end{equation}
	\end{enumerate}
\end{proposition}

\begin{proof}
	We start with changing of variables to transform the PDE \eqref{eqAiry: PDE, u, sigma, r} into the Airy PDE \eqref{eqAiry: PDE, w, q} which is within the setting of Proposition \ref{propoAiry: any q potential}. Introduce
	\begin{equation}\label{eqAiry: theta(t)}
		\theta(t):=\int_0^t\frac{\sigma(s)^2}{2}\,\td s.
	\end{equation}
	Thus,
	\[
	\theta(t)=
	\begin{cases}
		\displaystyle
		\frac{S(t)^{1-2\alpha}-s_0^{1-2\alpha}}{2(1-2\alpha)},
		&0<\alpha<1/2,\\[1.1ex]
		\displaystyle
		\frac{1}{2}\log\frac{S(t)}{s_0},
		&\alpha=1/2.
	\end{cases}
	\]
	Set
	\[
	u(t,x)=
	\exp\Biggl\{\int_0^t\frac{\eta(s)+r(s)}{\sigma(s)}\,\td s\Biggr\}
	w(\theta(t),x).
	\]
	A direct computation shows that $w$ solves
	\begin{equation}\label{eqAiry: PDE, w, q}
		\partial_\theta w=\partial_{xx}w-q(\theta)xw,
		\qquad q(\theta(t))=:Q(t),
	\end{equation}
	where
	\[
	Q(t):=\frac{2}{\sigma(t)^2}
	\Bigl(\frac{k}{\sigma^2}\Bigr)'(t).
	\]
	The corresponding kernels satisfy
	\begin{equation}\label{eqAiry: G_u = sth K_q}
		G_u(x,y;t_1,t_2)
		=
		\exp\Biggl\{\int_{t_1}^{t_2}\frac{\eta(s)+r(s)}{\sigma(s)}\,\td s\Biggr\}
		K_q(x,y;\theta(t_1),\theta(t_2)).
	\end{equation}
	Since
	\[
	\frac{k}{\sigma^2}
	=\frac{c_A}{\sigma}-\frac{\eta+r}{c_A\sigma^2},
	\]
	we have
	\begin{equation}\label{eqAiry: Q decomp}
		Q(t)=V(t)+R(t),
		\qquad
		R(t):=-\frac{2}{c_A\sigma(t)^2}
		\Bigl(\frac{\eta(t)+r(t)}{\sigma(t)^2}\Bigr)',
	\end{equation}
	with $V$ defined in \eqref{eqAiry: V(t)}.
	
	We now define the variables concerned in the conditions of Proposition \ref{propoAiry: any q potential}. Set
	\begin{equation}
		s=s(\theta):=\int_{\theta(t_1)}^\theta q(u)^{2/3}\,\td u.
	\end{equation}
	By \eqref{eqAiry: theta(t)},
	\[
	\frac{\td s}{\td t}=Q(t)^{2/3}\frac{\sigma(t)^2}{2}.
	\]
	Recall that, in Proposition \ref{propoAiry: any q potential}, we define $B_q(s):=(\td/\td s)\log q(\theta(s))$. Equivalently,
	\[
	B_q(s(t))
	=\frac{(\log Q)'(t)}{Q(t)^{2/3}\sigma(t)^2/2}.
	\]
	Also, by the definition \eqref{eqAiry: Lambda q} of $\Lambda_q$,
	\begin{equation}\label{eqAiry: Lambda q}
		\Lambda_q(\theta(t_1),\theta(t_2))
		=\int_{\theta(t_1)}^{\theta(t_2)}q(\theta)^{2/3}\,\td\theta
		=\int_{t_1}^{t_2}Q(t)^{2/3}\frac{\sigma(t)^2}{2}\,\td t.
	\end{equation}
	We also define
	\[
	\Lambda_V(t_1,t_2):=
	\int_{t_1}^{t_2}V(t)^{2/3}\frac{\sigma(t)^2}{2}\,\td t.
	\]
	with $V$ defined in \eqref{eqAiry: V(t)}. Note that we have chosen $\eta$ in \eqref{eqAiry: eta} so that
	\begin{equation}\label{eqAiry: why eta}
		\alpha_1\Lambda_V(t_1,t_2)
		=\int_{t_1}^{t_2}\frac{\eta(t)}{\sigma(t)}\,\td t,
	\end{equation}
	which will finally be used to cancel the term $\eta(s)/\sigma(s)$ in \eqref{eqAiry: G_u = sth K_q}.
	
	In what follows, we will prove some conditions required in Proposition \ref{propoAiry: any q potential} under our assumptions \eqref{eqAiry: conditions on r} on $r$ and definition \eqref{eqAiry: eta} for $\eta$:
	\begin{equation}\label{eqAiry: Q symp V}
		\frac{1}{2}V(t)\leq Q(t)\leq\frac{3}{2}V(t),
		\qquad t\geq 0,
	\end{equation}
	\begin{equation}\label{eqAiry: check 1}
		\varepsilon(q):=\lVert B_q\rVert_{L^\infty}\leq Cs_0^{-1/3},
		\qquad
		\mathcal{E}(q):=\int_0^{\Lambda_q(t_1,t_2)}\lvert B_q(s)\rvert^2\,\td s
		\leq Cs_0^{-1/3},
	\end{equation}
	and
	\begin{equation}\label{eqAiry: check 2}
		\sup_{0\leq t_1<t_2<\infty}
		\Bigl\lvert
		\Lambda_q(\theta(t_1),\theta(t_2))-\Lambda_V(t_1,t_2)
		\Bigr\rvert
		\leq C.
	\end{equation}
	Also, we will check
	\begin{equation}\label{eqAiry: check 3}
		\Lambda_q(\theta(t_1),\theta(t_2))\geq CL_\star
	\end{equation}
	under the condition $t_2-t_1\geq L_\star S(t_1)^{2/3}$.
	
	Once \eqref{eqAiry: Q symp V}--\eqref{eqAiry: check 3} are proven, Proposition \ref{propoAiry: any q potential} gives
	\[
	\begin{aligned}
		K_q(x,y;\theta(t_1),\theta(t_2))
		&\leq Cxy\sqrt{q(\theta(t_1))q(\theta(t_2))}
		\exp\{-\alpha_1\Lambda_q(\theta(t_1),\theta(t_2))\}\\
		&=Cxy\sqrt{Q(t_1)Q(t_2)}
		\exp\{-\alpha_1\Lambda_q(\theta(t_1),\theta(t_2))\}\\
		&\leq Cxy\sqrt{V(t_1)V(t_2)}
		\exp\Biggl\{-\int_{t_1}^{t_2}\frac{\eta(s)}{\sigma(s)}\,\td s\Biggr\},
	\end{aligned}
	\]
	where the second inequality is due to \eqref{eqAiry: Q symp V} and \eqref{eqAiry: check 2}. Then the upper bound \eqref{eqAiry: upper bound} follows from \eqref{eqAiry: G_u = sth K_q}. Using the lower-bound estimate in Proposition \ref{propoAiry: any q potential} and a similar discussion proves \eqref{eqAiry: lower bound}.
	
	To prove \eqref{eqAiry: Q symp V}--\eqref{eqAiry: check 3}, we make some auxiliary estimates first. We start with the bounds for $R(t)$ appearing in the decomposition \eqref{eqAiry: Q decomp}. The definition of $\eta$ gives
	\[
	\eta(t)=c_\eta S(t)^{-\alpha-2/3},
	\qquad
	c_\eta:=
	\alpha_1 2^{-1/3}c_A^{2/3}\alpha^{2/3}c_0^{7/6}.
	\]
	Therefore,
	\begin{equation}\label{eqAiry: R explicit}
		R(t)
		=-\frac{2c_\eta}{c_Ac_0}
		\Bigl(\alpha-\frac{2}{3}\Bigr)S(t)^{3\alpha-5/3}
		-\frac{2}{c_Ac_0^2}S(t)^{4\alpha}r'(t)
		-\frac{4\alpha}{c_Ac_0}S(t)^{4\alpha-1}r(t).
	\end{equation}
	Since $V(t)\asymp S(t)^{3\alpha-1}$, we have
	\[
	\Bigl\lvert\frac{R(t)}{V(t)}\Bigr\rvert
	\leq CS(t)^{-2/3}
	+CS(t)^{\alpha+1-\beta_2}
	+CS(t)^{\alpha-\beta_1}
	\leq\frac{1}{2}
	\]
	when $s_\star>0$ is large and thus $S(t)=s_0+c_0t\geq s_\star$ is large, which is due to our choice \eqref{eqAiry: beta 123 assumptions}. Thus \eqref{eqAiry: Q symp V} follows from $Q(t)=V(t)+R(t)$.
	
	Next, differentiating \eqref{eqAiry: R explicit} and using the decay assumptions \eqref{eqAiry: conditions on r} on $r,r',r''$ gives
	\[
	\begin{aligned}
		\lvert R'(t)\rvert
		&\leq C\Bigl(
		S(t)^{3\alpha-8/3}
		+S(t)^{4\alpha-\beta_3}
		+S(t)^{4\alpha-1-\beta_2}
		+S(t)^{4\alpha-2-\beta_1}
		\Bigr)\\
		&\leq CS(t)^{3\alpha-2}.
	\end{aligned}
	\]
	Thus,
	\begin{equation}\label{eqAiry: |Q'|}
		\begin{aligned}
			\lvert Q'(t)\rvert
			&\leq\lvert V'(t)\rvert+\lvert R'(t)\rvert\\
			&\leq C\Bigl(
			S(t)^{3\alpha-2}
			+S(t)^{3\alpha-8/3}
			+S(t)^{4\alpha-8/3}
			+S(t)^{4\alpha-1-\beta_2}
			+S(t)^{4\alpha-2-\beta-1}
			\Bigr)\\
			&\leq CS(t)^{3\alpha-2},
		\end{aligned}
	\end{equation}
	by our assumptions \eqref{eqAiry: beta 123 assumptions}.
	
	With \eqref{eqAiry: Q symp V} and \eqref{eqAiry: |Q'|}, we have
	\[
	\lvert B_q(x)\rvert
	=\Biggl\lvert
	\frac{2(\log Q)'(t)}{Q(t)^{2/3}\sigma(t)^2}
	\Biggr\rvert
	\leq\frac{2\lvert Q'(t)\rvert}{Q(t)^{5/3}\sigma(t)^2}
	\leq CS(t)^{-1/3},
	\]
	and
	\[
	\mathcal{E}(q)
	=\frac{1}{2}\int_{t_1}^{t_2}
	B_q(s(t))^2Q(t)^{2/3}\sigma(t)^2\,\td t
	\leq CS(t_1)^{-1/3}
	\leq Cs_0^{-1/3},
	\]
	which prove \eqref{eqAiry: check 1}. Note that
	\[
	\begin{aligned}
		&\Bigl\lvert
		\Lambda_q(\theta(t_1),\theta(t_2))-\Lambda_V(t_1,t_2)
		\Bigr\rvert\\
		&\quad\leq C\int_{t_1}^{t_2}\sigma(t)^2V(t)^{-1/3}\lvert R(t)\rvert\,\td t\\
		&\quad\leq C\int_{t_1}^{t_2}
		\Bigl(
		S(t)^{-4/3}
		+S(t)^{\alpha+1/3-\beta_2}
		+S(t)^{\alpha-2/3-\beta_1}
		\Bigr)\,\td t
		\leq C,
	\end{aligned}
	\]
	where the last line is due to the definition \eqref{eqAiry: R explicit} of $R(t)$ and \eqref{eqAiry: beta 123 assumptions}, which proves \eqref{eqAiry: check 2}. To check \eqref{eqAiry: check 3}, it suffices to notice that
	\[
	\begin{aligned}
		\Lambda_q(\theta(t_1),\theta(t_2))
		&\geq c\int_{t_1}^{t_1+L_\star S(t_1)^{2/3}}S(t)^{-2/3}\,\td t\\
		&\geq cL_\star
		\bigl(1+c_0L_\star S(t_1)^{-1/3}\bigr)^{-2/3}.
	\end{aligned}
	\]
	After $L_\star$ is chosen, enlarge $s_\star$ so that $c_0L_\star s_\star^{-1/3}\leq 1$; then $\Lambda_q(\theta(t_1),\theta(t_2))\geq c'L_\star$.
\end{proof}

\begin{remark}\label{remark: G(x,y;t1,t2) killed BM estimate}
	In Proposition~\ref{propoAiry: Airy estimate}, the time interval is assumed to be sufficiently long so that the Airy spectral behavior dominates. We will also need a short-time upper bound on the fundamental solution $G_u(x,y;t_1,t_2)$ of \eqref{eqAiry: PDE, u, sigma, r}. This follows from comparison with Brownian motion killed at the origin.
	
	Indeed, after increasing $s_\star$ if necessary, \eqref{eqAiry: Q symp V} shows that the linear potential $-(k/\sigma^2)'(t)x$ in \eqref{eqAiry: PDE, u, sigma, r} is nonpositive. This holds on $\mathbb{R}_+\times\mathbb{R}_+$. Hence, by the comparison principle,
	\[
	G_u(x,y;t_1,t_2)\leq G_{\widetilde{u}}(x,y;t_1,t_2),
	\]
	where $G_{\widetilde{u}}$ denotes the fundamental solution of
	\[
	\partial_t\widetilde{u}
	=\frac{1}{2}\sigma(t)^2\partial_{xx}\widetilde{u}
	+\frac{\eta+r}{\sigma}(t)\widetilde{u}
	\]
	on $\mathbb{R}_+$, with Dirichlet boundary condition $\widetilde{u}(t,0)=0$. By the estimate of the Green kernel of killed Brownian motion, with $a(t_1,t_2):=\int_{t_1}^{t_2}\sigma(s)^2\,\td s$, there exists $C>0$ such that
	\[
	G_{\widetilde{u}}(x,y;t_1,t_2)
	\leq C\frac{xy}{a(t_1,t_2)^{3/2}}
	\exp\Biggl\{-\frac{(x-y)^2}{Ca(t_1,t_2)}\Biggr\}
	\exp\Biggl\{\int_{t_1}^{t_2}\frac{\eta+r}{\sigma}(s)\,\td s\Biggr\}.
	\]
	In particular,
	\[
	\begin{aligned}
		\left.\partial_y G_{\widetilde{u}}(x,y;t_1,t_2)\right|_{y=0}
		&\leq Cx\,a(t_1,t_2)^{-3/2}
		\exp\Biggl\{-\frac{x^2}{Ca(t_1,t_2)}\Biggr\}
		\exp\Biggl\{\int_{t_1}^{t_2}\frac{\eta+r}{\sigma}(s)\,\td s\Biggr\}\\
		&\leq Cx^{-2}
		\exp\Biggl\{\int_{t_1}^{t_2}\frac{\eta+r}{\sigma}(s)\,\td s\Biggr\},
	\end{aligned}
	\]
	where the last line is due to
	\[
	a^{-3/2}\exp\Bigl\{-\frac{x^2}{Ca}\Bigr\}\leq Cx^{-3}
	\]
	uniformly for all $a>0$ and $x>0$, which can be checked by taking derivatives with respect to $a$. The same bounds hold for $G_u$ by comparison; that is,
	\begin{equation}\label{eqA: killed BB ker estimate}
		G_u(x,y;t_1,t_2)
		\leq C\frac{xy}{a(t_1,t_2)^{3/2}}
		\exp\Biggl\{-\frac{(x-y)^2}{Ca(t_1,t_2)}\Biggr\}
		\exp\Biggl\{\int_{t_1}^{t_2}\frac{\eta+r}{\sigma}(s)\,\td s\Biggr\},
	\end{equation}
	and
	\begin{equation}\label{eqA: killed BB ker estimate, derivative}
		\left.\partial_y G_u(x,y;t_1,t_2)\right|_{y=0}
		\leq Cx^{-2}
		\exp\Biggl\{\int_{t_1}^{t_2}\frac{\eta+r}{\sigma}(s)\,\td s\Biggr\}.
	\end{equation}
\end{remark}

\section{Cooling BBM with fixed splitting times}
\label{Appendix: BBM with fixed time}

The proof in this section follows the same strategy as in Subsection~\ref{subsection: micro-BBM, upper tail}.

We first introduce the model. Fix a branching period $\Delta t>0$ and an integer offspring number $m\geq2$. Let $T=N\Delta t$ for some integer $N\geq1$. At each deterministic branching time $k\Delta t$, $k=1,\ldots,N$, every particle is replaced by exactly $m$ offspring. As usual, we write $\mathcal{N}_t$ for the set of particles alive at time $t\in[0,T]$.

Fix $\alpha\in(0,1/2]$. Assume the movement of each particle follows the law
\begin{equation}\label{eqaBBM: one particle movement}
	Y_t^\alpha
	:=
	\int_0^t \sigma_\alpha(s)\,\td B_s,
	\qquad t\in[0,T],
\end{equation}
where $B$ is a standard Brownian motion and
\begin{equation}\label{eqaBBM: sigma}
	\sigma_\alpha(t)
	:=
	\sqrt{c_0}\,(s_0+c_0t)^{-\alpha},
\end{equation}
for some constants $c_0>0$ and $s_0\geq1$.

Set $c_A:=\sqrt{2\log m/\Delta t}$. Define
\begin{equation}\label{eqaBBM: eta}
	\eta(t)
	:=
	\alpha_1 2^{-1/3}c_A^{2/3}
	\lvert\sigma_\alpha'(t)\rvert^{2/3}\lvert\sigma_\alpha(t)\rvert^{1/3}
	=
	c_\eta (s_0+c_0t)^{-\alpha-2/3},
\end{equation}
where $c_\eta:=\alpha_1 2^{-1/3}c_A^{2/3}\alpha^{2/3}c_0^{7/6}$.
Let
\begin{equation}\label{eqaBBM: r}
	r(t)
	:=
	-(s_0+c_0t)^{-1-\beta},
	\qquad
	\beta\in(0,\alpha).
\end{equation}

It is straightforward to check that $r$ satisfies the assumptions in
\eqref{eqAiry: conditions on r}. We then define
\begin{equation}\label{eqaBBM: Ar(t)}
	k_r(t)
	:=
	c_A\sigma_\alpha(t)-\frac{\eta(t)+r(t)}{c_A},
	\qquad
	A_r(t)
	:=
	\int_0^t k_r(s)\,\td s,
\end{equation}
and
\begin{equation}\label{eqaBBM: A(t)}
	A(t)
	:=
	\int_0^t
	\bigl(c_A\sigma_\alpha(s)-\frac{\eta(s)}{c_A}
	\bigr)\,\td s.
\end{equation}
Since $r$ is integrable on $[0,\infty)$, there exists a constant $C>0$ such
that
\begin{equation}\label{eqaBBM: A and Ar small difference}
	\lvert A(T)-A_r(T)\rvert \leq \int_0^T \frac{\lvert r(t)\rvert}{c_A} \td t \leq C,
	\qquad T\geq0.
\end{equation}
Moreover, for all sufficiently large $t$,
\begin{equation}\label{eqaBBM: the potential is positive}
	\frac{k_r(t)}{\sigma_\alpha(t)^2}
	\asymp
	\frac{c_A}{\sigma_\alpha(t)}
	\asymp
	(s_0+c_0t)^\alpha.
\end{equation}

For later use in Proposition~\ref{Appendix: Bridge like probability}, we also
record the following bound. Since
\[
\frac{\eta(s)+r(s)}{\sigma_\alpha(s)}
=
\frac{c_\eta}{\sqrt{c_0}}(s_0+c_0s)^{-2/3}
-
\frac{1}{\sqrt{c_0}}(s_0+c_0s)^{-(1+\beta-\alpha)},
\]
and since both exponents are strictly larger than $1/2$, we have
\begin{equation}\label{eqaBBM: check an integral}
	0\leq
	\int_0^T\frac{(\eta(s)+r(s))^2}{\sigma_\alpha(s)^2}\,\td s
	\leq C,
\end{equation}
for some $0<C<\infty$, uniformly for all sufficiently large $T$.

\begin{proposition}\label{proposition: aBBM exceeding probability}
	Fix $\alpha\in(0,1/2]$. Consider the BBM with deterministic branching period $\Delta t>0$, offspring number $m\geq2$, and spatial motion given by \eqref{eqaBBM: one particle movement}. Then there exist constants
	$c,C>0$ and $K_*,T_*>0$ such that, for all integers $N\geq1$ with
	$T=N\Delta t\geq T_*$ and all $K\geq K_*$,
	\begin{equation}\label{eqaBBM: exceeding probability to be proven}
		\mathbb{P}\Bigl(\exists\,t\in[0,T],\ \exists\,v\in\mathcal{N}_t\colon
		Y_t^{\alpha,v}\geq A(t)+K
		\Bigr)
		\leq
		Ce^{-cK}.
	\end{equation}
\end{proposition}

\begin{proof}
	By \eqref{eqaBBM: A and Ar small difference}, it suffices to prove \eqref{eqaBBM: exceeding probability to be proven} with $A(t)$ replaced by $A_r(t)$, after adjusting the constants.
	
	We first assume that $s_0$ is large enough for the Airy estimate in Proposition~\ref{propoAiry: Airy estimate} to apply. This restriction is only
	technical. For general $s_0\geq1$, we can take the same time decomposition strategy as in Subsection~\ref{subsection: micro-BBM, upper tail}.
	
	For any continuous process $X$, let $\tau_0(X):=\inf\{t\geq0\colon X_t=0\}$ denote its first hitting time of $0$. By a union bound over the branching
	intervals, we have
	\begin{equation}\label{eqaBBM: upper bound of exceeding at T}
		\begin{aligned}
			& \mathbb{P}\Bigl(\exists\,t\in[0,T],\ \exists\,v\in\mathcal{N}_t\colon
			Y_t^{\alpha,v}\geq A_r(t)+K
			\Bigr)  \\
			&\leq
			\sum_{j=0}^{N-1}
			m^{j+1}
			\mathbb{P}\Bigl(\tau_0\bigl(Y^\alpha_\cdot-A_r(\cdot)-K\bigr)
			\in[j\Delta t,(j+1)\Delta t]
			\Bigr).
		\end{aligned}
	\end{equation}
	Here $m^{j+1}$ is a crude upper bound on the number of particles that may
	contribute on the interval $[j\Delta t,(j+1)\Delta t]$.
	
	By the definition of $A_r$ and Proposition~\ref{Appendix: Girsanov, exceeding probability},
	\[
	\begin{aligned}
		&\mathbb{P}\Bigl(
		\tau_0\bigl(Y^\alpha_\cdot-A_r(\cdot)-K\bigr)
		\in[j\Delta t,(j+1)\Delta t]
		\Bigr)\\
		&\quad\leq
		\exp\Bigl\{-\frac12c_A^2j\Delta t
		-\frac{k_r(0)}{\sigma_\alpha(0)^2}K\Bigr\}\\
		&\qquad\times
		\mathbb{E}^{(K)}\Biggl[
		\begin{aligned}
			&\exp\Biggl\{
			\int_0^{\tau_0(Y)}
			\biggl(
			-\biggl(\frac{k_r}{\sigma_\alpha^2}\biggr)'(s)Y_s
			+\frac{\eta(s)+r(s)}{\sigma_\alpha(s)}
			\biggr)\,\td s
			\Biggr\}
			\mathds{1}_{\{\tau_0(Y)\in[j\Delta t,(j+1)\Delta t]\}}
		\end{aligned}
		\Biggr].
	\end{aligned}
	\]
	where $\mathbb{E}^{(K)}$ denotes expectation for the process $Y$ starting at $Y_0=K$.
	
	Since $c_A^2=2\log m/\Delta t$, we have $\exp\{-\tfrac12c_A^2j\Delta t\}=m^{-j}$.
	Therefore \eqref{eqaBBM: upper bound of exceeding at T} gives
\begin{equation}\label{eqaBBM: upper bound of exceeding at T, 2}
	\begin{aligned}
		&\mathbb{P}\Bigl(
		\exists\,t\in[0,T],\ \exists\,v\in\mathcal{N}_t\colon
		Y_t^{\alpha,v}\ge A_r(t)+K
		\Bigr)\\
		&\quad\le
		m\exp\Bigl\{-\frac{k_r(0)}{\sigma_\alpha(0)^2}K\Bigr\}
		\mathbb{E}^{(K)}\Biggl[
		\exp\Biggl\{
		\int_0^{\tau_0(Y)}
		\left[
		-\Bigl(\frac{k_r}{\sigma_\alpha^2}\Bigr)'(s)Y_s
		+\frac{\eta(s)+r(s)}{\sigma_\alpha(s)}
		\right]\td s
		\Biggr\}
		\mathds{1}_{\{\tau_0(Y)\le T\}}
		\Biggr].
	\end{aligned}
\end{equation}

	By the Feynman--Kac representation, see for instance the discussion after
	equation $(2.5)$ in \cite{MZ2016_slowdownBBM}, the expectation on the
	right-hand side can be written as
	\[
	\int_0^T
	\frac12\sigma_\alpha(t)^2
	\left.\partial_y G_r(K,y;0,t)\right|_{y=0}
	\,\td t,
	\]
	where $G_r$ is the fundamental solution of the Airy-type equation
	\begin{equation}\label{eqaBBM: Airy-like PDE}
		\partial_t u
		=
		\frac12\sigma_\alpha(t)^2\partial_{xx}u
		-
		\bigl(\frac{k_r}{\sigma_\alpha^2}\bigr)'(t)xu
		+
		\frac{\eta(t)+r(t)}{\sigma_\alpha(t)}u,
		\qquad (t,x)\in\mathbb{R}_+\times\mathbb{R}_+,
	\end{equation}
	with Dirichlet boundary condition $u(t,0)=0$.
	
	The same argument leading to
	\eqref{eqBBM: partial derivative of G, 1} and
	\eqref{eqBBM: partial derivative of G, 2}, \eqref{eqBBM: upper bound, int ...< C lambda}, gives
	\par
	\begin{equation}\label{eqaBBM: FK Airy integral upper bound}
		\int_0^T\frac12\sigma_\alpha(t)^2
		\left.\partial_yG_r(K,y;0,t)\right|_{y=0}\,\td t
		\leq CK,
	\end{equation}
	uniformly in $T$ and $K$ large.
	
	Since, under the present assumption that $s_0$ is sufficiently large, $\frac{k_r(0)}{\sigma_\alpha(0)^2}>0$,
	we obtain from \eqref{eqaBBM: upper bound of exceeding at T, 2} and
	\eqref{eqaBBM: FK Airy integral upper bound} that

	\begin{equation}\label{eqaBBM: exceeding prob}
		\begin{aligned}
			\mathbb{P}\Bigl(
			\exists\,t\in[0,T],\ \exists\,v\in\mathcal{N}_t\colon
			Y_t^{\alpha,v}\ge A_r(t)+K
			\Bigr)
			&\le Cm(1+K)
			\exp\Bigl\{-\frac{k_r(0)}{\sigma_\alpha(0)^2}K\Bigr\}\\
			&\le Ce^{-cK}.
		\end{aligned}
	\end{equation}
	
	This proves the proposition.
\end{proof}

The next estimate gives a bridge-type upper bound for the fixed-splitting BBM.

\begin{proposition}\label{proposition: BBM fixed splitting time, tail estimate}
	Fix $\alpha\in(0,1/2]$ and consider the same BBM as in Proposition~\ref{proposition: aBBM exceeding probability}. Recall that $\mathcal{N}_T$ is the index set of all particles at the terminal time $T$. There exist constants $C,c,c_1,c_2>0$ and thresholds $K_*,T_*>0$ such that, for all integers $N\geq1$ with $T=N\Delta t\geq T_*$, all $K_0\geq K_*$, all
	$n\in\mathbb{Z}_+$, and every deterministic subset $\Gamma\subset\mathcal{N}_T$,
	\begin{equation}\label{eqaBBM: BBM fixed splitting time, tail estimate}
		\begin{aligned}
			&\mathbb{P}\Bigl(\exists\,v\in\Gamma\colon
			Y_T^{\alpha,v}\geq A(T)+K_0-n\Bigr)\\
			&\quad\leq Ce^{-c(K_0+n)}
			+C \lvert\Gamma\rvert 
			\exp\bigl\{-\tfrac12c_A^2T-c_1K_0+c_2T^\alpha(n+1)\bigr\},
		\end{aligned}
	\end{equation}
	where $c_A:=\sqrt{2\log m/\Delta t}$.
\end{proposition}

\begin{proof}
	As in the proof of Proposition~\ref{proposition: aBBM exceeding probability}, we may assume that $s_0$ is large enough for Proposition~\ref{propoAiry: Airy estimate} to apply.
	
	Replacing $K$ by $K_0+n$ in \eqref{eqaBBM: exceeding prob}, we obtain
	\[
	\mathbb{P}\Bigl(\exists\,t\in[0,T],\ \exists\,v\in\mathcal{N}_t\colon
	Y_t^{\alpha,v}\geq A_r(t)+K_0+n
	\Bigr)
	\leq
	Ce^{-c(K_0+n)}.
	\]
	Therefore,
	\begin{equation}\label{eqaBBM: probability of existing some v in Gamma good}
		\begin{aligned}
			& \mathbb{P}\Bigl(\exists\,v\in\Gamma\colon
			Y_T^{\alpha,v}\geq A(T)+K_0-n
			\Bigr) \\
			&\leq
			Ce^{-c(K_0+n)}
			+
			\lvert\Gamma\rvert\,
			\mathbb{P}\Bigl(
				Y_T^\alpha\geq A(T)+K_0-n,\ 
				Y_t^\alpha\leq A_r(t)+K_0+n
				\quad \forall t\in[0,T]
			\Bigr)
		\end{aligned}
	\end{equation}
	where $Y^\alpha_\cdot$ represents the movement of one particle and has the same law as $\int_0^\cdot \sigma_\alpha(s) \td B_s$.
	
	By \eqref{eqaBBM: A and Ar small difference}, there exists a constant
	$C_0>0$ such that $A(T)\geq A_r(T)-C_0$, for all $T\geq0$.
	Hence the last probability in
	\eqref{eqaBBM: probability of existing some v in Gamma good} is bounded
	above by
	\[
	\mathbb{P}\Bigl(
		Y_T^\alpha\geq A_r(T)+K_0-n-C_0,\\[-2pt]
		Y_t^\alpha\leq A_r(t)+K_0+n
		\quad \forall t\in[0,T]
	\Bigr).
	\]
	
	By Proposition~\ref{Appendix: Bridge like probability} and
	\eqref{eqaBBM: check an integral}, this probability is bounded by
	\[
	\begin{aligned}
		&\exp\biggl\{-\frac12c_A^2T
		-\frac{k_r(0)}{\sigma_\alpha(0)^2}(K_0+n)\biggr\}
		\int_0^{2n+C_0}
		\exp\biggl\{\frac{k_r(T)}{\sigma_\alpha(T)^2}z\biggr\}
		G_r(K_0+n,z;0,T)\,\td z,
	\end{aligned}
	\]
	where $G_r$ is the fundamental solution of \eqref{eqaBBM: Airy-like PDE}.
	
	By Proposition~\ref{propoAiry: Airy estimate},
	\[
	G_r(K_0+n,z;0,T)
	\leq C(K_0+n)z
	\sqrt{s_0^{3\alpha-1}(s_0+c_0T)^{3\alpha-1}}
	\exp\biggl\{\int_0^T\frac{r(s)}{\sigma_\alpha(s)}\,\td s\biggr\}.
	\]
	Since
	\[
	\int_0^T\frac{r(s)}{\sigma_\alpha(s)}\,\td s
	=-\frac{1}{\sqrt{c_0}}\int_0^T
	(s_0+c_0s)^{\alpha-1-\beta}\,\td s
	\leq C-cT^{\alpha-\beta},
	\]
	for some constants $C,c>0$, where $\beta<\alpha$, we have
	\[
	G_r(K_0+n,z;0,T)
	\leq
	C(K_0+n)z
	(s_0+c_0T)^{(3\alpha-1)/2}
	e^{-cT^{\alpha-\beta}}.
	\]
	Consequently,
	\begin{equation}\label{eqaBBM: Brownian bridge like estimate}
		\begin{aligned}
			&\mathbb{P}\Bigl(
			Y_T^\alpha\ge A(T)+K_0-n,\ 
			Y_t^\alpha\le A_r(t)+K_0+n\ \text{for all }t\in[0,T]
			\Bigr)\\
			&\quad\le
			C(K_0+n)(s_0+c_0T)^{(3\alpha-1)/2}
			\exp\Bigl\{-\frac{c_A^2}{2}T
			-\frac{k_r(0)}{\sigma_\alpha(0)^2}(K_0+n)
			-cT^{\alpha-\beta}\Bigr\}\\
			&\qquad\times
			\int_0^{2n+C_0}
			z\exp\Bigl\{\frac{k_r(T)}{\sigma_\alpha(T)^2}z\Bigr\}\,\td z.
		\end{aligned}
	\end{equation}
	
	Using \eqref{eqaBBM: the potential is positive}, we have $k_r(T)/\sigma_\alpha(T)^2\leq CT^\alpha$ for all sufficiently large $T$. Therefore,
	\[
	\int_0^{2n+C_0}
	z\,\exp\biggl\{\frac{k_r(T)}{\sigma_\alpha(T)^2}z\biggr\}
	\,\td z
	\leq
	C(n+1)^2\exp\{CT^\alpha(n+1)\}.
	\]
	Since $T\geq T_*$ and $K_0\geq K_*$ may be chosen large, after adjusting constants the last probability is at most
	$C\exp\bigl\{-\tfrac12c_A^2T-c_1K_0+c_2T^\alpha(n+1)\bigr\}$. Plugging this estimate into
	\eqref{eqaBBM: probability of existing some v in Gamma good} completes the proof.
\end{proof}

\section*{Acknowledgments}
The author thanks Nikos Zygouras for suggesting the problem that initiated this work, and for his guidance and valuable suggestions. The author also thanks Mo Dick Wong for pointing out the useful reference \cite{Acosta2014_tightness} and for an early discussion. Finally, the author thanks Clément Cosco and Shuta Nakajima, and Ofer Zeitouni for useful and inspiring discussions. The author is supported by the China Scholarship Council--University of Warwick Award.

\bibliographystyle{alpha}
\bibliography{max_loglog}

\end{document}